\documentclass[onefignum,onetabnum]{siamonline220329}

\usepackage{amsfonts,amssymb}
\usepackage{graphicx}
\usepackage{subfigure}

\usepackage{soul}

\usepackage{mathrsfs,mathtools,bbm,bm,stmaryrd}
\usepackage{nicefrac}

\usepackage{enumitem}
\setlist[enumerate]{leftmargin=.5in}
\setlist[itemize]{leftmargin=.5in}

\headers{Optimized filter functions for filtered back projection reconstructions}{M.~Beckmann, and J.~Nickel}

\title{Optimized filter functions for filtered back projection reconstructions%
\thanks{The authors are listed in alphabetical order. Corresponding author: Judith Nickel.\funding{This work was supported by the Deutsche Forschungsgemeinschaft (DFG) - project numbers 530863002 and 281474342/GRK2224/2 - as well as the Bundesministerium für Bildung und Forschung (BMBF) - funding code 03TR07W11A. The responsibility for the content of this publication lies with the authors.}}}

\author{Matthias Beckmann\thanks{Center for Industrial Mathematics, University of Bremen, Germany \& Department of Electrical and Electronic Engineering, Imperial College London, UK 
  (\email{research@mbeckmann.de}).}
\and Judith Nickel\thanks{Center for Industrial Mathematics, University of Bremen, Germany 
  (\email{judith.nickel@uni-bremen.de}).}}

\ifpdf
\hypersetup{
  pdftitle={Optimal Filter for FBP},
  pdfauthor={M. Beckmann and J. Nickel}
}
\fi

\def\N			{\mathbb N}
\def\Z			{\mathbb Z}

\def\R			{\mathbb R}

\def\Sphere		{\mathbb{S}^1}

\def\Radon		{\mathcal R}
\def\Back		{\mathcal B}
\def\Fourier	{\mathcal F}
\def\Int		{\mathcal I}

\def\ind		{\mathbbm 1}

\def\L			{\mathrm L}
\def\H			{\mathrm H}

\def\Schwartz	{\mathcal S}

\def\Test		{\mathscr D}

\def\d			{\mathrm d}
\def\e			{\mathrm e}
\def\i			{\mathrm i}

\def\FBP		{\mathrm{FBP}}
\def\loc		{\mathrm{loc}}

\def\F          {\mathcal{F}}
\def\P          {\mathrm{P}}

\DeclareMathOperator{\sinc}{sinc}
\DeclareMathOperator{\supp}{supp}
\DeclareMathOperator{\diam}{diam}
\DeclareMathOperator{\Span}{span}

\DeclareMathOperator*{\argmin}{arg\,min}

\DeclareMathOperator{\Ex}{\mathbb{E}}

\newcommand{\norm}[1]{\left\lVert#1\right\rVert}

\newsiamremark{remark}{Remark}

\begin{document}

\maketitle

\begin{abstract}
The method of filtered back projection (FBP) is a widely used reconstruction technique in X-ray computerized tomography (CT), which is particularly important in clinical diagnostics. To reduce scanning times and radiation doses in medical CT settings, enhancing the reconstruction quality of the FBP method is essential. To this end, this paper focuses on analytically optimizing the applied filter function. We derive a formula for the filter that minimizes the expected squared $\L^2$-norm of the difference between the FBP reconstruction, given infinite noisy measurement data, and the true target function. Additionally, we adapt our derived filter to the case of finitely many measurements. The resulting filter functions have a closed-form representation, do not require a training dataset, and, thus, provide an easy-to-implement, out-of-the-box solution. Our theoretical findings are supported by numerical experiments based on both simulated and real CT data.
\end{abstract}

\begin{keywords}
X-ray tomography, image reconstruction, filtered back projection, filter design, Gaussian noise.
\end{keywords}

\begin{AMS}
44A12, 65R32, 94A12.
\end{AMS}


\section{Introduction}

The method of \textit{filtered back projection} (FBP) is a commonly used reconstruction algorithm in X-ray computerized tomography (CT) \cite{Pan2009}, which aims at determining the internal structure of objects under investigation by measuring the attenuation of X-rays. Until today, CT is widely employed in medical diagnostics and non-destructive testing. The measured data can be modelled as line integral values of the \textit{attenuation function} $f\equiv f(x,y)$ of the object, formally represented via its \textit{Radon transform} $\Radon f\equiv \Radon f(s,\varphi)$, defined by
\begin{equation*}
\Radon f (s,\varphi) = \int_{\{x\cos(\varphi) + y\sin(\varphi) = s\}} f(x,y) \: \d(x,y)
\quad \text{for } (s,\varphi) \in \R \times [0,\pi).
\end{equation*}
With this, we can formulate the \textit{CT reconstruction problem} as determining the attenuation function $f$ from its Radon data
\begin{equation*}
\{\Radon f(s,\varphi) \mid s \in \R, ~ \varphi \in [0,\pi)\}.
\end{equation*}
Hence, one needs to invert the Radon transform $\Radon$. Based on the work of Johann Radon~\cite{Radon1917}, its analytical inversion was proven to be performed by the FBP formula~\cite{Feeman2015,Natterer2001a}, given by
\begin{equation*}
f(x,y) = \frac{1}{2} \, \Back(\Fourier^{-1}[\vert\sigma\vert\Fourier (\Radon f)(\sigma,\varphi)])(x,y)
\quad \text{for all } (x,y) \in \R^2,
\end{equation*}
where $\Fourier h$ denotes the univariate Fourier transform acting on the radial variable $s$ of a function $h\equiv h(s,\varphi)$, i.e.,
\begin{equation*}
\Fourier h(\sigma,\varphi) = \int_\R h(s,\varphi)\, \e^{-i s \sigma} \: \d s
\quad \text{for } (\sigma,\varphi) \in \R \times [0,\pi),
\end{equation*}
and $\Back h$ is the back projection of $h \equiv h(s,\varphi)$, defined by
\begin{equation*}
\Back h(x,y) = \frac{1}{\pi} \int_0^\pi h(x\cos(\varphi)+y\sin(\varphi),\varphi) \: \d \varphi
\quad \text{for } (x,y) \in \R^2.
\end{equation*}
The FBP formula is highly sensitive to noise as the involved filter $|\sigma|$ particularly amplifies high-frequency components.
Since, in practice, only noisy Radon data $g^\varepsilon$ can be measured, it is necessary to stabilize the FBP formula.
To this end, the filter $|\sigma|$ is typically replaced by a compactly supported low-pass filter $A_L\equiv A_L(\sigma)$ with bandwidth $L>0$ \cite{Natterer2001}, leading to the \textit{approximate FBP reconstruction formula} 
\begin{equation*}
f_L^\varepsilon(x,y) = \frac{1}{2} \, \Back(\Fourier^{-1}[A_L(\sigma)\Fourier g^\varepsilon(\sigma,\varphi)])(x,y)
\quad \text{for } (x,y) \in \R^2.
\end{equation*}
This approximate reconstruction formula is computational efficient, especially when compared to iterative reconstruction approaches such as the algebraic reconstruction technique (ART) and simultaneous iterative reconstruction technique (SIRT).

However, the reconstruction quality deteriorates with noisy measurement data, in particular when using standard filter functions (cf.~\cite{Pelt2014}). Additionally, in medical CT, reducing scanning times and radiation doses is crucial due to the harmful effects of X-radiation, which degrades the reconstruction quality even further. Consequently, comprehensive research has focused on improving the reconstruction quality of the FBP method.
One approach is to optimize the introduced filter function $A_L$ of the approximate FBP formula. This topic has recently regained attention in the research community, with several studies focusing on enhancing reconstruction quality by adapting the filter function. These approaches can be divided into two categories: analytical methods that define desirable properties of the filter and data-driven methods that involve learning a filter function.  However, most methods in the literature do not provide a closed-form solution for the filter. Instead, they require solving an optimization problem or rely on datasets that are challenging to obtain, particularly in the context of medical CT. Moreover, noise in the measured data is often only implicitly addressed in the design of the filter function. For further details, we refer to Section~\ref{Sec:Related_lit}.

As opposed to this, in this work, we derive an analytical formula for an optimized filter function enhancing the reconstruction quality of the approximate FBP formula. To achieve this, we model the noise in the measured data as a generalized stochastic process that approximates Gaussian white noise, which is grounded on observations in the literature. Based on this noise model, we construct our filter function by minimizing the expectation of the squared $\L^2$-norm of the difference between the approximate FBP reconstruction $f_L^\varepsilon$ and the true attenuation function $f$. We derive a formula for the resulting filter function in the context of complete measurement data and provide an adaptation for the case of a finite number of measurements.
Our proposed filters are as straightforward to implement as standard filters, quite fast to compute and do not require a training dataset, making them readily applicable.

The manuscript is organized as follows. Section~\ref{Sec:Related_lit} is devoted to a more detailed review of related result in the literature. In Section~\ref{Sec:Noise}, we develop a suitable model for the measurement noise based on generalized stochastic Gaussian white noise processes. We then derive our optimized filter function in the continuous setting in Section~\ref{sec:continuous filter} and adapt it to finite measurement data in Section~\ref{sec: discrete filter}. We validate the performance of our filter function through numerical experiments in Section~\ref{sec:Numerik} by comparing it with filter functions from the literature. 
Finally, Section~\ref{sec:outlook} concludes with a discussion of our findings and future research directions.

\section{Related results}\label{Sec:Related_lit}

The FBP method remains one of the most commonly used reconstruction algorithms in CT and extensive research has focused on improving its reconstruction quality. 
In the following, we describe various approaches that have been developed in the literature to address this task.
One idea involves developing pre- and post-processing methods, either for denoising the measured data or enhancing the computed reconstructions. Besides classical methods like Wiener filtering and wavelet decomposition, recent works have explored the use of deep learning techniques. For more details on pre- and post-processing methods in the context of CT, we refer the reader to the overview articles \cite{Balogh2022, Diwakar2018, Omar2018, Sadia2024, Ravishankar2017}.

In addition, the FBP method allows for a direct way to improve the reconstruction quality by changing the involved filter function $A_L$, which has been the focus of several studies. For noiseless measurement data, \cite{Beckmann2015} noted that the classical Ram-Lak filter is the best choice w.r.t.\ the $\L^2$-norm of the reconstruction error given complete measurement data $g$. In~\cite{Shi2013}, the authors adapt the Ram-Lak filter for noiseless finite measurements $g_D$ by expressing the filter function through its Fourier series expansion and determining the coefficients by solving an optimization problem derived from a reformulation of the FBP method.

In real-world applications, however, noise in the measured data has to be considered when designing filter functions. In \cite{Pelt2014}, a data-dependent filter for the FBP method is introduced, which minimizes the squared difference between the Radon transform of the reconstruction and the noisy data $g_D^\varepsilon$ so that a minimization problem must be solved for each measurement independently. The authors call the resulting method the minimum residual FBP method, abbreviated as MR-FBP. They also add a regularization term to the minimization problem, resulting in the so-called MR-FBP$_\mathrm{GM}$ method. The regularization term incorporates a constraint on the horizontal and vertical discrete gradients of the reconstruction and can be controlled by an additional hyperparameter $\lambda>0$. 
 
Recently, the authors in \cite{Kabri2024} aim at choosing the filter function by minimizing the true risk w.r.t.\ the squared $\L^2$-error between the FBP reconstruction and the ground truth, i.e., they consider the minimization problem
\begin{equation*}
\min_{A_L} \Ex \left(\Bigl\lVert f - \frac{1}{2} \, \Back(\Fourier^{-1} A_L \ast g^\varepsilon)\Bigr\rVert_{\L^2(\R^2)}^2\right),
\end{equation*}
where the expectation is w.r.t.\ the ground truth $f$ and the noisy measurement data $g^\varepsilon$.
Since calculating the true risk is infeasible in practice due to an unknown data distribution, they propose to minimize the empirical risk instead by replacing the expectation by the mean over a finite dataset consisting of noiseless measurements and additive noise samples.
We refer to this as the empirical risk minimization FBP method, in short ERM-FBP.

Another approach for improving reconstruction quality is to mimic other reconstruction methods, such as iterative reconstruction techniques. This approach is demonstrated in \cite{Pelt2015}, where the filter is designed in such a way that the FBP method approximates the algebraic simultaneous iterative reconstruction technique, resulting in the SIRT-FBP method. The SIRT-FBP filter depends solely on the discretization parameters and can be pre-calculated before reconstructing. This filter is applied to a large-scale real-world dataset in \cite{Pelt2016}.

The authors of~\cite{Horbelt2002} design filters that integrate both the filtering process with ideal ramp filter and an interpolation scheme, which is commonly employed when applying the FBP method, to reduce errors caused by interpolation. The authors discovered that their filters substantially enhance reconstructions only for interpolation schemes of up to second order.

Besides these analytical approaches, there have been recent developments in learning-based selection of filter functions. As this paper focuses on analytically optimizing the filter function, we will refrain from discussing it here. Instead, the interested reader is encouraged to consult \cite{Lagerwerf2021, Pelt2013, Shtok2008, Syben2018, Xu2021} for various methods involving neural networks.

\section{White noise and technical lemmata}\label{Sec:Noise}

In applications, the Radon data $\Radon f$ is always corrupted by noise caused, for example, by scattered radiation or the conversion of the measured X-ray photons into a digital signal by the detectors, cf.~\cite{Buzug2010, Epstein2008, Herman2009, Kak2001, Whiting2006}. Due to the various sources of measurement errors and unknown system calibrations in real-world applications, it is in general not feasible to analytically model the probability distribution of the noise, see~\cite{Fu2017, Li2004}. Instead, one needs to approximate the distribution based on measurement data, which is dealt with in several papers, e.g.,~\cite{Li2004, Liao2012, Lu2001, Wang2008}. The findings in the literature indicate that the measurement noise can be modelled by independent identically distributed (iid) additive Gaussian random variables, even in low intensity regimes. This process is commonly referred to as Gaussian white noise. 

Its formal definition is based on the classical Wiener process $W:\Omega^\prime \times \Omega\to\R$, also called Brownian motion process, with $\Omega^\prime\subseteq \R^n$ and a probability space $(\Omega,\F,\P)$, which describes the path of a particle without inertia caused by Brownian motion. Gaussian white noise, typically denoted by $H$, is then defined as the distributional derivative of the Wiener process and characterizes the velocity of a particle without inertia due to Brownian motion. To be more precise, $H:\Test(\Omega^\prime)\times \Omega\to\R$ is given by
\begin{equation*}
H(\phi,\omega) = -\langle W(\cdot, \omega),\phi^\prime \rangle = -\int_{\Omega^\prime} W(x,\omega)\phi^\prime(x) \: \d x
\end{equation*}
for all $\phi\in \Test(\Omega^\prime)$ and $\omega\in\Omega$, where $\Test(\Omega^\prime)$ denotes the space of real-valued test functions. 
Note that, for simplicity, we consider the space of real-valued test functions $\Test(\Omega^\prime)$ as well as the space of real-valued Schwartz functions $\Schwartz(\R^n)$ throughout this work. Our results, however, can be easily modified to complex-valued functions.
The properties of the Wiener process imply that $H(\phi,\cdot)$ is a Gaussian random variable for all $\phi\in \Test(\Omega^\prime)$ with
\begin{equation*}
\Ex (H(\phi,\cdot)) = 0 \quad \text{and} \quad \Ex (H(\phi,\cdot) H(\psi, \cdot)) = (\phi,\psi)_{\L^2(\Omega^\prime)}
\end{equation*}
for all $\phi,\psi\in \Test(\Omega^\prime)$. For a precise definition of a Wiener process and a detailed derivation of the above properties we refer the reader to \cite[Chapter III]{Gelfand1964} and \cite[Chapter 1]{Sobczyk1991}. 
Since a Gaussian process is uniquely defined by its mean and covariance, cf.~\cite[Chapter III.2.3]{Gelfand1964}, we obtain the following equivalent definition, which is more commonly used in the setting of inverse problems, see, e.g.,~\cite[p.3]{Kekkonen2014}.

\begin{definition}[Continuous white noise]
{\em Continuous white noise} is a function  
\begin{equation*}
H: \Test(\Omega^\prime) \times \Omega \to \R 
\end{equation*}
such that $H(\cdot,\omega)\in\Test^\prime(\Omega^\prime)$ for all $\omega\in\Omega$ and
$H(\phi, \cdot): \Omega \to \R$ is a Gaussian random variable on the probability space $(\Omega, \F, \P)$ for all $\phi \in \Test(\Omega^\prime)$ with 
\begin{equation*}
\Ex (H(\phi,\cdot)) = 0 \quad \text{and} \quad \Ex (H(\phi,\cdot) H(\psi, \cdot)) = (\phi,\psi)_{\L^2(\Omega^\prime)}
\end{equation*}
for all $\phi,\, \psi \in \Test(\Omega^\prime)$.
In addition, continuous white noise $H^\varepsilon: \Test(\Omega^\prime) \times \Omega \to \R$ with {\em noise level} $\varepsilon>0$ is defined by $H^\varepsilon = \varepsilon H$.
\end{definition}

We remark that continuous white noise is a so-called {\em generalized stochastic process}, as introduced in \cite[Chapter III]{Gelfand1964}. Moreover, one can prove that $H(\cdot,\omega)\in \Schwartz^\prime(\R^n)$ for almost all $\omega\in \Omega$, see~\cite{Dalang2017, Fageot2022} and, thus, $H$ can be equivalently defined as a mapping from $\Schwartz(\R^n) \times \Omega$ to~$\R$ such that $H(\cdot,\omega)$ has support in $\Omega^\prime \subseteq \R^n$ in the sense of distributions.

The definition of continuous white noise $H$ implies that the amount of power is the same for all frequencies, i.e., the power spectrum of the process is flat, cf.~\cite[Chapter 1.5.4]{Gardiner2009}. Such a process cannot exist in reality as it would require an infinite amount of energy. Instead, continuous white noise can be understood as an idealised model of a process with approximately independent values at every point. To circumvent this shortcoming of white noise, physicists often approximate white noise by assuming that points very close to each other are correlated, see~\cite[Chapter 4.1]{Gardiner2009} and~\cite[Chapter I.10]{Sobczyk1991}. This motivates the following definition of approximate white noise.

\begin{definition}[Approximate continuous white noise]
\label{Def:Approx white noise}
{\em Approximate continuous white noise} $H_a: \Schwartz(\R^n) \times \Omega \to \R $ with support in $\Omega^\prime\subseteq \R^n$ is defined as
\begin{equation*}
H_a(\phi,\omega) = \int_{\R^n} h_a(x,\omega) \, \phi(x) \, \ind_{\Omega^\prime}(x) \: \d x, 
\end{equation*}
where $\ind_{\Omega^\prime}: \R^n \to \{0,1\}$ is the characteristic function of $\Omega^\prime$ and $h_a : \R^n\times \Omega\to\R$ for $a > 0$ is a jointly-measurable function with the property that, for any $k \in \N$, $(h_a(x_1, \cdot),\dots, h_a(x_k, \cdot))$ with $x_1, \dots, x_k\in \R^n$ is a (multivariate) Gaussian random variable on the probability space $(\Omega, \F, \P)$ with 
\begin{equation*}
\Ex (h_a(x,\cdot)) = 0 \quad \text{and} \quad \Ex (h_a(x,\cdot) h_a(y, \cdot)) = \delta_a(x-y) \quad \text{for all } x,y \in\R^n,
\end{equation*}
where $(\delta_a)_{a>0}$ is a Dirac sequence with even $\delta_a: \R^n\to\R_{\geq 0}$ satisfying $\delta_a \to \delta$ for $a \to \infty$.
In addition, approximate continuous white noise $H_a^\varepsilon: \Schwartz(\R^n) \times \Omega \to \R $ of {\em noise level} $\varepsilon>0$ is defined by $H_a^\varepsilon= \varepsilon\, H_a$.
\end{definition}

In the following, we restrict our investigations to $H_a=H_a^1$. Note, however, that all results also hold for the general case of $H_a^\varepsilon$ with $\varepsilon>0$.
Let us first observe that for fixed $\omega \in \Omega$, the realization $H_a(\cdot,\omega) \in \Schwartz^\prime(\R^n)$ is a tempered distribution.
Moreover, for fixed $x \in \R^n$, the variance $\Ex (h_a(x,\cdot) h_a(x, \cdot)) = \delta_a(0)$ diverges for $a\to \infty$.

In the case that $\delta_a(t) = \frac{a}{2} \exp^{-a | t |}$ for $t\in\R$, $h_a$ corresponds to a stationary {\em Ornstein-Uhlenbeck} process with zero mean \cite[Chapter 3.8.4 \& 4.1]{Gardiner2009}. The Ornstein-Uhlenbeck process was first introduced as an improved model of the velocity of a particle caused by Brownian motion \cite{Uhlenbeck1930} and is commonly used as an approximation of Gaussian white noise in various use cases, see, e.g., \cite{Bibbona2008}, \cite[Chapter 4.1]{Gardiner2009}, \cite{Hottovy2019} and \cite[Chapter I.10]{Sobczyk1991}. In this sense, Definition~\ref{Def:Approx white noise} can be seen as a generalization of the Ornstein-Uhlenbeck process, allowing for a wider variety of processes. 

The approximate white noise $H_a$ is again a Gaussian random variable for fixed $\phi\in\Schwartz(\R^n)$, whose first and second moments depend on the function $h_a$. This is characterized in the following lemma. 

\begin{lemma}\label{Lem:Moments Approx Noise}
Let $H_a: \Schwartz(\R^n) \times \Omega \to \R $ be an approximate continuous white noise with support in $\Omega^\prime\subseteq \R^n$. Then, $H_a(\phi, \cdot): \Omega \to \R$ is a Gaussian random variable for all $\phi \in \Schwartz(\R^n)$ with
\begin{equation*}
\Ex (H_a(\phi,\cdot)) = 0 \quad \text{and} \quad \Ex (H_a(\phi,\cdot) H_a(\psi, \cdot)) = 
\int_{\Omega^\prime} \int_{\Omega^\prime} \delta_a(x-y) \, \phi(x) \, \psi(y) \: \d x \, \d y
\end{equation*}
for all $\psi, \phi\in \Schwartz(\R^n)$.
\end{lemma}

\begin{proof}
Let us start with the calculation of the expectation of $H_a(\phi,\cdot)$ for $\phi \in \Schwartz(\R^n)$. To this end, observe that $H_a(\phi,\cdot)$ is a random variable as $h_a$ is jointly-measurable and
\begin{equation*}
\int_{\Omega}\int_{\Omega^\prime} \lvert h_a(x,\omega)\rvert \, \lvert \phi(x) \rvert \: \d x \, \d\P (\omega)
= \sqrt{\frac{2\delta_a(0)}{\pi}} \int_{\Omega^\prime} \lvert\phi(x) \rvert \: \d x < \infty,
\end{equation*}
which follows from Fubini's Theorem and $\Ex(\lvert h_a(x,\cdot)\rvert) = \sqrt{\nicefrac{2\delta_a(0)}{\pi}}$ for all $x\in\R^n$ due to~\cite{Tsagris2014}.
Hence, applying Fubini's Theorem again results in
\begin{equation*}
\Ex (H_a(\phi,\cdot)) = \int_{\Omega}\int_{\Omega^\prime}  h_a(x,\omega)\, \phi(x) \: \d x \, \d\P (\omega) = \int_{\Omega^\prime} \Ex( h_a(x,\cdot))\, \phi(x) \: \d x = 0.
\end{equation*} 
For the calculation of the covariance of $H_a(\phi,\cdot)$ and $H_a(\psi,\cdot)$ for $\phi,\psi \in \Schwartz(\R^n)$ note that
\begin{align*}
\int_{\Omega} \int_{\Omega^\prime} \int_{\Omega^\prime} & \lvert h_a(x,\omega) h_a(y,\omega)\rvert \, \lvert \phi(x) \psi(y) \rvert \: \d x\,\d y \, \d\P(\omega) \\
&= \int_{\Omega^\prime} \int_{\Omega^\prime} \Ex\left(\lvert h_a(x,\cdot) h_a(y,\cdot)\rvert\right) \lvert \phi(x) \psi(y) \rvert \: \d x \,\d y \\
&\leq \int_{\Omega^\prime} \int_{\Omega^\prime} \Ex\left(\lvert h_a(x,\cdot) \rvert^2\right)^{\nicefrac{1}{2}} \Ex\left(\lvert h_a(y,\cdot) \rvert^2\right)^{\nicefrac{1}{2}} \, \lvert\phi(x) \psi(y) \rvert \: \d x \,\d y 
\end{align*}
by Fubini's Theorem and Cauchy Schwarz's inequality so that
\begin{equation*}
\int_{\Omega} \int_{\Omega^\prime} \int_{\Omega^\prime} \lvert h_a(x,\omega) h_a(y,\omega)\rvert \, \lvert \phi(x) \psi(y) \rvert \: \d x\,\d y \, \d\P(\omega) \leq \delta_a(0) \int_{\Omega^\prime} \int_{\Omega^\prime} \lvert\phi(x) \psi(y)\rvert \: \d x \,\d y < \infty.
\end{equation*} 
Consequently, we obtain
\begin{align*}
\Ex(H_a(\phi,\cdot) H_a(\psi, \cdot)) &= \int_{\Omega}\int_{\Omega^\prime} \int_{\Omega^\prime} h_a(x,\omega) h_a(y,\omega) \, \phi(x) \psi(y)  \: \d x \, \d y \, \d\P(\omega) \\
&= \int_{\Omega^\prime} \int_{\Omega^\prime} \Ex(h_a(x,\cdot) h_a(y,\cdot)) \, \phi(x) \psi(y) \: \d x \, \d y \\
&= \int_{\Omega^\prime} \int_{\Omega^\prime} \delta_a(x-y) \, \phi(x) \psi(y) \: \d x \, \d y.
\end{align*}

In order to prove that $H_a(\phi, \cdot): \Omega \to \R$ is a Gaussian random variable for all $\phi \in \Schwartz(\R^n)$, we consider the set
\begin{align*}
\mathcal{H} = \overline{\Span\{h_a(x,\cdot): \Omega \to \R \mid x \in \R^n\}},
\end{align*}
where the closure is w.r.t.\ the canonical norm of the Hilbert space of square integrable random variables on $(\Omega, \F, \P)$. As $h_a$ is a Gaussian process by assumption, $\mathcal{H}$ consists of Gaussian random variables. The aim now is to show that $H_a(\phi, \cdot)\in\mathcal{H}$, where we have $\mathcal{H} = (\mathcal{H}^\perp)^\perp$ since $\mathcal{H}$ is closed. We have already seen that $H_a(\phi, \cdot)$ is a square-integrable random variable. Thus, it remains to prove that $\Ex(X H_a(\phi, \cdot)) = 0$ for all $X\in \mathcal{H}^\perp$. To this end, note that 
\begin{align*}
\int_{\Omega} \Ex\left(\lvert X\rvert \, \lvert h_a(x,\cdot)\rvert\right) \lvert \phi(x) \rvert \: \d x &\leq \int_{\Omega} \Ex\left(\lvert X \rvert^2\right)^{\nicefrac{1}{2}} \Ex\left(\lvert h_a(x,\cdot) \rvert^2\right)^{\nicefrac{1}{2}} \lvert \phi(x) \rvert \: \d x \\
&= \Ex\left(\lvert X \rvert^2\right)^{\nicefrac{1}{2}} \delta_a(0)^{\nicefrac{1}{2}} \int_{\Omega} \lvert \phi(x) \rvert \: \d x < \infty
\end{align*}
due to Cauchy Schwarz's inequality. As a result, Fubini's Theorem yields
\begin{align*}
\Ex(X H_a(\phi, \cdot)) = \int_{\Omega^\prime} \Ex(X h_a(x, \cdot)) \phi(x) \: \d x = 0
\end{align*}
because $h_a(x,\cdot) \in \mathcal{H}$ for all $x \in \R^n$ and the proof is complete.
\end{proof}

These properties of the approximate continuous white noise $H_a$ imply its convergence to white noise $H$ as specified in the following lemma.

\begin{lemma}\label{Lem:Convergence of approximate white noise}
Let $H_a: \Schwartz(\R^n) \times \Omega \to \R $ be an approximate continuous white noise with support in $\Omega^\prime\subseteq \R^n$. Then, for all $\phi \in \Test(\Omega^\prime)$, the sequence of random variables $(H_a (\phi,\cdot))_{a>0}$ converges weakly\footnote{A sequence of random variables $(X_n)_n$ converges weakly to a random variable $X$ with cumulative distribution function $F$ if the corresponding cumulative distribution functions $(F_n)_n$ converge pointwise to $F$ at all points of continuity of $F$, cf.~\cite[Chapter 5]{Billingsley1995}.} to a Gaussian random variable $H_0(\phi,\cdot): \Omega \to \R$ with 
\begin{equation}\label{Eq: Ex and Cov of limit}
\Ex\left(H_0(\phi,\cdot)\right) = 0
\quad \text{and} \quad
\Ex\left(H_0(\phi,\cdot) H_0(\psi, \cdot)\right) = (\phi,\psi)_{\L^2(\Omega^\prime)}
\end{equation}
for all $\psi \in \Test(\Omega^\prime)$. 
\end{lemma}

\begin{proof}
First, observe that Lemma~\ref{Lem:Moments Approx Noise} yields
$\lim_{a\to \infty} \Ex (H_a(\phi,\cdot)) = 0$ and
\begin{equation*}
\lim_{a\to \infty} \Ex(H_a(\phi,\cdot) H_a(\psi, \cdot)) = (\phi,\psi)_{\L^2(\Omega^\prime)}
\end{equation*}
for all $\phi,\, \psi \in \Test(\Omega^\prime)$ as $(\delta_a)_{a \in \R_+}$ is a Dirac-sequence.
The method of moments \cite[Theorem 30.2]{Billingsley1995} in combination with the property that the moments of a Gaussian random variable are uniquely determined by its first and second moment \cite[Chapter 5.4]{Papoulis2002} now imply that $H_a(\phi, \cdot)$ converges weakly to a Gaussian random variable $H_0(\phi,\cdot): \Omega \to \R$ for all $\phi \in \Test(\Omega^\prime)$ with expectation zero and variance $(\phi,\phi)_{\L^2(\Omega^\prime)}$.  
For the covariance, observe that 
\begin{align*}
\Ex \left(H_0(\phi,\cdot) H_0(\psi, \cdot)\right) &=  \frac{1}{2}\left(\Ex \left((H_0(\phi,\cdot) + H_0(\psi, \cdot))^2\right) - \Ex \left(H_0(\phi,\cdot)^2\right) - \Ex\left(H_0(\psi,\cdot)^2\right) \right) \\
&= \frac{1}{2} \left( (\phi+\psi,\phi+\psi)_{\L^2(\Omega^\prime)} - (\phi,\phi)_{\L^2(\Omega^\prime)} - (\psi,\psi)_{\L^2(\Omega^\prime)} \right) \\
&= (\phi,\psi)_{\L^2(\Omega^\prime)},
\end{align*}
where we used that $H_0(\phi,\cdot) + H_0(\psi, \cdot)$ is the weak limit of $H_a(\phi,\cdot) + H_a(\psi, \cdot)$ and the same arguments as above prove that the variance of $H_0(\phi,\cdot) + H_0(\psi, \cdot)$ is given by $(\phi+\psi,\phi+\psi)_{\L^2(\Omega^\prime)}$. This completes the proof.
\end{proof}

The last lemma shows that our definition of approximate white noise is sensible in the sense that its weak limit has the same first and second moments as white noise $H$. 

An important characteristic of white noise is the regularity of its realizations. There are several works addressing the regularity of white noise $H$, e.g.,~\cite{Aziznejad2020, Veraar2010, Fageot2017, Fageot2017a}. For instance, \cite{Aziznejad2020} shows that the realizations of white noise $H: \Schwartz(\R^n) \times \Omega \to \R$ belong to the weighted Besov space $B_p^\tau(\R^n,\rho)$ with smoothness $\tau\in\R$, integrability $p\in \N$ and weight exponent $\rho\in\R$, cf.~\cite{Aziznejad2020}, \cite[Chapter 1.2.3]{Triebel2008}, almost surely if and only if $\tau < -\nicefrac{n}{2}$ and $\rho < -\nicefrac{n}{p}$. Similar arguments as in \cite{Aziznejad2020} show that the realizations of white noise $H: \Schwartz(\Omega^\prime) \times \Omega \to \R$ with compact support $\Omega^\prime\subset \R^n$ belong to $B_p^\tau(\R^n,\rho)$ almost surely if and only if $\tau < -\nicefrac{n}{2} + \nicefrac{n}{p}$ and $\rho \in\R$. Consequently, embedding results of weighted Besov spaces imply that the realizations of white noise with compact support belong to the Sobolev space of fractional order $\H^\alpha(\R^n)$ almost surely if $\alpha < 0$, see~\cite[Theorem 6.4.4]{Bergh1976}, \cite{Kuhn2005} and \cite{Meyries2012}. Here, the Sobolev space $\H^\alpha(\R^n)$ of fractional order $\alpha\in\R$ is defined as 
\begin{equation*}
\H^\alpha (\R^n) = \{f \in \Schwartz^\prime(\R^n) \mid \lVert f \rVert_{\H^\alpha(\R^n)} < \infty\}
\end{equation*}
with 
\begin{equation*}
\lVert f \rVert_{\H^\alpha(\R^n)}^2 = \int_{\R^n} (1+ \lVert \xi \rVert_{\R^n}^2)^\alpha \, \lvert \Fourier f(\xi) \rvert^2 \: \d\xi.
\end{equation*}
In contrast, approximate continuous white noise $H_a: \Schwartz(\Omega^\prime) \times \Omega \to \R $  with compact support $\Omega^\prime\subset \R^n$ exhibits a higher regularity. This is shown in the subsequent lemma. 

\begin{lemma}\label{Lem: Smoothness approx noise}
Assume that $H_a: \Schwartz(\R^n) \times \Omega \to \R $ is an approximate continuous white noise with compact support in $\Omega^\prime\subset \R^n$. Then, $H_a(\cdot, \omega) \in \L^2(\R^n)$ for almost all (f.a.a.) $\omega\in\Omega$.
In addition, if, for some $\alpha \in \R$,
\begin{equation*}
\int_{\R^n} (1+ \lVert \xi \rVert_{\R^n}^2)^\alpha \int_{\Omega^\prime} \int_{\Omega^\prime} \delta_a(x-y) \, \e^{-\i (x-y)\xi} \: \d x \, \d y \, \d\xi < \infty,
\end{equation*}
then $H_a(\cdot, \omega)\in \H^\alpha(\R^n)$ f.a.a.\ $\omega\in\Omega$.
\end{lemma}

\begin{proof}
For the expectation of the $\L^2$-norm of $H_a$ we obtain
\begin{equation*}
\Ex\left(\lVert H_a \rVert_{\L^2(\R^n)}^2\right) = \int_\Omega \int_{\Omega^\prime} \lvert h_a(x,\omega) \rvert^2 \: \d x \, \d\P(\omega) = \int_{\Omega^\prime} \delta_a(0) \: \d x < \infty
\end{equation*}
by Fubini's Theorem and the compactness of $\Omega^\prime$. As a result, $H_a(\cdot, \omega)\in \H^0(\R^n)$ f.a.a.\ $\omega\in\Omega$. 

We proceed with determining the expectation of the $\H^\alpha$-norm of $H_a$ with $\alpha\in\R$. To this end, observe that the expectation of $\lvert\Fourier h_a(\xi,\cdot) \rvert^2$ for fixed $\xi\in\R^n$ can be rewritten as
\begin{align*}
\Ex\left(\lvert \Fourier h_a(\xi,\cdot) \rvert^2\right) &= \int_\Omega \Fourier h_a(\xi,\omega) \, \overline{\Fourier h_a(\xi,\omega)} \: \d\P(\omega) \\
&= \int_{\Omega^\prime} \int_{\Omega^\prime} \int_\Omega h_a(x,\omega) h_a(y,\omega) \: \d\P(\omega) \, \e^{-\i(x-y)\xi} \: \d x \, \d y \\
&= \int_{\Omega^\prime} \int_{\Omega^\prime} \delta_a(x-y)\, \e^{-\i(x-y)\xi} \: \d x\,\d y 
\end{align*}
due to Fubini's Theorem, which is applicable in this setting as
\begin{align*}
\int_{\Omega^\prime} \int_{\Omega^\prime} \int_\Omega \lvert h_a(x,\omega) h_a(y,\omega) \rvert \: \d\P(\omega) \, \d x \,\d y
&= \int_{\Omega^\prime} \int_{\Omega^\prime} \Ex\left(\lvert h_a(x,\cdot) h_a(y,\cdot) \rvert\right) \: \d x \,\d y \\
&\leq \int_{\Omega^\prime} \int_{\Omega^\prime} \Ex\left(\lvert h_a(x,\cdot) \rvert^2\right)^{\nicefrac{1}{2}} \Ex\left(\lvert h_a(y,\cdot) \rvert^2\right)^{\nicefrac{1}{2}} \: \d x \,\d y  \\
&= \int_{\Omega^\prime} \int_{\Omega^\prime} \delta_a(0) \:  \d x \, \d y < \infty.
\end{align*}
Consequently,
\begin{align*}
\Ex\left(\lVert H_a \rVert_{\H^\alpha(\R^n)}^2\right) &= \int_\Omega \int_{\R^n} (1+ \lVert \xi \rVert_{\R^n}^2)^\alpha \, \lvert\Fourier h_a(\xi,\omega) \rvert^2 \: \d\xi \, \d\P(\omega) \\
&= \int_{\R^n} (1+ \lVert \xi \rVert_{\R^n}^2)^\alpha \int_{\Omega^\prime} \int_{\Omega^\prime} \delta_a(x-y) \, \e^{-\i(x-y)\xi} \: \d x \, \d y \, \d\xi < \infty
\end{align*}
by assumption, which completes the proof. 
\end{proof}

The last result shows that the smoothness of approximate white noise is solely determined by the covariance $\delta_a$. For example, realizations of the stationary Ornstein-Uhlenbeck process with zero mean and $\delta_a(t) = \frac{a}{2} \exp^{-a | t |}$ are almost surely continuous but nowhere differentiable, cf.~\cite[Theorem 1.2]{Doob1942} in combination with the equivalent characterization of the Ornstein-Uhlenbeck process therein and \cite{Payley1933}.

\section{Optimized filter function}\label{sec: optimized filter}

The aim of this section is to derive optimized filter functions for Radon data corrupted by white noise.
To this end, we consider the approximate FBP formula
\begin{equation*}
f_L^\varepsilon(x,y) = \frac{1}{2} \, \Back(\Fourier^{-1}[A_L(\sigma)\Fourier g^\varepsilon(\sigma,\varphi)])(x,y)
\quad \text{for } (x,y)\in\R^2
\end{equation*}
with noisy measurements $g^\varepsilon: \R \times [0,\pi) \to \R$ of noise level $\varepsilon \geq 0$, low-pass filter $A_L: \R \to \R$, given by 
\begin{equation*}
A_L(\sigma) = |\sigma|\,W(\nicefrac{\sigma}{L})
\quad \text{for } \sigma\in \R,
\end{equation*} 
fixed bandwidth $L>0$ and even window $W \in \L^{\infty}(\R)$ with compact support in $[-1,1]$. Observe that the assumptions on $W$ imply $A_{L} \in \L^2(\R)$ with 
\begin{equation*}
\lVert A_{L} \rVert_{\L^2(\R)}^2 = \int_{-L}^L \lvert \sigma \rvert^2 \, \lvert W(\nicefrac{\sigma}{L}) \rvert^2 \: \d \sigma \leq \lVert W \rVert_{\L^\infty(\R)}^2 \int_{-L}^L \lvert \sigma \rvert^2 \: \d \sigma = \frac{2}{3} \, L^3 \, \lVert W \rVert_{\L^\infty(\R)}^2 < \infty.
\end{equation*}
In addition, if $g^\varepsilon \in \L^2(\R\times [0,\pi))$ has compact support, i.e., there exists an $R \geq 0$ with
\begin{align*} 
g^\varepsilon(s,\varphi) = 0
\qquad \forall \: |s|>R, ~ \varphi \in [0,\pi),
\end{align*}
the approximate FBP reconstruction $f_L^\varepsilon \in \L^2(\R^2)$ is defined almost everywhere on $\R^2$ and can be rewritten as 
\begin{equation*}
f_L^\varepsilon = \frac{1}{2} \, \Back(\Fourier^{-1} A_L \ast g^\varepsilon)
\quad \text{almost everywhere on } \R^2,
\end{equation*}
where
\begin{equation*}
(\Fourier^{-1} A_L \ast g^\varepsilon)(s,\varphi) = \int_{\R} \Fourier^{-1} A_L(s-S,\varphi) \, g^\varepsilon(S,\varphi) \: \d S
\quad \text{for } (s,\varphi) \in \R \times [0,\pi).
\end{equation*}
Moreover, the (distributional) Fourier transform of $f_L^\varepsilon$ is given by 
\begin{align}\label{Eq: Fourier trafo FBP}
\Fourier f_L^\varepsilon(\sigma\cos(\varphi), \sigma\sin(\varphi)) = W(\nicefrac{\sigma}{L}) \, \Fourier g^\varepsilon(\sigma,\varphi)
\quad \text{for almost all } \sigma \in \R, ~ \varphi\in [0,\pi).
\end{align}
For details, we refer the reader to~\cite[Section 8.1]{Beckmann2019}.

\subsection{Continuous setting}\label{sec:continuous filter}

In Section~\ref{Sec:Noise}, we argued that a suitable model for the measurement noise in X-ray CT is given by approximate additive Gaussian white noise. Based on this model, we now derive an optimized filter function by minimizing the expectation of the squared $\L^2$-error of the approximate FBP reconstruction $f_L^\varepsilon$. More precisely, as the filter $A_L$ is uniquely determined by its window $W$, we consider the minimization problem
\begin{align*}
W^\ast = \argmin_{\supp(W)\subseteq[-1,1]} \Ex\left(\lVert f_L^\varepsilon - f \rVert_{\L^2(\R^2)}^2\right),
\end{align*}
where the real-valued target function $f\in\L^2(\R^2)$ is assumed to be compactly supported with 
\begin{equation*}
\supp(f)\subseteq B_R(0) \quad \text{for some } R > 0. 
\end{equation*}
Note that this assumption is not very restrictive as objects under investigation are usually of finite extend, which implies that the Radon transform $\Radon f$ of $f$ belongs to $\L^2(\R \times [0,\pi))$ with
\begin{equation*}
\Radon f (s,\varphi) = 0
\quad \forall \: \lvert s \rvert > R, ~ \varphi \in [0,\pi).
\end{equation*}
Moreover, we assume that the Radon transform of $f$ is corrupted by additive approximate white noise $H_a^\varepsilon: \Schwartz(\R^2) \times\Omega \to \R$ of noise level $\varepsilon \geq 0$ with support in $\Omega^\prime = [-R,R] \times [0,\pi)$ resulting in measured data $G^\varepsilon_a: \Schwartz(\R^2) \times \Omega \to \R$ defined as
\begin{equation}\label{Eq: Measured data with additive noise}
G^\varepsilon_a(\phi,\omega) = \int_0^\pi \int_{-R}^R g^\varepsilon_a(s,\varphi,\omega) \, \phi(s,\varphi) \: \d s \, \d \varphi
\quad \text{for } \phi \in \Schwartz(\R^2), ~ \omega \in \Omega
\end{equation}
with $g^\varepsilon_a(s,\varphi,\omega) = \Radon f(s,\varphi) + h^\varepsilon_a(s,\varphi,\omega)$ for all $s \in [-R,R]$, $\varphi\in[0,\pi)$, $\omega \in \Omega$, where $h^\varepsilon_a = \varepsilon \, h_a$.
Observe that we allow the case $\varepsilon=0$, which corresponds to noiseless measured data and degenerate approximate white noise taking the value 0 with probability 1.
Moreover, as $\Radon f$ is $2\pi$-periodic in the angular variable $\varphi$, it would be more appropriate to consider the measured data $G_a^{\varepsilon}$ as a function from $\Schwartz(\R \times \Sphere) \times \Omega$ to $\R$. This, however, does not change our results, but only affects the notation.
Lemma~\ref{Lem: Smoothness approx noise} in combination with $\Radon f \in \L^2(\R \times [0,\pi))$ implies that $G^\varepsilon_a(\cdot,\omega) \in \L^2(\R \times [0,\pi))$ for almost all $\omega \in \Omega$ and, thus, $F_L^\varepsilon: \Schwartz(\R^2) \times \Omega \to \R$ with
\begin{equation}\label{eq: FBP formula with conv}
F_L^\varepsilon(\phi,\omega) = \int_{\R^2} f_L^\varepsilon (x,y,\omega) \, \phi(x,y) \: \d(x,y)
\quad \text{for } \phi \in \Schwartz(\R^2), ~ \omega \in \Omega,
\end{equation}
where 
\begin{equation*}
f_L^\varepsilon(x,y,\omega) = \frac{1}{2} \, \Back(\Fourier^{-1} A_L(s) \ast g^\varepsilon_a(s,\varphi,\omega))(x,y)
\quad \text{for almost all } (x,y) \in \R^2,
\end{equation*}
is well-defined with $F_L^\varepsilon(\cdot,\omega) \in \L^2(\R^2)$ for almost all $\omega \in \Omega$.
Moreover, $F_L^\varepsilon (\phi,\cdot):\Omega\to\R$ is a random variable for fixed $\phi\in\Omega$.
In what follows, we identify $F_L^\varepsilon$ with $f_L^\varepsilon$ as well as $G^\varepsilon_a$ with $g^\varepsilon_a$ to improve readability.

With these preliminaries, we now calculate the expectation of the squared $\L^2(\R^2)$-norm of the error $f_L^\varepsilon - f$.

\begin{lemma}\label{Lem:Expected value error} 
Let $f \in \L^2(\R^2)$ with $\supp(f) \subseteq B_R(0)$ for fixed $R > 0$ and $A_L \in \L^2(\R)$ for $L > 0$. Moreover, assume that the measured data $G^\varepsilon_a: \Schwartz(\R^2) \times \Omega \to \R$ is defined as in~\eqref{Eq: Measured data with additive noise}.
Then, 
\begin{align*}
\Ex\left(\lVert f_L^\varepsilon - f \rVert_{\L^2(\R^2)}^2 \right) = \frac{L^2}{4\pi^2} \int_\R \lvert \sigma \rvert \Big[& (W(\sigma)-1)^2 \int_0^\pi \lvert\Fourier (\Radon f)(L\sigma,\varphi) \rvert^2 \: \d \varphi \\
&+ \pi \, \varepsilon^2 \, W(\sigma)^2 \, \Ex\left(\lvert \Fourier h_a(L\sigma,0,\cdot) \rvert^2\right)\Big] \: \d \sigma
\end{align*} 
with 
\begin{equation*}
\Ex\left(\lvert \Fourier h_a(L\sigma,0,\cdot) \rvert^2\right) = \int_{-R}^R \int_{-R}^R \delta_a(s-\hat{s},0) \, \e^{-i(s-\hat{s}) \, L \sigma} \: \d s \, \d \hat{s}.
\end{equation*}
\end{lemma}

\begin{proof} 
Let us start with determining the $\L^2(\R^2)$-norm of $f_L^\varepsilon(\cdot,\omega) - f$ with fixed $\omega \in \Omega$. To this end, recall that $f_L^\varepsilon(\cdot,\omega)\in\L^2(\R^2)$ for almost all $\omega\in\Omega$ and $f \in \L^2(\R^2)$ by assumption. As a result, the Rayleigh-Plancherel theorem implies
\begin{align*}
\lVert f_L^\varepsilon(\cdot,\omega) - f \rVert_{\L^2(\R^2)}^2 &= \frac{1}{4\pi^2} \, \lVert \Fourier(f_L^\varepsilon(\cdot,\omega)) - \Fourier f \rVert_{\L^2(\R^2)}^2 \\
&= \frac{1}{4\pi^2} \int_{0}^\pi \int_\R \lvert [\Fourier (f_L^\varepsilon(\cdot,\cdot,\omega)) - \Fourier f](\sigma \cos(\varphi),\sigma\sin(\varphi)) \rvert^2 \, \lvert \sigma \rvert \: \d \sigma \, \d \varphi
\end{align*}
for almost all $\omega \in \Omega$ by transitioning to polar coordinates. Observe that $f \in \L^1(\R^2)$ as $f \in \L^2(\R^2)$ with $\supp(f) \subseteq B_R(0)$ and, as a consequence, the classical Fourier slice theorem
\begin{equation*}
\Fourier(\Radon f)(\sigma,\varphi) = \Fourier f(\sigma \cos(\varphi), \sigma \sin(\varphi))
\quad \forall \: \sigma \in \R, ~ \varphi \in [0,\pi)
\end{equation*}
and~\eqref{Eq: Fourier trafo FBP} yield
\begin{align*}
&\lVert f_L^\varepsilon(\cdot,\omega) - f \rVert_{\L^2(\R^2)}^2 = \frac{1}{4\pi^2} \int_{0}^\pi \int_\R \lvert W(\nicefrac{\sigma}{L}) \, \Fourier g^\varepsilon_a(\sigma,\varphi,\omega) - \Fourier(\Radon f)(\sigma,\varphi) \rvert^2 \, \lvert \sigma \rvert \: \d \sigma \, \d \varphi\\
&\qquad \begin{aligned}= \frac{1}{4\pi^2} \int_{0}^\pi \int_\R \lvert \sigma \rvert \, \Big[& W(\nicefrac{\sigma}{L})^2 \, |\Fourier g^\varepsilon_a(\sigma,\varphi,\omega)|^2 - W(\nicefrac{\sigma}{L}) \, \Fourier(\Radon f)(\sigma,\varphi) \, \overline{\Fourier g^\varepsilon_a(\sigma,\varphi,\omega)} \\
& - W(\nicefrac{\sigma}{L}) \, \overline{\Fourier(\Radon f)(\sigma,\varphi)} \, \Fourier g^\varepsilon_a(\sigma,\varphi,\omega) + \lvert \Fourier(\Radon f)(\sigma,\varphi) \rvert^2\Big] \: \d \sigma \, \d \varphi.\end{aligned}
\end{align*}
By exploiting $g^\varepsilon_a = \Radon f + h^\varepsilon_a$ with $h^\varepsilon_a = \varepsilon \, h_a$, we obtain
\begin{align*}
\lVert f_L^\varepsilon(\cdot,\omega) - f \rVert_{\L^2(\R^2)}^2 = \frac{1}{4\pi^2} \int_{0}^\pi \int_\R \lvert \sigma \rvert \Big[& (W(\nicefrac{\sigma}{L})-1)^2 \, \lvert \Fourier(\Radon f)(\sigma,\varphi) \rvert^2 + W(\nicefrac{\sigma}{L})^2 \, \lvert \Fourier h^\varepsilon_a(\sigma,\varphi,\omega) \rvert^2 \\
&+ \left(W(\nicefrac{\sigma}{L})^2 -  W(\nicefrac{\sigma}{L})\right) \big[ \Fourier(\Radon f)(\sigma,\varphi) \, \overline{\Fourier  h^\varepsilon_a(\sigma,\varphi,\omega)} \\
&+ \overline{\Fourier (\Radon f)(\sigma,\varphi)} \, \Fourier   h^\varepsilon_a(\sigma,\varphi,\omega)\big]\Big] \: \d \sigma \, \d \varphi
\end{align*}
for almost all $\omega\in\Omega$.
Now, for the expectation of $\lVert f_L^\varepsilon - f \rVert_{\L^2(\R^2)}^2: \Omega \to \overline{\R}$ follows that
\begin{align*}
\Ex\left(\lVert f_L^\varepsilon - f \rVert_{\L^2(\R^2)}^2\right) &= \int_\Omega \lVert f_L^\varepsilon(\cdot,\omega) - f \rVert_{\L^2(\R^2)}^2 \: \d \P(\omega) \\
& \begin{aligned}= \frac{1}{4\pi^2} \int_\Omega \int_{0}^\pi \int_\R \lvert \sigma \rvert  \Big[&(W(\nicefrac{\sigma}{L})-1)^2 \, \lvert \Fourier(\Radon f)(\sigma,\varphi) \rvert^2 \\
& + W(\nicefrac{\sigma}{L})^2 \, \lvert \Fourier h^\varepsilon_a(\sigma,\varphi,\omega) \rvert^2 \\
& + \left(W(\nicefrac{\sigma}{L})^2 -  W(\nicefrac{\sigma}{L})\right) \big[\Fourier(\Radon f)(\sigma,\varphi) \, \overline{\Fourier  h^\varepsilon_a(\sigma,\varphi,\omega)} \\
& + \overline{\Fourier(\Radon f)(\sigma,\varphi)} \, \Fourier  h^\varepsilon_a(\sigma,\varphi,\omega)\big]\Big] \: \d \sigma \, \d \varphi \, \d \P(\omega)\end{aligned}\\
& \begin{aligned}= \frac{1}{4\pi^2} \int_{0}^\pi \int_\R \lvert \sigma \rvert \Big[& (W(\nicefrac{\sigma}{L})-1)^2 \, \lvert \Fourier(\Radon f)(\sigma,\varphi) \rvert^2 \\
& + W(\nicefrac{\sigma}{L})^2 \, \Ex(\lvert \Fourier h^\varepsilon_a(\sigma,\varphi,\cdot) \rvert^2) \\
& + \left(W(\nicefrac{\sigma}{L})^2 -  W(\nicefrac{\sigma}{L})\right) \big[\Fourier(\Radon f)(\sigma,\varphi) \, \Ex\left(\overline{\Fourier  h^\varepsilon_a(\sigma,\varphi,\cdot)}\right) \\
& + \overline{\Fourier(\Radon f)(\sigma,\varphi)} \, \Ex(\Fourier  h^\varepsilon_a(\sigma,\varphi,\cdot))\big]\Big] \: \d \sigma \, \d \varphi.\end{aligned}
\end{align*}
The expectation of $\Fourier h^\varepsilon_a(\sigma,\varphi,\cdot)$ is zero as
\begin{equation*}
\Ex(\Fourier h^\varepsilon_a(\sigma,\varphi,\cdot)) = \varepsilon \int_{-R}^R \Ex(h_a(s,\varphi,\cdot)) \, \e^{-\i s\sigma} \: \d s = 0
\quad \text{for all } s \in \R, ~ \varphi \in [0,\pi)
\end{equation*}
due to the definition of $h_a$ and Fubini's Theorem, which also implies that 
\begin{equation*}
\Ex\left(\overline{\Fourier h^\varepsilon_a(\sigma,\varphi,\cdot)}\right) =  \overline{\Ex (\Fourier  h^\varepsilon_a(\sigma,\varphi,\cdot))} = 0
\quad \text{for all } s \in \R, ~ \varphi \in [0,\pi).
\end{equation*} 
Similarly, for the variance of $\Fourier h^\varepsilon_a(\sigma,\varphi,\cdot)$ follows that
\begin{align*}
\Ex(\lvert \Fourier h^\varepsilon_a(\sigma,\varphi,\cdot) \rvert^2) &= \varepsilon^2 \Ex(\lvert \Fourier h_a(\sigma,\varphi,\cdot) \rvert^2)\\
&= \varepsilon^2 \int_{-R}^R \int_{-R}^R \Ex[h_a(s,\varphi,\cdot) h_a(\hat{s},\varphi,\cdot)]\,\e^{-i (s-\hat{s}) \sigma} \: \d s\,\d\hat{s} \\
&= \varepsilon^2 \int_{-R}^R \int_{-R}^R \delta_a(s-\hat{s},0)\,\e^{-i (s-\hat{s}) \sigma}\: \d s\,\d\hat{s}
\end{align*}
for all $\sigma\in\R$ and $\varphi\in [0,\pi)$.
As a result, we finally obtain
\begin{align*}
\Ex\left(\lVert f_L^\varepsilon - f \rVert_{\L^2(\R^2)}^2\right) = \frac{L^2}{4\pi^2} \int_\R \lvert \sigma \rvert \Big[& (W(\sigma)-1)^2 \int_0^\pi \lvert \Fourier(\Radon f)(L\sigma,\varphi) \rvert^2 \: \d \varphi \\
& + \pi \, \varepsilon^2 \, W(\sigma)^2 \int_{-R}^R \int_{-R}^R \delta_a(s-\hat{s},0) \, \e^{-i(s-\hat{s}) \, L \sigma} \: \d s \, \d \hat{s} \Big] \: \d \sigma,
\end{align*}
which completes the proof.
\end{proof}

We now proceed to determining a window function $W^\ast \in \L^\infty(\R)$ with $W^\ast(\sigma) = W^\ast(-\sigma)$ for all $\sigma \in \R$ and $\supp(W^\ast) \subseteq [-1,1]$ that minimizes
\begin{equation*}
\Ex\left(\lVert f_L^\varepsilon - f \rVert_{\L^2(\R^2)}^2\right) = \frac{L^2}{4\pi^2} \int_\R k(\sigma,W(\sigma)) \: \d \sigma 
\end{equation*}
with $k: \R \times \R \to \R$ defined as
\begin{equation}\label{Def: k_sigma}
k(\sigma,w) = \lvert \sigma \rvert (w-1)^2 \int_0^\pi \lvert \Fourier(\Radon f)(L\sigma,\varphi) \rvert^2 \: \d \varphi + \pi \, \varepsilon^2 \, \lvert \sigma \rvert \, w^2 \, \Ex\left(\lvert \Fourier h_a(L\sigma,0,\cdot) \rvert^2\right)
\end{equation}
for $\sigma, w \in \R$. 
To this end, we minimize $k(\sigma,\cdot)$ pointwise for fixed $\sigma \in \R$ and use the monotonicity of the integral to obtain a candidate for $W^\ast$.
Since the window $W$ is assumed to have compact support in $[-1,1]$, we obtain
\begin{equation}
\Ex\left(\lVert f_L^\varepsilon - f \rVert_{\L^2(\R^2)}^2\right) = \frac{L^2}{4\pi^2} \int_{-1}^1 k(\sigma,W(\sigma)) \: \d \sigma + c,
\end{equation}
where $c \in \R$ does not depend on $W$, and, therefore, it suffices to consider $k$ as a function on $[-1,1] \times\R$.
Hence, in the following lemma, we derive a formula for the global minimizer of $k(\sigma,\cdot): \R \to \R$ for $\sigma \in [-1,1]$.

\begin{lemma}\label{Lem:Minimizer k sigma} 
Let $f \in \L^2(\R^2)$ with $\supp(f) \subseteq B_R(0)$ for $R > 0$ and let $H_a^\varepsilon: \Schwartz(\R^2)\times\Omega\to\R$ be approximate white noise with support in $[-R,R] \times [0,\pi)$ and noise level $\varepsilon\geq 0$.
Then, for $\sigma \in [-1,1]$, the global minimizer $w_\sigma^\ast \in \R$ of $k(\sigma,\cdot): \R \to \R$ in~\eqref{Def: k_sigma} is given by
\begin{align*}
w_\sigma^\ast = \begin{dcases} \frac{\int_{0}^\pi \lvert \Fourier(\Radon f)(L\sigma,\varphi) \rvert^2 \: \d \varphi}{\int_{0}^\pi \lvert \Fourier(\Radon f)(L\sigma,\varphi) \rvert^2 \: \d \varphi + \pi \, \varepsilon^2 \, \Ex\left(\lvert \Fourier h_a(L\sigma,0,\cdot) \rvert^2\right)} &\text{for } \sigma \in D \\
1 &\text{for } \sigma \in [-1,1] \setminus D \end{dcases}
\end{align*} 
with 
\begin{equation*}
D = \Bigl\{\sigma\in [-1,1] \setminus \{0\} \Bigm| \varepsilon^2 \, \Ex\left(\lvert \Fourier h_a(L\sigma,0,\cdot) \rvert^2\right)  > 0\Bigr\},
\end{equation*}
which is unique for $\sigma \in D$.
\end{lemma}

\begin{proof}
The proof can be divided into three cases:

\smallskip

\noindent
\textit{Case $\sigma\in D$:}
The partial derivative of $k$ w.r.t.\ the second variable $w$ is given by
\begin{equation*}
\frac{\partial k}{\partial w} (\sigma,w) = 2 \, \lvert \sigma \rvert \, (w-1) \int_0^\pi \lvert \Fourier(\Radon f)(L\sigma,\varphi) \rvert^2 \: \d \varphi + 2\pi \, \varepsilon^2 \, \lvert \sigma \rvert \, w \, \Ex\left(\lvert \Fourier h_a(L\sigma,0,\cdot) \rvert^2\right)
\quad \text{for } w \in \R 
\end{equation*}
and, thus, the necessary condition of a minimum $w_\sigma^\ast$ of $k(\sigma,\cdot)$ reads
\begin{equation*}
\frac{\partial k}{\partial w} (\sigma,w_\sigma^\ast) = 0
\quad \iff \quad
w_\sigma^\ast = \frac{\int_{0}^\pi \lvert \Fourier(\Radon f)(L\sigma,\varphi) \rvert^2 \: \d \varphi}{\int_{0}^\pi \lvert \Fourier(\Radon f)(L\sigma,\varphi) \rvert^2 \: \d \varphi + \pi \, \varepsilon^2 \, \Ex\left(\lvert \Fourier h_a(L\sigma,0,\cdot) \rvert^2\right)}.
\end{equation*}
Moreover, the second partial derivative
\begin{equation*}
\frac{\partial^2 k}{\partial w^2} (\sigma,w) = 2 \, \lvert \sigma \rvert \int_0^\pi \lvert \Fourier(\Radon f)(L\sigma,\varphi) \rvert^2 \: \d \varphi + 2\pi \, \varepsilon^2 \, \lvert \sigma \rvert \, \Ex\left(\lvert \Fourier h_a(L\sigma,0,\cdot) \rvert^2\right)
\quad \text{for } w \in \R
\end{equation*}
is strictly larger than zero for all $\sigma \in D$ so that $k(\sigma,\cdot)$ is strictly convex and $w_\sigma^\ast$ is the unique global minimizer of $k(\sigma,\cdot)$ on $\R$.

\smallskip

\noindent
\textit{Case $\sigma = 0$:}
In this case, we have $k(\sigma,w) = 0$ for all $w \in \R$ and, consequently, $w_\sigma^\ast = 1$ is a global minimizer of $k(\sigma,\cdot)$ on $\R$.
\smallskip

\noindent
\textit{Case $\sigma \in [-1,1] \setminus (D \cup \{ 0 \} )$:}
In this setting, $k(\sigma,w)$ reads
\begin{equation*}
k(\sigma,w) = \lvert \sigma \rvert \, (w-1)^2 \int_0^\pi \lvert \Fourier(\Radon f)(L\sigma,\varphi) \rvert^2 \: \d \varphi
\quad \text{for } w \in \R 
\end{equation*}
with $k(\sigma,w) \geq 0$ for all $w \in \R$ and $k(\sigma,1) = 0$. Therefore, $w_\sigma^\ast = 1$ is a global minimizer.
\end{proof} 

The preceding lemma, in combination with the monotonicity property of the integral, shows that a candidate for an optimal window function $W^\ast: \R \to \R$ with $\supp(W^\ast) \subseteq [-1,1]$ is given by 
\begin{align}\label{Eq:Candidate W0}
W^\ast(\sigma) = \begin{dcases} \frac{\int_{0}^\pi \lvert \Fourier(\Radon f)(L\sigma,\varphi) \rvert^2 \: \d \varphi}{\int_{0}^\pi \lvert \Fourier(\Radon f)(L\sigma,\varphi) \rvert^2 \: \d \varphi + \pi \, \varepsilon^2 \, \Ex\left(\lvert \Fourier h_a(L\sigma,0,\cdot) \rvert^2\right)} &\text{for } \sigma \in D \\
1 &\text{for } \sigma \in [-1,1] \setminus D \\
0 &\text{else} \end{dcases}
\end{align} 
with 
\begin{equation*}
D = \Bigl\{ \sigma\in [-1,1] \setminus \{0\} \Bigm| \varepsilon^2 \,
\Ex\left(\lvert \Fourier h_a(L\sigma,0,\cdot) \rvert^2\right) > 0\Bigr\}.
\end{equation*}
It remains to prove that the candidate $W^\ast$ is indeed a window, i.e., $W^\ast \in \L^\infty(\R)$ is even.

\begin{lemma}\label{Lem:Prop W} 
Assume that $f \in \L^2(\R^2)$ is real-valued with $\supp(f) \subseteq B_R(0)$ for $R > 0$ and let $H_a^\varepsilon: \Schwartz(\R^2)\times\Omega\to\R$ be approximate white noise with support in $[-R,R] \times [0,\pi)$ and noise level $\varepsilon\geq 0$.
Then, $W^\ast:\R\to\R$ in~\eqref{Eq:Candidate W0} is well-defined, even and in $\L^\infty(\R)$.
\end{lemma}

\begin{proof}
First, observe that
\begin{align*}
\int_0^\pi \lvert \Fourier(\Radon f)(L\sigma,\varphi) \rvert^2 \: \d \varphi = \int_0^\pi \lvert \Fourier f(L\sigma\cos(\varphi),L\sigma\sin(\varphi)) \rvert^2 \: \d \varphi \leq \pi \norm{\Fourier f}_\infty^2 < \infty
\end{align*}
and
\begin{align*}
\Ex\left(\lvert \Fourier h_a(L\sigma,0,\cdot) \rvert^2\right) &\leq \int_{-R}^R \int_{-R}^R \Ex\left(\lvert h_a(s,0,\cdot) \, h_a(\hat{s},0,\cdot) \rvert\right) \: \d s \, \d \hat{s} \\
&\leq \int_{-R}^R \int_{-R}^R \Ex\left(\lvert h_a(s,0,\cdot) \rvert^2 \right)^{\nicefrac{1}{2}} \Ex\left(\lvert h_a(\hat{s},0,\cdot) \rvert^2\right)^{\nicefrac{1}{2}} \: \d s \,\d \hat{s} = 4 R^2 \,\delta_a(0) < \infty
\end{align*}
for all $\sigma\in\R$ due to the Fourier slice theorem, the Riemann-Lebesgue lemma, Fubini's theorem and Cauchy Schwarz's inequality.
In addition, 
\begin{equation*}
\lvert W^\ast(\sigma) \rvert = \frac{\int_{0}^\pi \lvert \Fourier(\Radon f)(L\sigma,\varphi) \rvert^2 \: \d \varphi}{\int_{0}^\pi \lvert \Fourier(\Radon f)(L\sigma,\varphi) \rvert^2 \: \d \varphi + \pi \, \varepsilon^2 \, \Ex\left(\lvert \Fourier h_a(L\sigma,0,\cdot) \rvert^2\right)}\leq 1
\quad \text{for } \sigma\in D
\end{equation*}
and $W^\ast(\sigma) = 1$ for $\sigma \in [-1,1] \setminus D$. Consequently, $W^\ast$ is well-defined for all $\sigma\in\R$ and satisfies $W^\ast \in \L^\infty(\R)$. 
In order to prove that $W^\ast$ is an even function, note that 
\begin{equation*}
\lvert \Fourier g(\sigma,\varphi) \rvert^2 = \Big\lvert \int_\R g(s,\varphi)\,\e^{-\i s\sigma} \: \d s\Big\rvert^2 = \Big\lvert \int_\R \overline{g(s,\varphi)\,\e^{-\i s (-\sigma)}} \: \d s\Big\rvert^2 = \lvert \Fourier g(-\sigma,\varphi)  \rvert^2 
\end{equation*}
for all real-valued $g\in\L^2(\R\times [0,\pi))$ and all $\varphi\in [0,\pi)$. This implies that
\begin{equation*}
\lvert\Fourier (\Radon f)(L\sigma,\varphi) \rvert^2\, = \lvert \Fourier (\Radon f)(-L\sigma,\varphi) \rvert^2
\quad \text{and} \quad
\Ex\left(\lvert \Fourier h_a(L\sigma,0,\cdot) \rvert^2\right) = \Ex\left(\lvert \Fourier h_a(-L\sigma,0,\cdot) \rvert^2\right)
\end{equation*}
for all $\sigma \in \R$ and $\varphi \in [0,\pi)$. Consequently, we have $\sigma \in D$ if and only if $-\sigma \in D$ and $W^\ast(\sigma) = W^\ast(-\sigma)$ for all $\sigma \in \R$.
\end{proof}

The above Lemmata show that $W^\ast$ is an optimized window function for the approximate FBP reconstruction in the sense that it minimizes the expectation of the squared $\L^2(\R^2)$-norm of $f_L^\varepsilon - f$ for window functions with compact support in $[-1,1]$. This is summarized in the following theorem.

\begin{theorem}[Optimized filter with compact support]\label{theo:optimized_filter_continuous}
Assume that $f \in \L^2(\R^2)$ is real-valued with $\supp(f) \subseteq B_R(0)$ for fixed $R > 0$. In addition, assume that the Radon transform of $f$ is corrupted by additive approximate white noise $H_a^\varepsilon: \Schwartz(\R^2)\times\Omega\to\R$ with support in $[-R,R]\times [0,\pi)$ and noise level $\varepsilon\geq 0$, i.e., the measured data is given by $G^\varepsilon_a: \Schwartz(\R^2) \times \Omega \to \R $ with
\begin{equation*}
G^\varepsilon_a(\phi,\omega) = \int_0^\pi \int_{-R}^R g^\varepsilon_a(s,\varphi,\omega) \, \phi(s,\varphi) \: \d s \, \d\varphi
\quad \text{for } \phi \in \Schwartz(\R^2), ~ \omega \in \Omega,
\end{equation*}
where $g^\varepsilon_a(s,\varphi,\omega) = \Radon f(s,\varphi) + \varepsilon \, h_a(s,\varphi,\omega)$ for all $s \in [-R,R]$, $\varphi \in [0,\pi)$, $\omega \in \Omega$. 
Then, $A_L^\ast: \R \to \R$, defined by
\begin{align*}
A_L^\ast(\sigma) = \begin{dcases} \frac{\lvert \sigma \rvert \int_{0}^\pi \lvert \Fourier(\Radon f)(\sigma,\varphi) \rvert^2 \: \d \varphi}{\int_{0}^\pi \lvert \Fourier(\Radon f)(\sigma,\varphi) \rvert^2 \: \d \varphi + \pi \, \varepsilon^2 \, \Ex\left(\lvert \Fourier h_a(\sigma,0,\cdot) \rvert^2\right)} & \text{for } \sigma \in D_L \\
\lvert \sigma \rvert & \text{for } \sigma \in [-L,L] \setminus D_L \\
0 & \text{else} \end{dcases}
\end{align*} 
with
\begin{equation*}
D_L = \Bigl\{\sigma\in [-L,L] \setminus \{0\} \Bigm| \varepsilon^2 \,
\Ex\left(\lvert \Fourier h_a(\sigma,0,\cdot) \rvert^2\right) > 0\Bigr\}
\end{equation*} 
and 
\begin{equation*}
\Ex\left(\lvert \Fourier h_a(\sigma,0,\cdot) \rvert^2\right) = \int_{-R}^R \int_{-R}^R \delta_a(s-\hat{s},0) \, \e^{-\i(s-\hat{s}) \, \sigma} \: \d s \, \d \hat{s},
\end{equation*}
is the optimized filter function with compact support that minimizes
\begin{align*}
\Ex\left(\lVert f_L^\varepsilon - f \rVert_{\L^2(\R^2)}^2\right) 
\end{align*}
for filter functions $A_{L}(\cdot) = \lvert\,\cdot\,\rvert\, W(\nicefrac{\cdot}{L})$ with $\supp(A_L) \subseteq [-L,L]$ and fixed $L > 0$.
\end{theorem}

\begin{proof}
The statement is a direct consequence of Lemmata~\ref{Lem:Expected value error}, \ref{Lem:Minimizer k sigma}, \ref{Lem:Prop W}, the monotonicity property of the integral and the identity $A_L^\ast(\cdot) = \lvert\,\cdot\,\rvert \, W^\ast(\nicefrac{\cdot}{L})$.
\end{proof}

Observe that the optimized filter function depends on the covariance $\delta_a$ of $h_a$, the noise level $\varepsilon$ and the true Radon transform $\Radon f$ of the target function $f$. The latter, however, is typically not known in applications, which makes the filter function difficult to use. We will address this in Chapter~\ref{sec:Numerik}, where we present a possible adaptation of the above filter to the case of unknown true data.

In the case of noise-free measurement data, i.e., $\varepsilon=0$, the optimized filter function reduces to 
\begin{align*}
A_L^\ast(\sigma) = \begin{dcases} \lvert \sigma \rvert & \text{for } \sigma \in [-L,L] \\
0 & \text{else.} \end{dcases}
\end{align*}
Consequently, in this case, the optimized filter function is equal to the Ram-Lak filter. Similar observations were made in~\cite{Beckmann2019, Beckmann2020}.

To close this section, we consider the example of the above mentioned stationary Ornstein-Uhlenbeck process with zero mean and $\delta_a(t) = \frac{a}{2}\exp^{-a\,\lvert t \rvert}$ with $t\in\R$, which we extend to $\R \times [0,\pi)$ by setting
\begin{equation*}
\delta_a(t,\varphi) = \frac{a^2}{4}\exp^{-a\,\lvert t \rvert} \exp^{-a\,\lvert \varphi \rvert}
\quad \text{for } t \in \R, ~ \varphi \in [0,\pi).
\end{equation*}
Standard computations then give
\begin{equation*}
\Ex\left(\lvert \Fourier h_a(\sigma,0,\cdot) \rvert^2\right) = a^2 \left( \frac{a R}{a^2+\sigma^2} + \frac{\frac{1}{2}\left(\e^{-2aR}\cos(2R\sigma)-1\right)\left(a^2-\sigma^2\right)-a\sigma\e^{-2aR}\sin(2R\sigma)}{a^4+2a^2\sigma^2+\sigma^2}  \right)
\end{equation*}
for all $\sigma\in\R$ and $a>0$. Hence, in the limit $a \to \infty$
we obtain
\begin{equation*}
\Ex\left(\lvert \Fourier h_a(\sigma,0,\cdot) \rvert^2\right) \xrightarrow{a \to \infty} \infty
\quad \text{for all } \sigma \in \R
\end{equation*}
and the corresponding optimized filter function $A_L^\ast(\sigma)$ converges to 0 for all $\sigma\in\R$ and $\varepsilon>0$.
This behaviour is reasonable as $H_a$ converges weakly to white noise for $a \to \infty$, cf.~Lemma~\ref{Lem:Convergence of approximate white noise}, with $\Ex(h_a(x,\cdot)^2) = \delta_a(0) \xrightarrow{a \to \infty} \infty$.
In this sense, white noise has an infinite variance and corresponds to a process requiring an infinite amount of energy as already highlighted in Section~\ref{Sec:Noise}. To compensate for this, the filter function must suppress all frequencies and, thus, $A_L^\ast \equiv 0$.
The same behaviour can be observed for a larger variety of covariance functions. 
To this end, consider fixed $\varepsilon^2>0$ and assume that the covariance $\delta_a: \R^2 \to \R$ of the approximate white noise is of the form
\begin{equation*}
\delta_a(s,\varphi) = a^2 \, \phi(a s) \, \phi(a \varphi)
\quad \text{for } s, \varphi \in \R,
\end{equation*}
where $\phi \in \L^1(\R)$ with $\phi(x) \geq 0$ for all $x \in \R$, $\phi(0) > 0$ and $\int_\R \phi(x) \: \d x = 1$. Then, $\delta_a$ is a Dirac sequence and 
\begin{align*}
\Ex\left(\lvert \Fourier h_a(\sigma,0,\cdot) \rvert^2\right) &= a^2 \, \phi(0) \int_{-R}^R \int_{-R}^R \phi(a(s-\hat{s})) \, \e^{-\i(s-\hat{s}) \, \sigma} \: \d s \, \d \hat{s}
\quad \text{for all } \sigma \in \R.
\end{align*}
Observe that 
\begin{align*}
\int_{-R}^R \int_{-R}^R a \, \phi(a(s-\hat{s})) \, \e^{-\i(s-\hat{s}) \, \sigma} \: \d s \, \d \hat{s} \xrightarrow{a\to\infty} 2R > 0
\end{align*}
uniformly for $\sigma\in[-L,L]$, which follows from Lebesgue's dominated convergence theorem, the uniform convergence of the complex exponential function on any bounded interval and the properties of $\phi$. As a result, there exists an $a_0 > 0$ with $\Ex\left(\lvert \Fourier h_a(\sigma,0,\cdot) \rvert^2\right) > 0$ for all $\sigma \in [-L,L]$ and all $a \geq a_0$. Consequently, 
\begin{align*}
A_L^\ast(\sigma) = \frac{\lvert \sigma \rvert \int_{0}^\pi \lvert \Fourier(\Radon f)(\sigma,\varphi) \rvert^2 \: \d \varphi}{\int_{0}^\pi \lvert \Fourier(\Radon f)(\sigma,\varphi) \rvert^2 \: \d \varphi + \pi \, \varepsilon^2 \, \Ex\left(\lvert \Fourier h_a(\sigma,0,\cdot) \rvert^2\right)}
\quad \text{for } \sigma \in [-L,L], ~ a \geq a_0,
\end{align*} 
and in the limit $a\to\infty$, we obtain
\begin{align*}
A_L^\ast(\sigma) \xrightarrow{a\to\infty} 0
\quad \text{for all } \sigma \in \R,
\end{align*} 
which follows from
\begin{align*}
\Ex\left(\lvert \Fourier h_a(\sigma,0,\cdot) \rvert^2\right) &= a \, \phi(0) \int_{-R}^R \int_{-R}^R a \, \phi(a(s-\hat{s})) \, \e^{-\i(s-\hat{s}) \, \sigma} \: \d s \, \d \hat{s} \xrightarrow{a \to \infty} \infty.
\end{align*}

\subsection{Discrete setting}\label{sec: discrete filter}

The optimized filter function derived in the last section is based on the assumption that the measured data is known for every line in the plane. In real-world applications, however, this is not the case, since only finitely many measurements are available. 
Consequently, the approximate FBP formula and with this the optimized filter function are not applicable in these settings and need to be adapted to discrete measurement data.
For simplicity, we assume that the data is measured using a parallel scanning geometry.
To be more precise, the discrete measured data $g_D^\varepsilon:\{ -M,\dots , M \}\times \{ 0,\dots , N_\varphi-1 \} \times \Omega\to\R$ with $M,\,N_\varphi\in\N$ is given by 
\begin{equation*}
g_D^\varepsilon(i,j,\omega) = \Radon f(s_i,\varphi_j) + h_a^\varepsilon(s_i,\varphi_j,\omega)
\quad \text{for } -M\leq i\leq M,\, 0\leq j\leq N_\varphi-1,\, \omega\in\Omega 
\end{equation*}
with $s_i = i h$ for $h>0$, $\varphi_j = j\, \nicefrac{\pi}{N_\varphi}$ and $g_a^\varepsilon:[-R,R]\times [0,\pi)\times\Omega\to\R$ as in \eqref{Eq: Measured data with additive noise}, where the radial discretization is chosen to cover the entire interval, i.e., $M h  \geq R$.
Moreover, we assume that $h_a$ is sufficiently close to white noise such that $g_D^\varepsilon(i,j,\cdot):\Omega\to\R$ and $g_D^\varepsilon(k,l,\cdot):\Omega\to\R$ appear to be uncorrelated if $i\neq k$ or $j\neq l$.
More precisely, the covariance of $h_a$ has to satisfy
\begin{equation}\label{eq: discrete white noise uncorrelated}
\Ex\left(h_a(s_i,\varphi_j,\cdot) h_a(s_k,\varphi_l,\cdot)\right) = \delta_a(s_i-s_k, \varphi_j -\varphi_l) = 0 \quad \text{if } i\neq k \text{ or } j\neq l
\end{equation}
for all $-M\leq i, k\leq M$ and $0\leq j, l\leq N_\varphi-1$. This is for example the case if 
\begin{equation*}
\delta_a(s,\varphi) = \begin{cases}
a^2 \quad &\text{if } \lvert s \rvert \leq \frac{1}{2a},\, \lvert \varphi\rvert \leq \frac{1}{2a}, \\
0 &\text{else}        
\end{cases}
\end{equation*}
with $a > \max\left(\frac{1}{2h}, \frac{N_\varphi}{2\pi}\right)$. In this setting, the discrete measured data $g_D^\varepsilon$ can be equivalently represented as
\begin{equation*}
g_D^\varepsilon(i,j,\omega) = \Radon f(s_i,\varphi_j) + \xi^{\varepsilon_a}_{i,j}(\omega)
\quad \text{for } -M\leq i\leq M,\, 0\leq j\leq N_\varphi-1,\, \omega\in\Omega 
\end{equation*}
with independent and identically distributed Gaussian random variables $\xi^{\varepsilon_a}_{i,j}:\Omega\to\R$ with expectation zero and variance $\varepsilon_a^2 = \varepsilon^2\,\delta_a(0,0)$, which corresponds to the classical definition of discrete additive Gaussian white noise in the literature.
Hence, in the following we use this more common definition of $g_D^\varepsilon$.
In addition, we assume that the target function $f:\R^2\to\R$ is real-valued and bounded with 
\begin{equation*}
\supp(f)\subseteq B_R(0)
\quad \text{for some } R > 0. 
\end{equation*}

To apply the approximate FBP formula to discrete measurement data, we follow a standard approach~\cite{Natterer2001} and discretize~\eqref{eq: FBP formula with conv} using the composite trapezoidal rule, analogous to~\cite{Beckmann2020, Beckmann2021}, resulting in 
\begin{equation}
\label{eq:discretized_fbp}
f_{L,D}^\varepsilon(x,y,\omega) = \frac{1}{2}\, \Back_D\left(\Fourier^{-1}A_L\ast_D g_D^\varepsilon(\cdot,\cdot,\omega)\right)(x,y)
\quad \text{for } (x,y)\in\R^2, ~ \omega\in\Omega
\end{equation} 
with 
\begin{equation*}
\Back_D\left(\Fourier^{-1}A_L\ast_D g_D^\varepsilon(\cdot,\cdot,\omega)\right)(x,y) = \frac{1}{N_\varphi} \sum\limits_{j=0}^{N_\varphi-1} (\Fourier^{-1}A_L\ast_D g_D^\varepsilon(\cdot,\cdot,\omega))(x\cos(\varphi_j)+y\sin(\varphi_j),\varphi_j)
\end{equation*}
for $(x,y)\in\R^2$ and
\begin{equation*}
(\Fourier^{-1}A_L\ast_D g_D^\varepsilon(\cdot,\cdot,\omega))(s,\varphi_j) = h\,\sum\limits_{i=-M}^M \Fourier^{-1}A_L(s-s_i)\,g_D^\varepsilon(i,j,\omega)
\end{equation*}
for $s \in \R$ and $0 \leq j \leq N_\varphi-1$.
Note that $f_{L,D}^\varepsilon$ is well-defined with $f_{L,D}^\varepsilon(\cdot,\cdot,\omega)\in\L^2_\loc(\R^2)$ for all $\omega\in\Omega$ if the filter function $A_L\in\L^2(\R)$ has finite bandwidth $L>0$.

Moreover, we need to discretize the optimized filter function derived in the last section. For this, we again use the composite trapezoidal rule, which leads to
\begin{equation*}
\int_{0}^\pi \lvert \Fourier(\Radon f)(\sigma,\varphi) \rvert^2 \: \d \varphi \approx \frac{\pi}{N_\varphi} \sum_{j=0}^{N_\varphi-1}  \vert \Fourier_D (\Radon f) (\sigma,j)\vert^2  \quad \text{for } \sigma\in \R
\end{equation*}
with 
\begin{equation*}
\Fourier_D (\Radon f) (\sigma,j) = h \sum_{i=-M}^M \Radon f(s_i,\varphi_j) \, \e^{-\i s_i \sigma} \quad \text{for } \sigma\in \R,\, 0\leq j\leq N_\varphi-1, 
\end{equation*}
and 
\begin{equation*}
\Fourier h_a(\sigma,0,\omega) \approx \Fourier_D h_a (\sigma,0,\omega) = h \sum_{i=-M}^M h_a(s_i,0,\omega) \, \e^{-\i s_i \sigma} \quad \text{for } \sigma\in \R,\,\omega\in \Omega
\end{equation*}
with 
\begin{align*}
\Ex \left( \lvert \Fourier_D h_a(\sigma,0,\cdot) \rvert^2  \right) &= h^2 \sum_{i=-M}^M \sum_{j=-M}^M \Ex(h_a(s_i,0,\cdot)\, h_a(s_j,0,\cdot)) \, \e^{-\i (s_i-s_j) \,\sigma} \\
&= h^2 \sum_{i=-M}^M \delta_a(0,0) = h^2\, (2M+1)\, \delta_a(0,0)
\quad \text{for all } \sigma\in \R,
\end{align*}
where we used \eqref{eq: discrete white noise uncorrelated}.
Replacing the continuous expressions in Theorem~\ref{theo:optimized_filter_continuous} with the discretized expressions yields an optimized filter for discrete measurements in parallel beam geometry.

\begin{definition}[Optimized filter function with compact support for discrete measurements]\label{def:optimized filter discrete data}
For a real-valued, bounded function $f:\R^2\to\R$ with $\supp(f) \subseteq B_R(0)$ for $R > 0$ and discrete measurements $g_D^{\varepsilon}:\{ -M,\dots , M \}\times \{ 0,\dots , N_\varphi-1 \} \times \Omega\to\R$ for fixed $M,\,N_\varphi\in\N$, given by 
\begin{equation*}
g_D^{\varepsilon}(i,j,\omega) = \Radon f(s_i,\varphi_j) + \xi^{\varepsilon_a}_{i,j}(\omega)
\end{equation*} 
with $s_i = i h$ for $h>0$ so that $Mh \geq R$, $\varphi_j = j\, \nicefrac{\pi}{N_\varphi}$ and discrete Gaussian white noise $\xi^{\varepsilon_a}$, the {\em optimized filter} $A_{L,D}^\ast:\R\to\R$ of bandwidth $L>0$ is defined by 
\begin{align*}
A_{L,D}^\ast(\sigma) = \begin{dcases} 
\frac{\vert\sigma\vert\, \frac{1}{N_\varphi} \sum_{j=0}^{N_\varphi-1} \vert\Fourier_D (\Radon f) (\sigma,j)\vert^2}{\frac{1}{N_\varphi} \sum_{j=0}^{N_\varphi-1} \vert\Fourier_D (\Radon f) (\sigma,j)\vert^2 + h^2\,\varepsilon_a^2\, (2M+1) } &\text{for } \sigma \in [-L,L] \\
0 &\text{for } \sigma \not\in [-L,L].
\end{dcases}
\end{align*}
\end{definition}

Note that the optimized filter for discrete measurements depends on the target function $f$, the discretization parameters $M$, $N_\varphi$ and $h$ as well as on the variance $\varepsilon_a^2$, which imposes practical limitations on the applicability of $A_{L,D}^\ast$ because the objective function and variance are rarely known.
We address this issue in the next chapter.
As in the continuous setting, the optimized filter $A_{L,D}^\ast$ for noise-free discrete measurements coincides with the Ram-Lak filter.

In contrast to the optimized filter function for continuous measurement data derived in Theorem~\ref{theo:optimized_filter_continuous}, $A_{L,D}^\ast$ does not minimize 
\begin{equation}\label{eq:objective function discrete setting}
\Ex \left( \lVert f_{L,D}^\varepsilon - f \rVert_{\L^2(K)}^2  \right)
\end{equation}
with compact $K\subset \R^2$, which is illustrated in the following remark.

\begin{remark}
The reconstruction error can be bounded above by
\begin{align*}
\lVert f_{L,D}^\varepsilon(\cdot,\cdot,\omega) - f \rVert_{\L^2(K)}^2 
\leq& \left\Vert  f_{L,D}^\varepsilon (\cdot,\cdot,\omega) - \nicefrac{1}{2}\, \Back_D\left( \Fourier^{-1} \left(\vert \cdot \vert\, \ind_{[-L,L]}(\cdot) \, \Fourier_D(\Radon f)\right)\right)\right\Vert^2_{\L^2(K)} \\
&+ \left\Vert\nicefrac{1}{2}\, \Back_D\left( \Fourier^{-1} \left(\vert \cdot \vert\, \ind_{[-L,L]}(\cdot) \, \Fourier_D(\Radon f)\right)\right) - f \right\Vert_{\L^2(K)}^2
\end{align*}
for fixed $\omega\in\Omega$ and compact $K\subset \R^2$.
The first summand can be interpreted as a combination of the error induced by the filter function $A_L$ and the measurement process.
In contrast, the second summand can be considered as an approximation error, which is independent of the chosen filter function.
The first summand can be bounded above by
\begin{align*}
&\left\Vert  f_{L,D}^\varepsilon (\cdot,\cdot,\omega) - \nicefrac{1}{2}\, \Back_D\left( \Fourier^{-1} \left(\vert \cdot \vert\, \ind_{[-L,L]}(\cdot) \, \Fourier_D(\Radon f)\right)\right)\right\Vert^2_{\L^2(K)}  \\
&\qquad \leq \diam(K)^2\frac{2 L N_\varphi}{16\pi^2 N_\varphi^2} \sum_{j=0}^{N_\varphi-1} \int_{-L}^L \vert \sigma \vert^2 \, \left\vert W(\nicefrac{\sigma}{L})\,\Fourier_D g_D^\varepsilon (\sigma,j,\omega) - \Fourier_D(\Radon f)(\sigma,\varphi_j) \right\vert^2 \: \d\sigma.
\end{align*}
Now, the optimized filter function for discrete measurement data defined in Definition~\ref{def:optimized filter discrete data} minimizes this upper bound, as can be shown by arguments similar to those in Section~\ref{sec:continuous filter}. To derive a filter function that minimizes \eqref{eq:objective function discrete setting}, one needs a different optimization approach that is more involved since the relations used in the continuous setting do not carry over to the discrete case.
\end{remark}

\section{Numerical Experiments}\label{sec:Numerik}

\begin{figure}[t]
\centering
\subfigure[Ram-Lak]{\includegraphics[width=.325\textwidth]{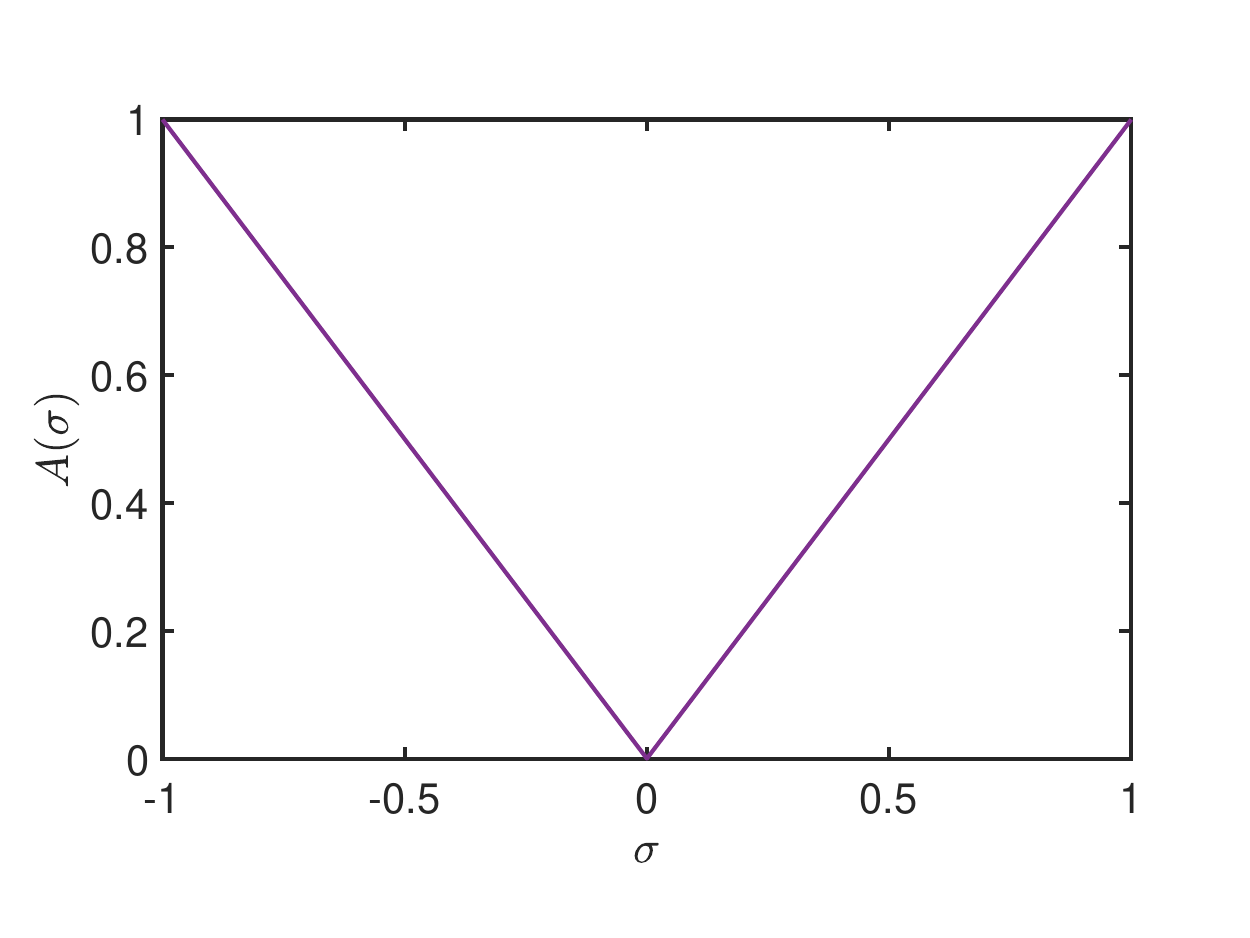}}
\hfill
\subfigure[Shepp-Logan]{\includegraphics[width=.325\textwidth]{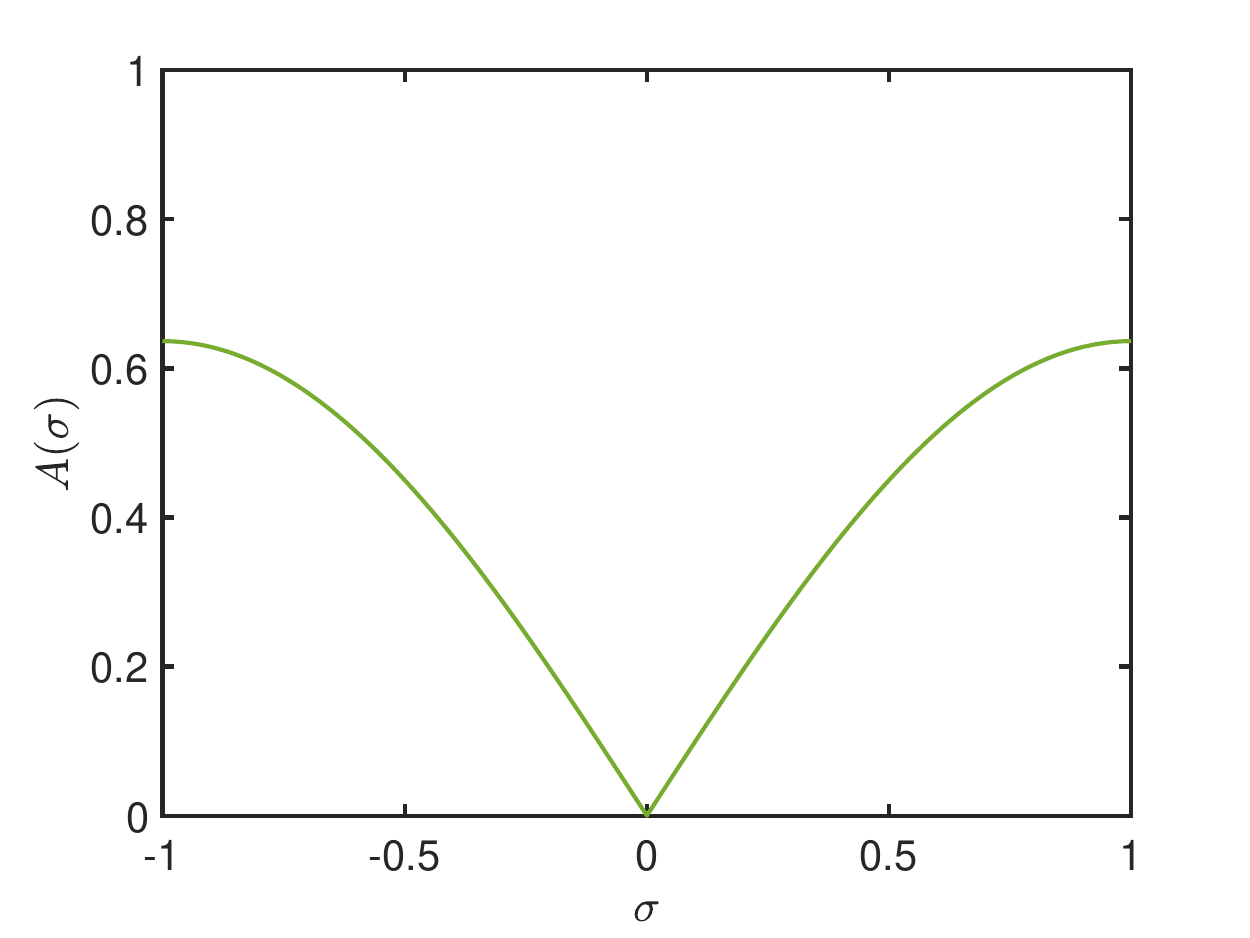}}
\hfill
\subfigure[Cosine]{\includegraphics[width=.325\textwidth]{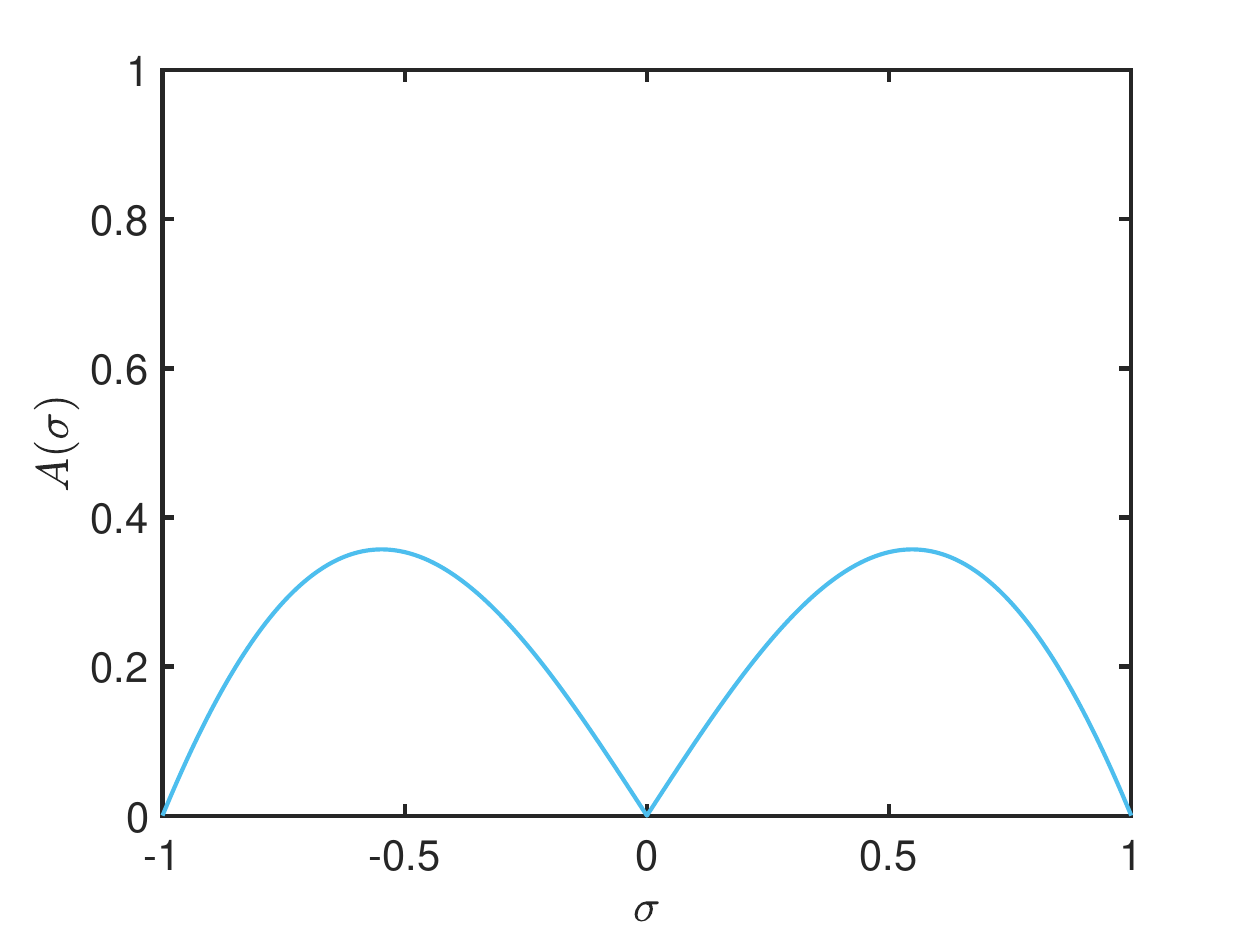}}
\hfill
\subfigure[Hamming ($\beta = 0.55$)]{\includegraphics[width=.325\textwidth]{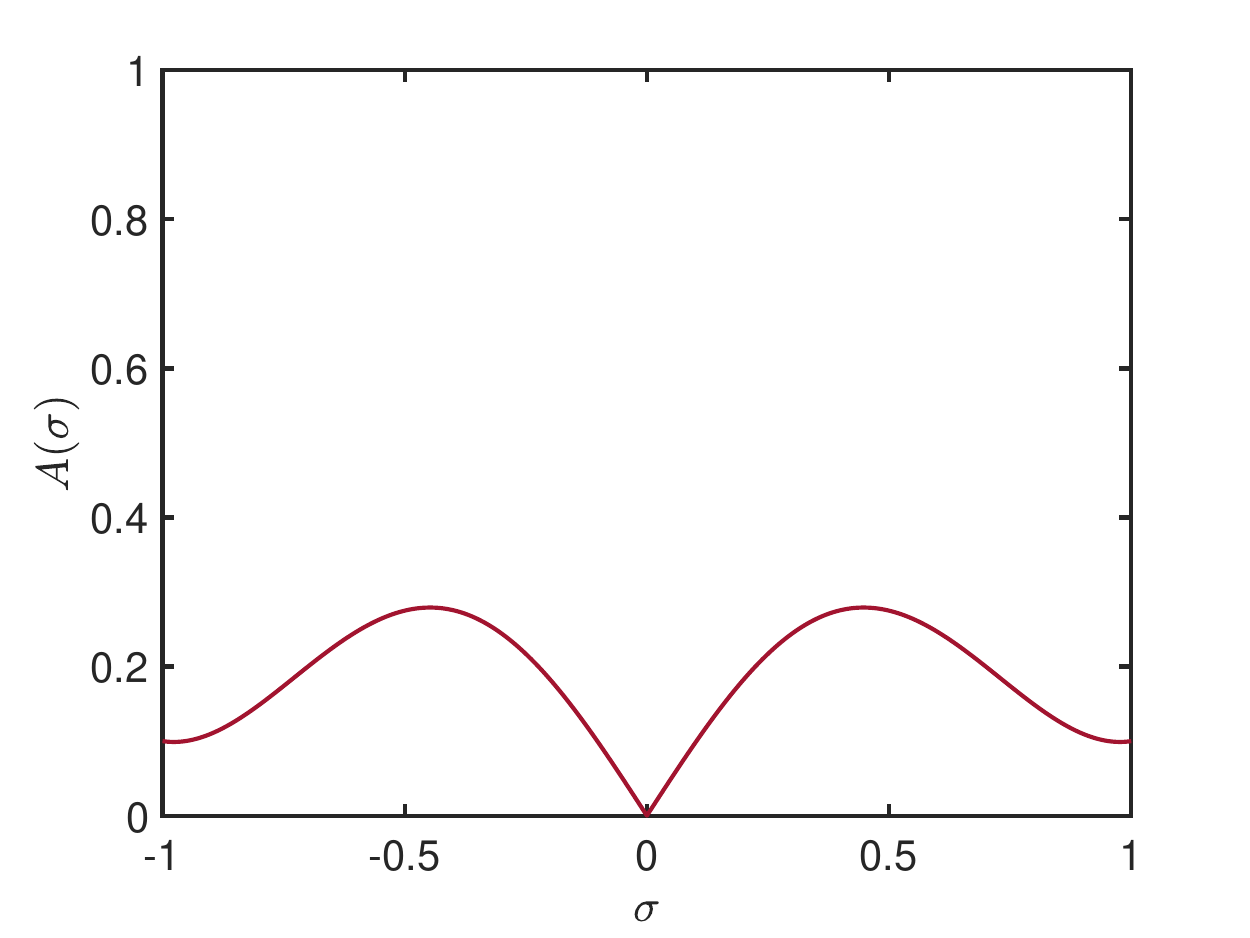}}
\hfill
\subfigure[Hamming ($\beta = 0.7$)]{\includegraphics[width=.325\textwidth]{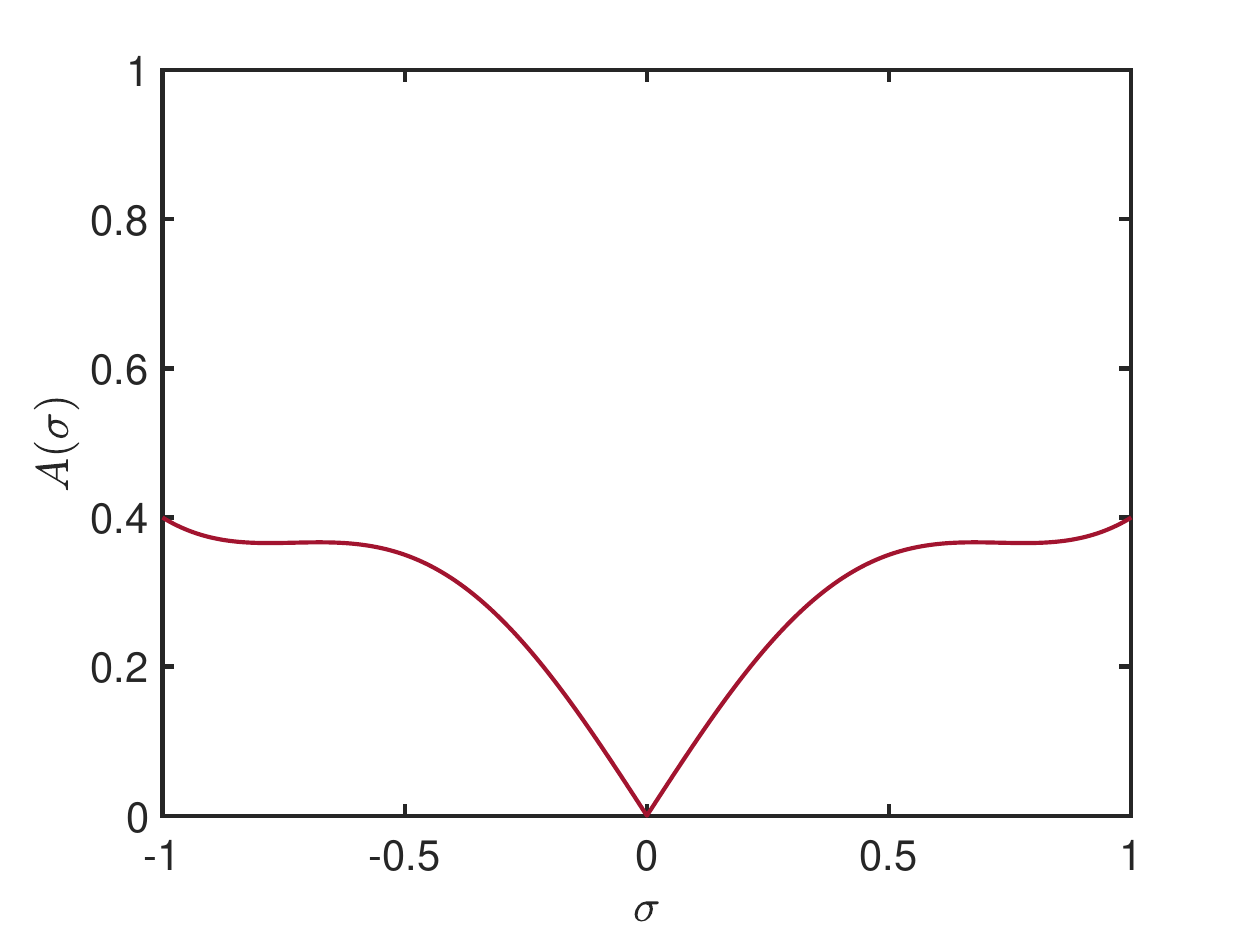}}
\hfill
\subfigure[Hamming ($\beta = 0.85$)]{\includegraphics[width=.325\textwidth]{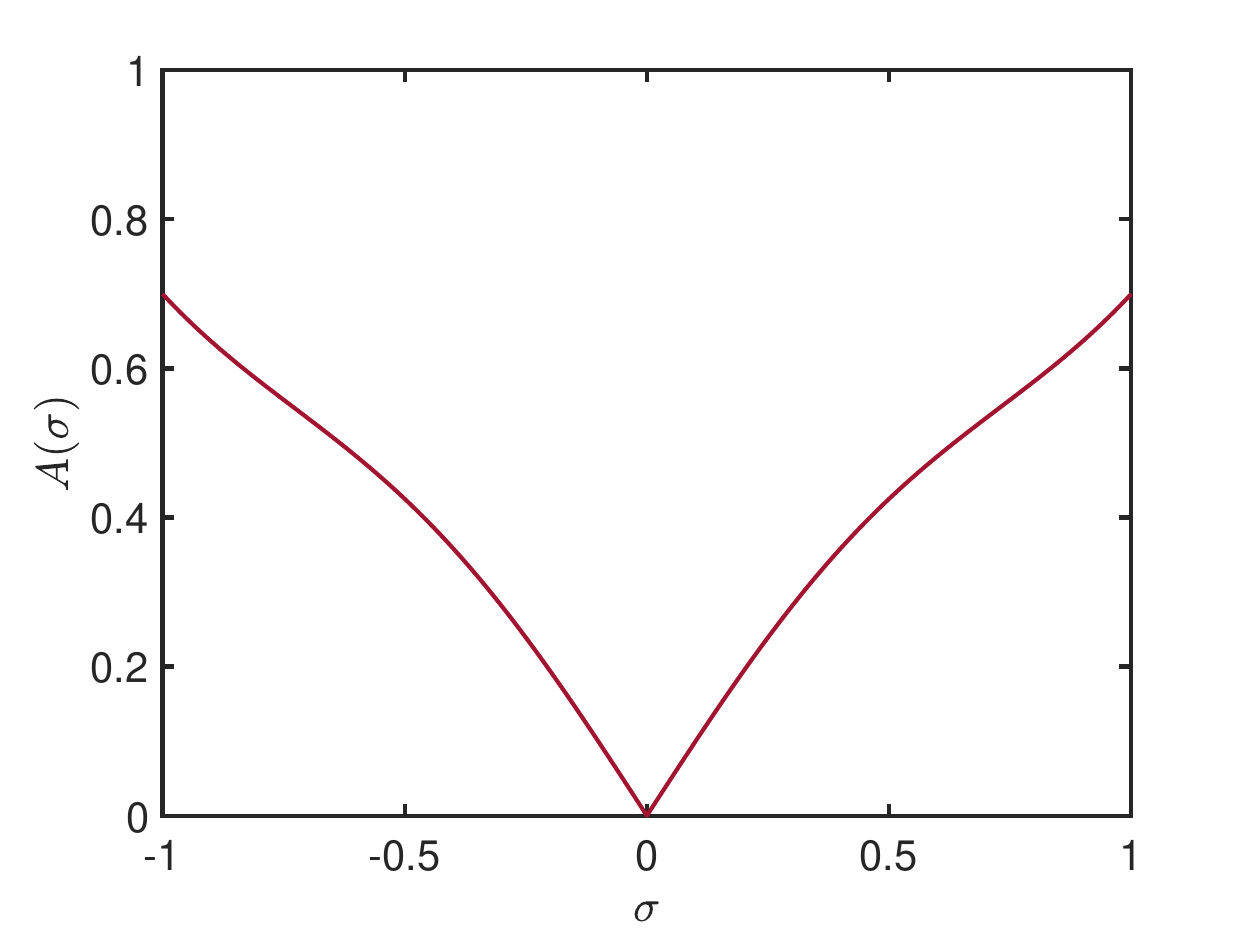}}
\caption{Selection of classical low-pass filters used in our numerical experiments.}
\label{fig:classical_filter}
\end{figure}

In this section, we finally present selected numerical experiments to illustrate and validate our theoretical findings.
To this end, recall that we aim at recovering the target function $f \in \L^2(\R^2)$ with $\supp(f) \subseteq B_R(0)$ from noisy measurements $g_D^\varepsilon: \{-M,\ldots,M\} \times \{0,\ldots,N_\varphi-1\} \times \Omega \to \R$ given by 
\begin{equation*}
g_D^\varepsilon(i,j,\omega) = \Radon f(s_i,\varphi_j) + \xi^{\varepsilon_a}_{i,j}(\omega)
\quad \text{for } -M \leq i \leq M,\, 0 \leq j \leq N_\varphi-1,\, \omega \in \Omega 
\end{equation*}
with $s_i = i h$ for $h\geq \nicefrac{R}{M}$, $\varphi_j = j \, \nicefrac{\pi}{N_\varphi}$ and noise $\xi^{\varepsilon_a}_{i,j}: \Omega \to \R$ with expectation zero and variance~$\varepsilon_a^2$, which is assumed to be known or can be estimated.
For simplicity, by rescaling the noise level we set $\varepsilon_a = \varepsilon$. Moreover, we from now on suppress the dependence on $\omega \in \Omega$ and simply write
\begin{equation*}
g_D^\varepsilon(i,j) = \Radon f(s_i,\varphi_j) + \xi^\varepsilon_{i,j}
\quad \text{for } -M \leq i \leq M,\, 0 \leq j \leq N_\varphi-1.
\end{equation*}

We wish to apply the discretized approximate FBP formula~\eqref{eq:discretized_fbp}, i.e.,
\begin{equation*}
f_{L,D}^\varepsilon = \frac{1}{2}\, \Back_D\bigl(\Fourier^{-1}A_L\ast_D g_D^\varepsilon\bigr).
\end{equation*}
The evaluation of $f_{L,D}^\varepsilon$, however, requires the computation of the values
\begin{equation*}
(\Fourier^{-1} A_L *_D g_D^\varepsilon)(x\cos(\varphi_j)+y\sin(\varphi_j),\varphi_j)
\quad \forall \, 0 \leq j \leq N-1
\end{equation*}
for each reconstruction point $(x,y) \in \R^2$, which is computationally expensive.
To reduce the computational costs, a standard approach is to evaluate the function
\begin{equation*}
s \mapsto (\Fourier^{-1} A_L *_D g_D^\varepsilon)(s,\varphi_j)
\end{equation*}
only at the points $s = s_l$, $l \in I$, for a sufficiently large index set $I \subset \Z$ and interpolate its function value for $s = x\cos(\varphi_j)+y\sin(\varphi_j)$ using an interpolation method $\Int$.
This leads us to the {\em discrete FBP reconstruction formula}
\begin{equation}
\label{eq:discrete_FBP_formula}
f_\FBP = \frac{1}{2} \Back_D \bigl(\Int[\Fourier^{-1} A_L *_D g_D^\varepsilon]\bigr),
\end{equation}
where linear or cubic spline interpolation is typically used depending on the regularity of~$f$.
Note, however, that the errors incurred by the utilized interpolation method are not taken into account in our theoretical derivation of optimized filter functions.
Moreover, we choose the bandwidth $L > 0$ of the filter $A_L$ and the discretization parameters $h > 0$, $M \in \N$ optimally depending on the number of angles $N_\varphi \in \N$ according to~\cite{Natterer2001} via
\begin{equation*}
M = \left\lfloor\frac{N_\varphi}{\pi}\right\rfloor,
\quad
L =\frac{\pi M}{R},
\quad
h = \frac{\pi}{L},
\end{equation*}
where we assume that $\supp(f)\subseteq B_R(0)$ with $R\in\N$.

\begin{table}[t]
\centering
\caption{Window functions of classical low-pass filters, where $W(\sigma) = 0$ for all $|\sigma| > 1$ in all cases.}
\begin{tabular}{l|c|c}
Name & $W(\sigma)$ for $|\sigma|\leq 1$ & Parameter\\
\hline
Ram-Lak & $1$ & - \\
Shepp-Logan & $\sinc(\nicefrac{\pi \sigma}{2})$ & - \\
Cosine & $\cos(\nicefrac{\pi \sigma}{2})$ & - \\
Hamming & $\beta + (1-\beta) \cos(\pi \sigma)$ & $\beta \in [\nicefrac{1}{2},1]$ \\
\end{tabular}
\label{tab:conventional_filters}
\end{table}

The application of our optimized filter for discrete measurements from Definition~\ref{def:optimized filter discrete data},
\begin{equation*}
A_{L,D}^\ast(\sigma) = \begin{dcases} 
\frac{\vert\sigma\vert\, \frac{1}{N_\varphi} \sum_{j=0}^{N_\varphi-1} \vert\Fourier_D (\Radon f) (\sigma,j)\vert^2}{\frac{1}{N_\varphi} \sum_{j=0}^{N_\varphi-1} \vert\Fourier_D (\Radon f) (\sigma,j)\vert^2 + h^2\,\varepsilon^2\, (2M+1) } &\text{for } \sigma \in [-L,L] \\
0 &\text{for } \sigma \not\in [-L,L],
\end{dcases}
\end{equation*}
requires the exact knowledge of the noiseless Radon samples $\Radon f(s_i,\varphi_j)$ for all $i = -M,\ldots,M$ and $j = 0,\ldots,N_\varphi-1$.
As this information is typically not available in applications, we investigate two options to overcome this bottleneck.
In the first approach, we simply replace the Radon samples $\Radon f(s_i,\varphi_j)$ by the available measurements $g_D^\varepsilon(i,j)$ leading to the filter function
\begin{equation*}
A_{L,D}^\varepsilon(\sigma) = \begin{dcases} 
\frac{\vert\sigma\vert\, \frac{1}{N_\varphi} \sum_{j=0}^{N_\varphi-1} \vert\Fourier_D g_D^\varepsilon (\sigma,j)\vert^2}{\frac{1}{N_\varphi} \sum_{j=0}^{N_\varphi-1} \vert\Fourier_D g_D^\varepsilon (\sigma,j)\vert^2 + h^2\,\varepsilon^2\, (2M+1) } &\text{for } \sigma \in [-L,L] \\
0 &\text{for } \sigma \not\in [-L,L].
\end{dcases}
\end{equation*}
In the second approach, we first convolve the given measurements $g_D^\varepsilon$ with a Wiener filter, which minimizes the expected squared error between convolved data $\hat{g}_D^\varepsilon$ and true data~$\Radon f$.
Thereon, we replace the Radon samples $\Radon f(s_i,\varphi_j)$ by the denoised measurements $\hat{g}_D^\varepsilon(i,j)$ yielding the filter function
\begin{equation*}
\hat{A}_{L,D}^\varepsilon(\sigma) = \begin{dcases} 
\frac{\vert\sigma\vert\, \frac{1}{N_\varphi} \sum_{j=0}^{N_\varphi-1} \vert\Fourier_D (\hat{g}_D^\varepsilon) (\sigma,j)\vert^2}{\frac{1}{N_\varphi} \sum_{j=0}^{N_\varphi-1} \vert\Fourier_D (\hat{g}_D^\varepsilon) (\sigma,j)\vert^2 + h^2\,\varepsilon^2\, (2M+1) } &\text{for } \sigma \in [-L,L] \\
0 &\text{for } \sigma \not\in [-L,L],
\end{dcases}
\end{equation*}
where the kernel size of the Wiener filter acts as hyperparameter.
Let us stress that alternative denoising techniques can be applied like, e.g., data-driven approaches based on neural networks, which is beyond the scope of this work and asks for more in-depth future research.

In our numerical experiments we compare the reconstruction performance of our optimized filter function with classical low-pass filters, which are illustrated in Figure~\ref{fig:classical_filter} and whose window functions are listed in Table~\ref{tab:conventional_filters}, as well as the recently proposed filters from~\cite{Kabri2024, Pelt2014, Pelt2015}.

\subsection{Shepp-Logan phantom}\label{sec:shepp_logan}

\begin{figure}[t]
\centering
\subfigure[Shepp-Logan phantom]{\includegraphics[width=.24\textwidth]{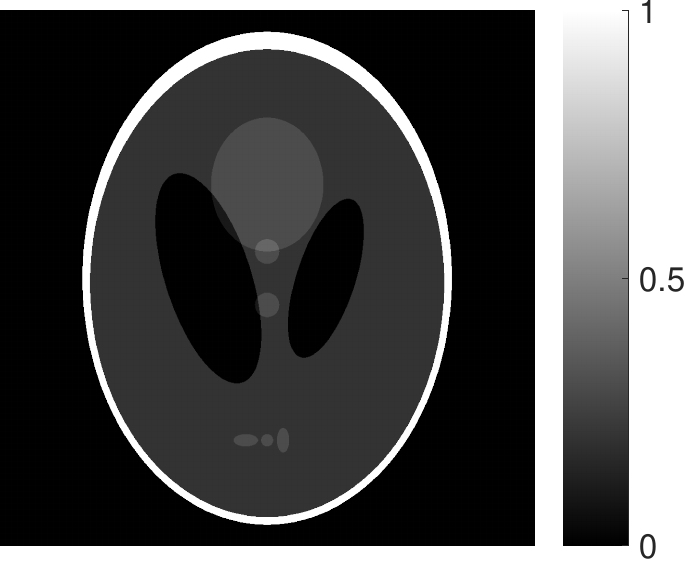}}
\hfil
\subfigure[Sinogram of (a)]{\includegraphics[width=.24\textwidth]{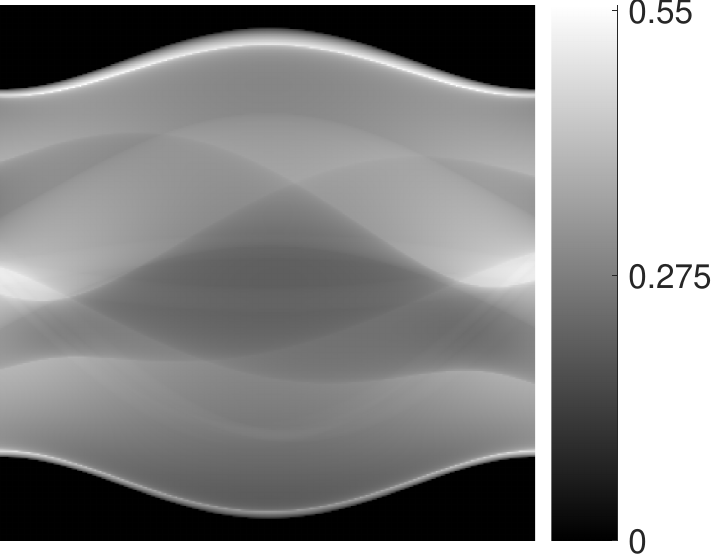}}
\hfil
\subfigure[Modified phantom]{\includegraphics[width=.24\textwidth]{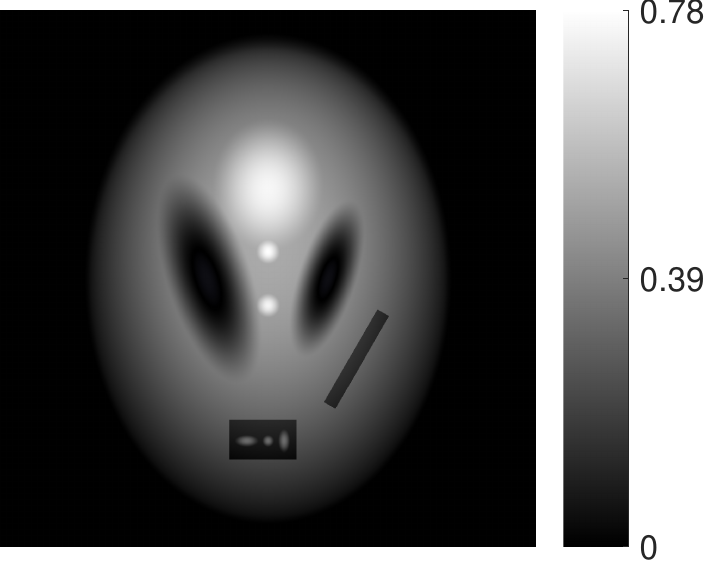}}
\hfil
\subfigure[Sinogram of (c)]{\includegraphics[width=.24\textwidth]{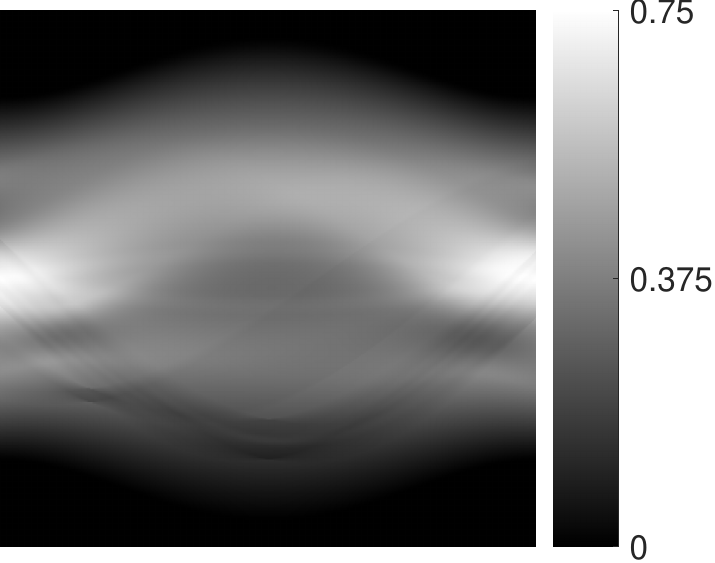}}
\caption{Phantoms used in our numerical experiments along with their sinograms (Radon data).}
\label{fig:phantoms}
\end{figure}

\begin{figure}[t]
\centering
\subfigure[$p_\mathrm{noise}=0.05$]{\includegraphics[width=.275\textwidth]{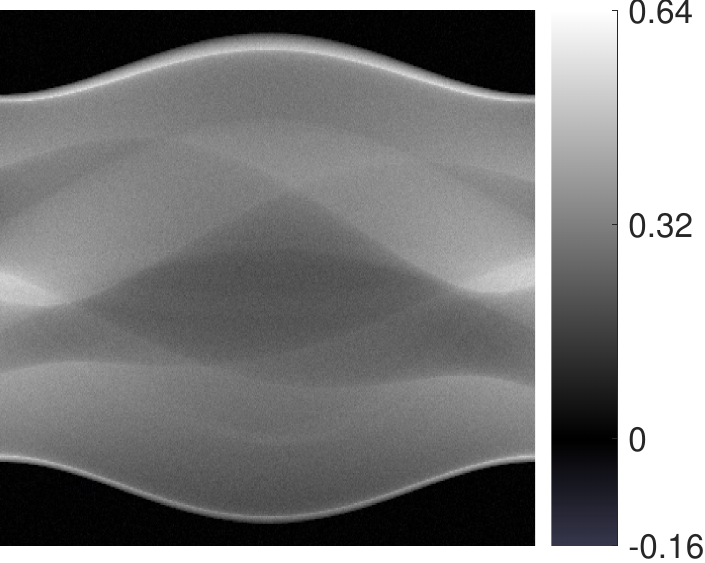}}
\hfil
\subfigure[$p_\mathrm{noise}=0.1$]{\includegraphics[width=.275\textwidth]{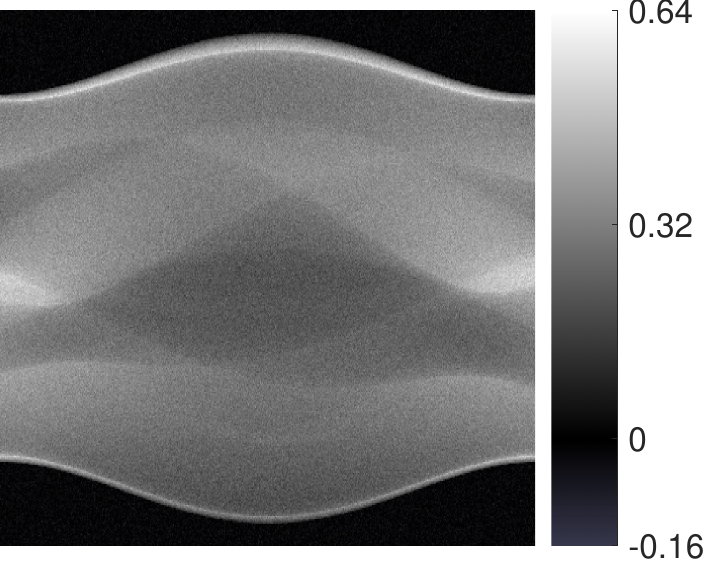}}
\hfil
\subfigure[$p_\mathrm{noise}=0.15$]{\includegraphics[width=.275\textwidth]{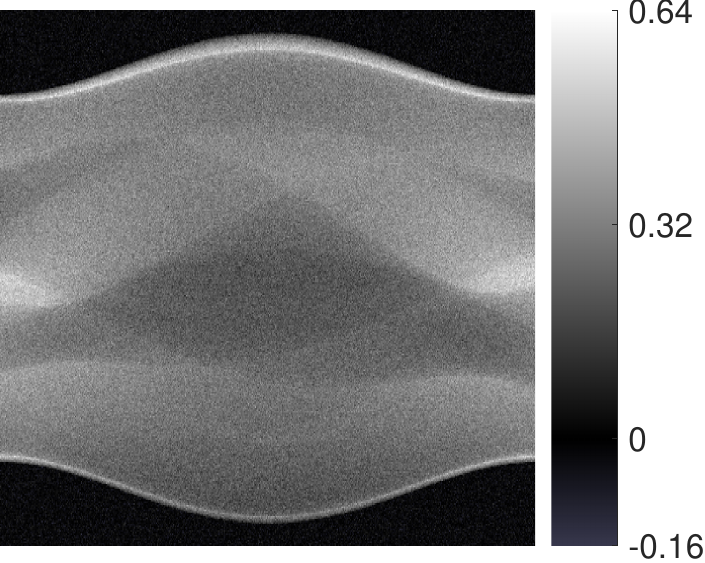}}
\caption{Noisy sinograms of the Shepp-Logan phantom with different noise levels.}
\label{fig:sinogram_shepp_logan_noise}
\end{figure}

In a first set of numerical simulations, we make use of the popular Shepp-Logan phantom, proposed in~\cite{Shepp1974} as a simplistic model of a human head section.
It consists of ten ellipses of different sizes, eccentricities and locations leading to a piecewise constant attenuation function $f_{\mathrm{SL}} \in \L^2(\R^2)$ with $\supp(f_{\mathrm{SL}}) \subseteq B_1(0)$, whose Radon transform can be computed analytically, cf.~Figure~\ref{fig:phantoms}~(a)-(b).
The noisy measurements $g_D^\varepsilon(i,j)$ are simulated by adding white Gaussian noise with variance $\varepsilon^2$ to the Radon samples $\Radon f_{\mathrm{SL}}(s_i,\varphi_j)$, where $\varepsilon = p_{\mathrm{noise}} \, m_{\Radon f_{\mathrm{SL}}}$ with $p_{\mathrm{noise}} \in \{0.05, \, 0.1, \, 0.15\}$ depends on the arithmetic mean
\begin{equation*}
m_{\Radon f_{\mathrm{SL}}} = \frac{1}{(2M+1) N_\varphi} \sum_{i=-M}^M \sum_{j=0}^{N_\varphi-1} \vert\Radon f_{\mathrm{SL}}(s_i,\varphi_j)\vert,
\end{equation*}
see Figure~\ref{fig:sinogram_shepp_logan_noise} for an illustration of the resulting noisy sinograms.

\begin{figure}[p]
\centering
\includegraphics[height=1cm]{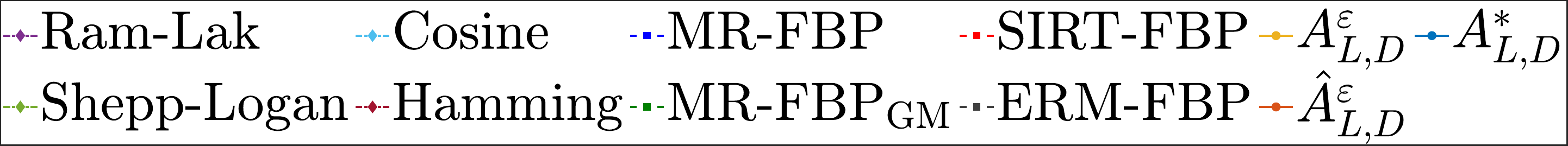}\\
\subfigure[$p_\mathrm{noise}=0.05$]{%
\includegraphics[width=.5\textwidth, viewport=50 25 1075 675]{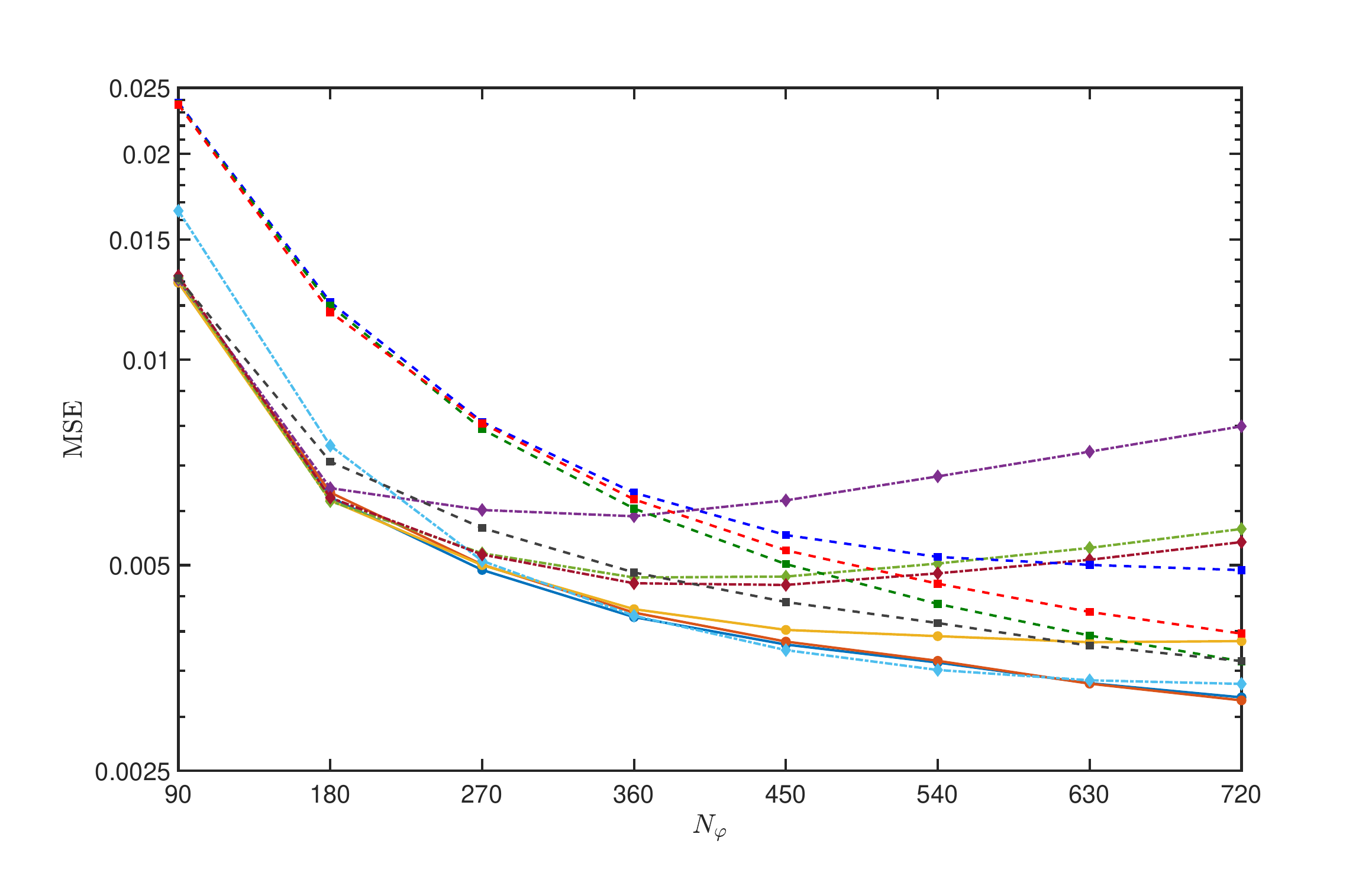}%
\includegraphics[width=.5\textwidth, viewport=50 25 1075 675]{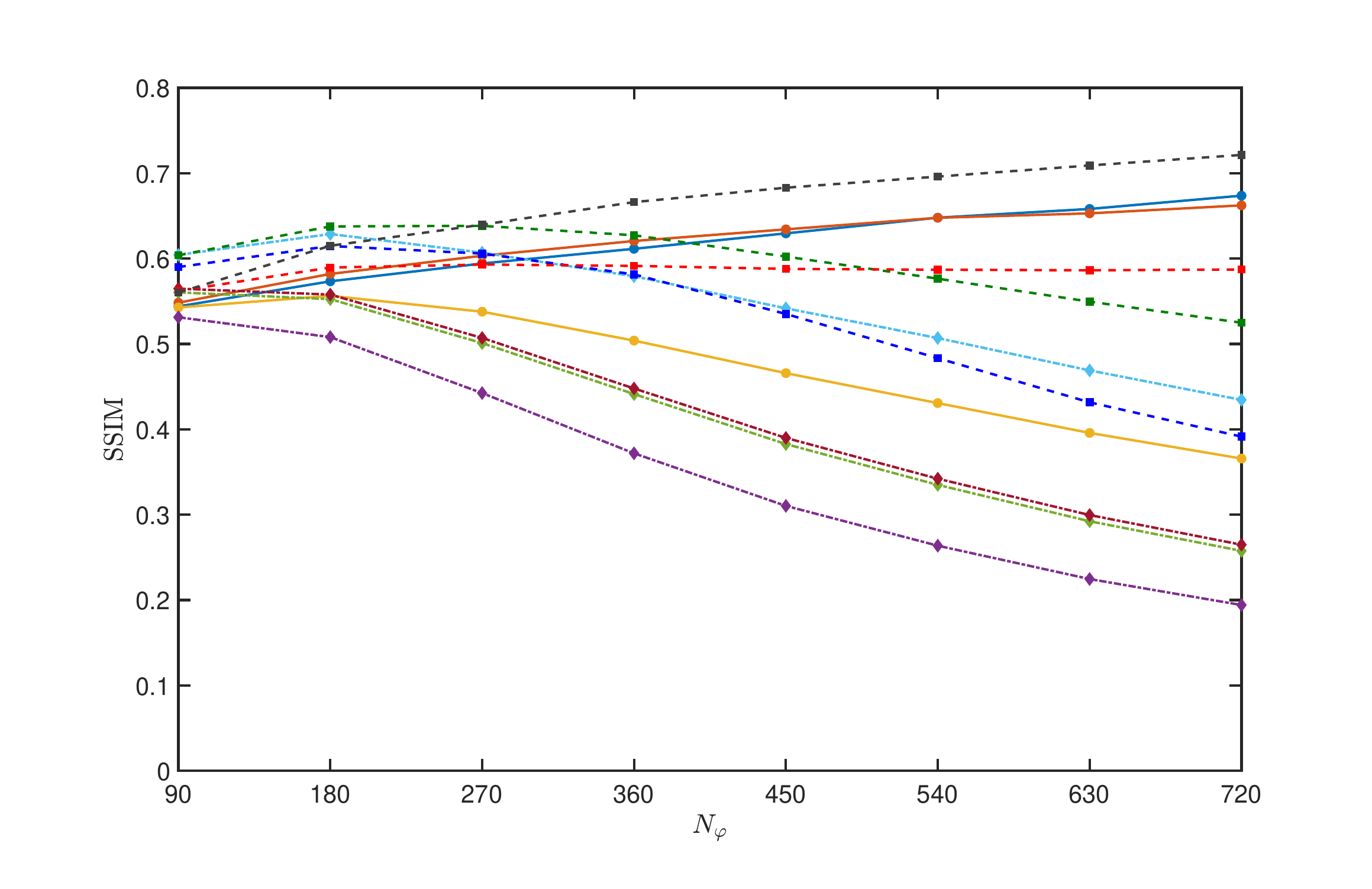}}
\hfill
\subfigure[$p_\mathrm{noise}=0.1$]{%
\includegraphics[width=.5\textwidth, viewport=50 25 1075 675]{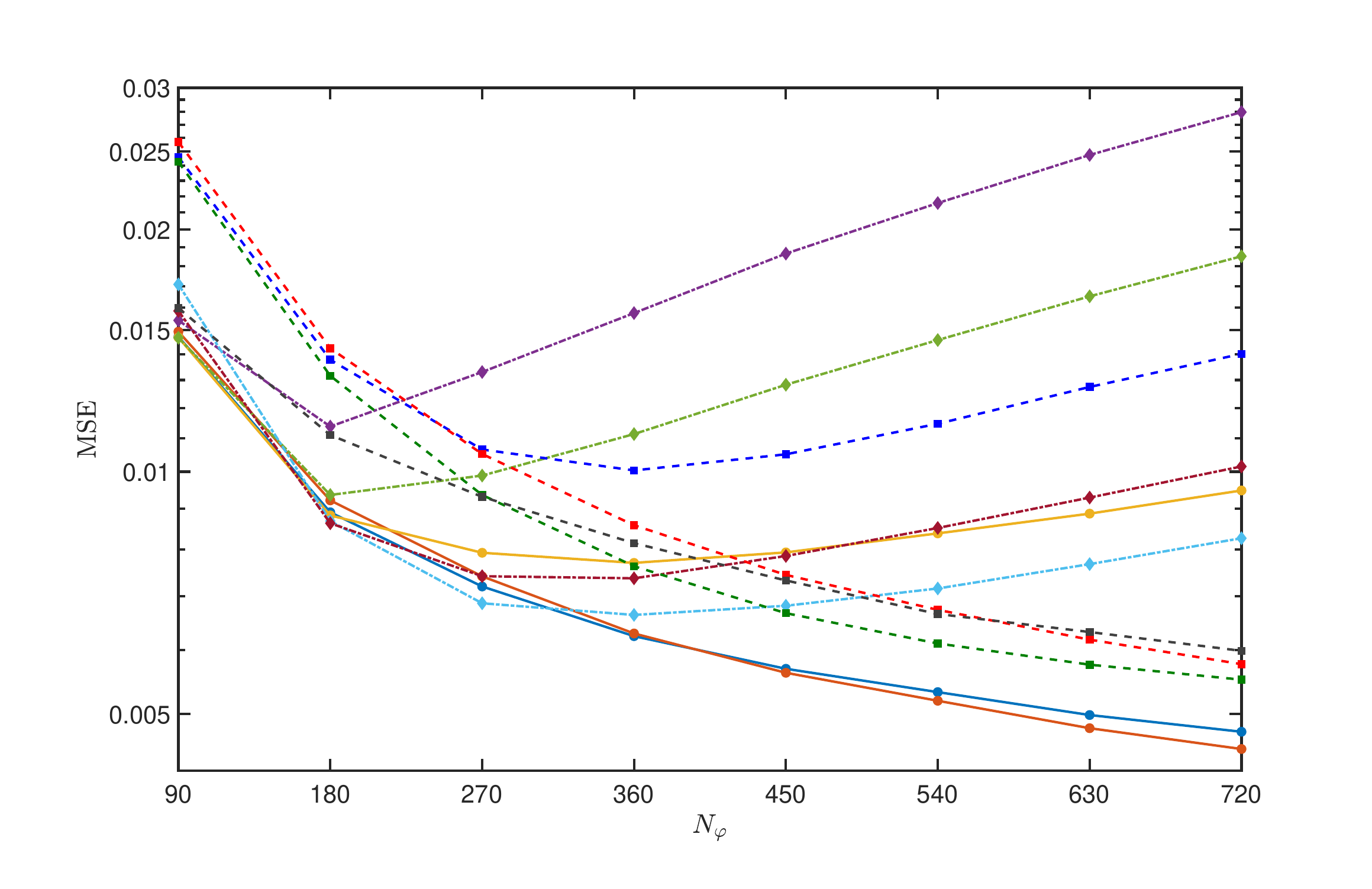}%
\includegraphics[width=.5\textwidth, viewport=50 25 1075 675]{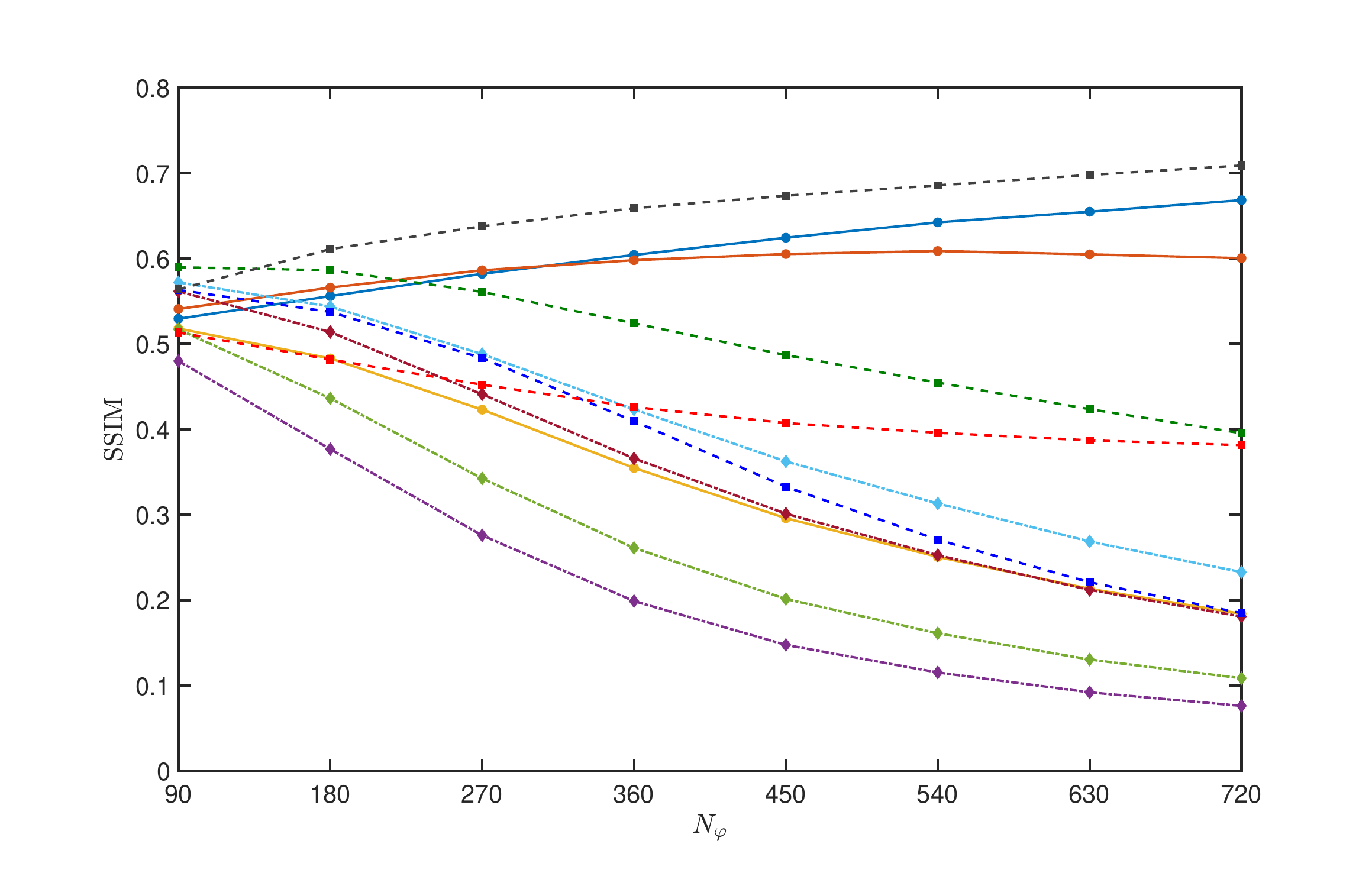}}
\hfill
\subfigure[$p_\mathrm{noise}=0.15$]{%
\includegraphics[width=.5\textwidth, viewport=50 25 1075 675]{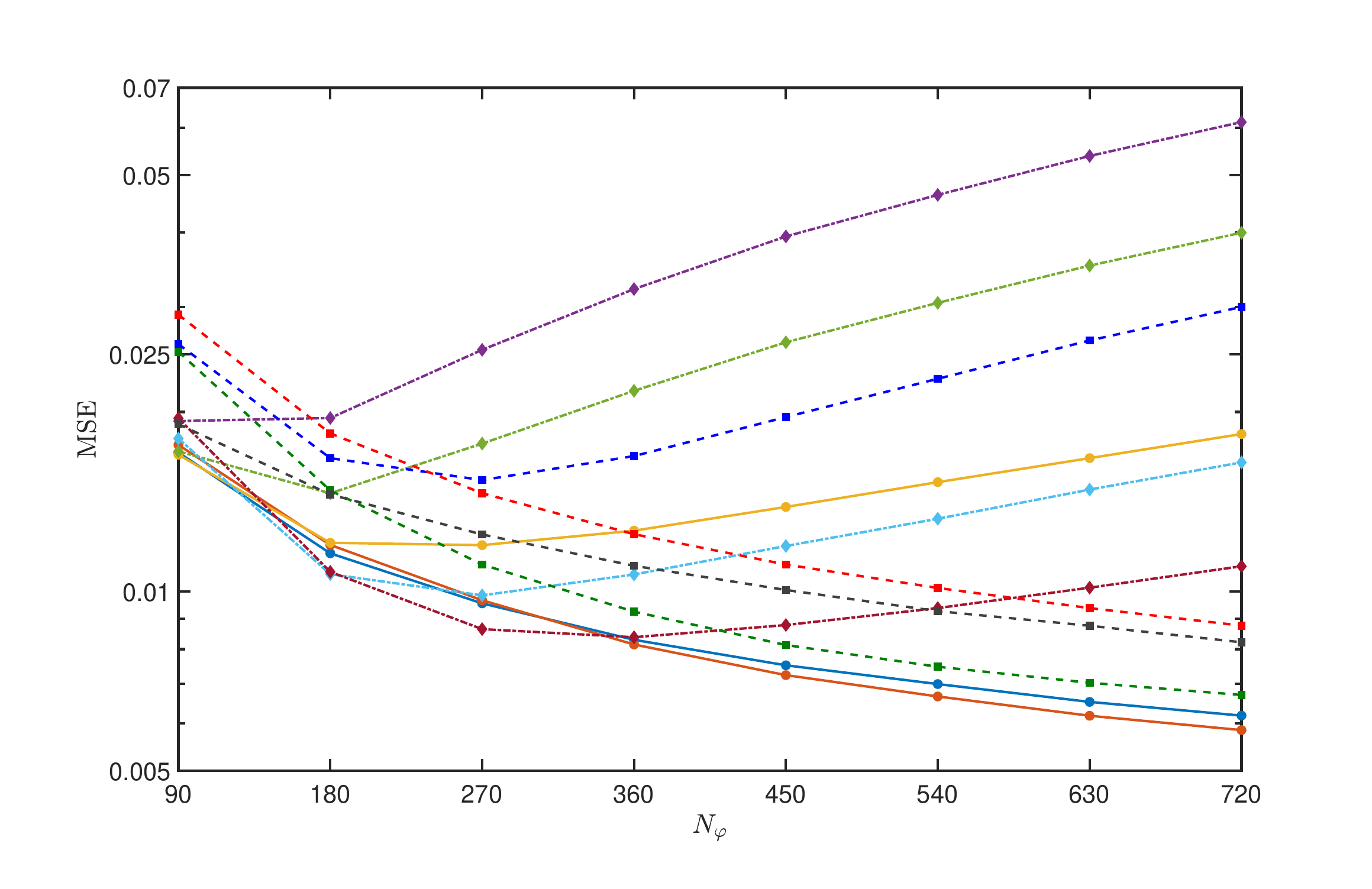}%
\includegraphics[width=.5\textwidth, viewport=50 25 1075 675]{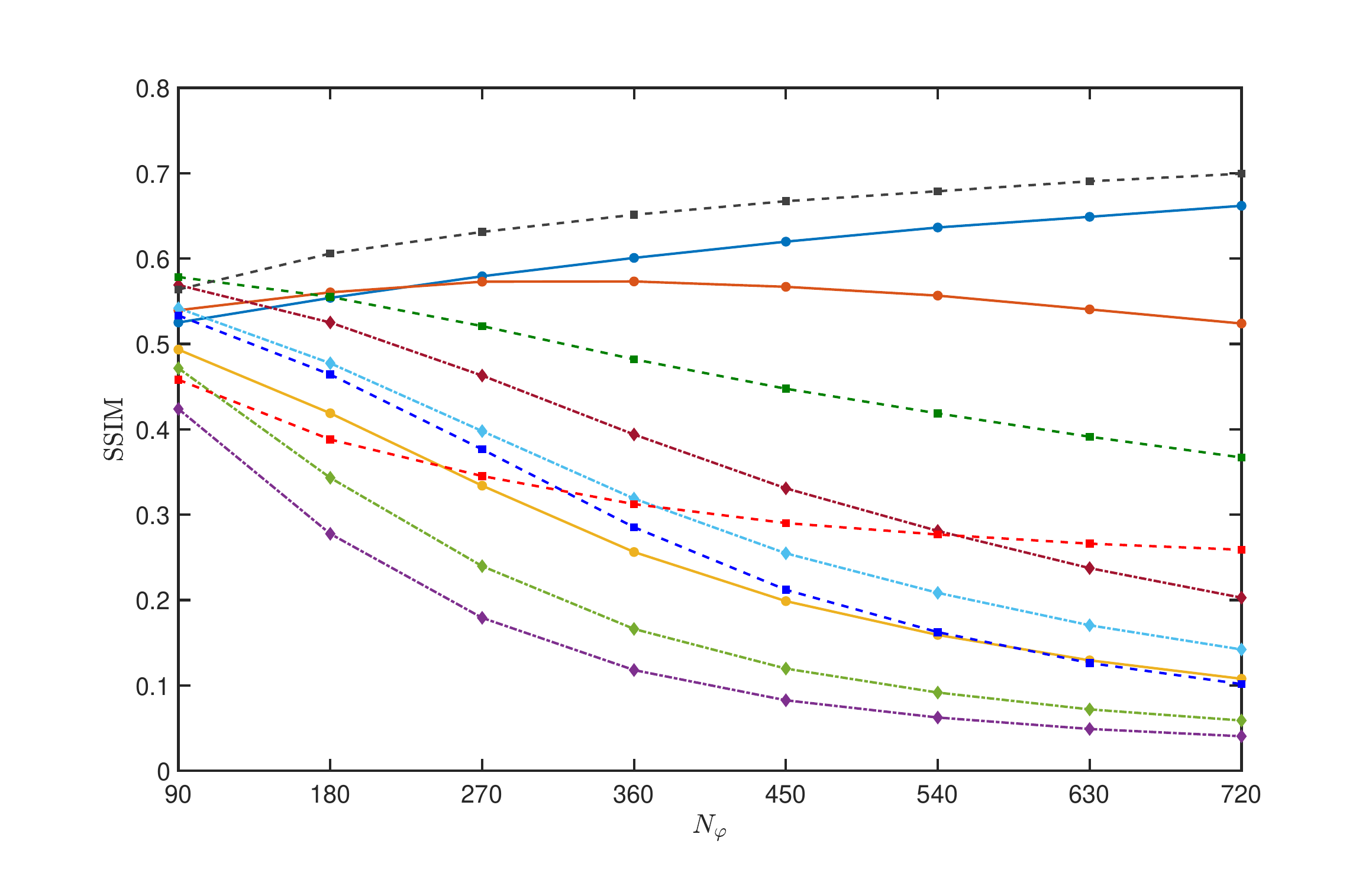}}
\caption{Plots of the MSE and SSIM of FBP reconstructions for the Shepp-Logan phantom.}
\label{fig:errors_shepp_logan_noise}
\end{figure}

\begin{figure}[p]
\centering
\subfigure[Ground truth]{\includegraphics[height=3.85cm]{SheppLogan_phantom.pdf}}
\hfill
\subfigure[Ram-Lak]{\includegraphics[height=3.85cm]{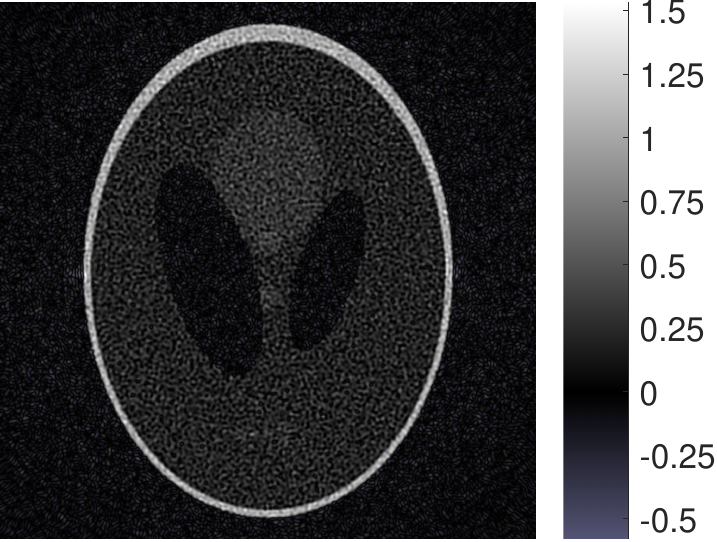}}
\hfill
\subfigure[Shepp-Logan]{\includegraphics[height=3.85cm]{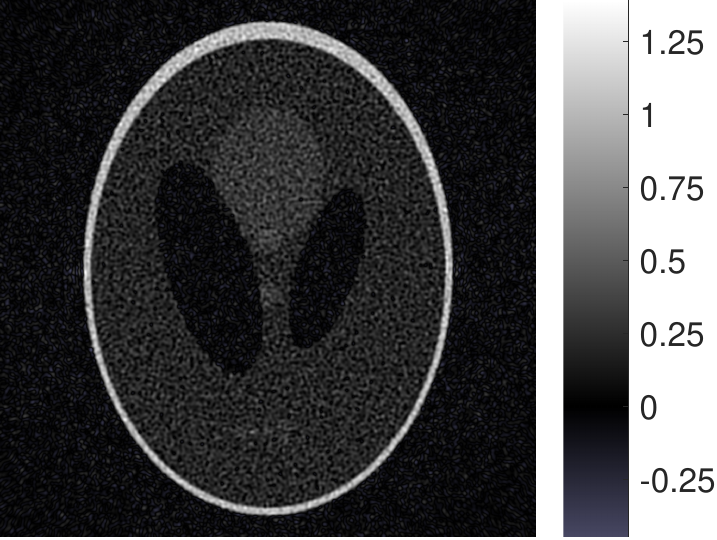}}
\hfill
\subfigure[Cosine]{\includegraphics[height=3.85cm]{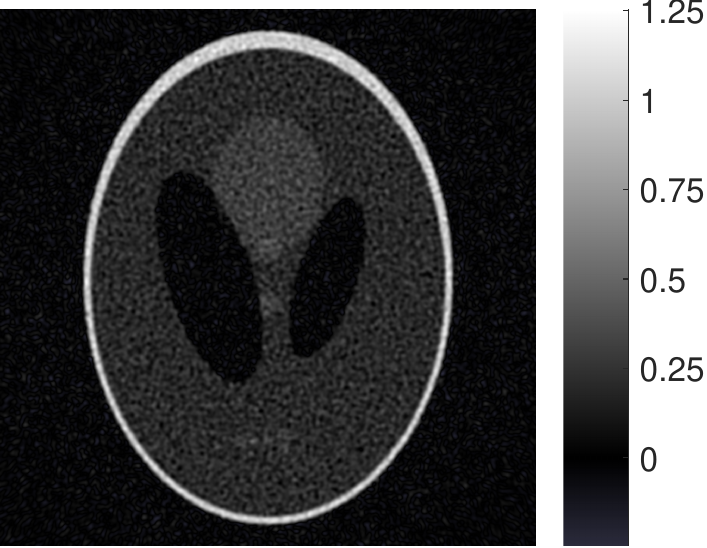}}
\hfill
\subfigure[Hamming]{\includegraphics[height=3.85cm]{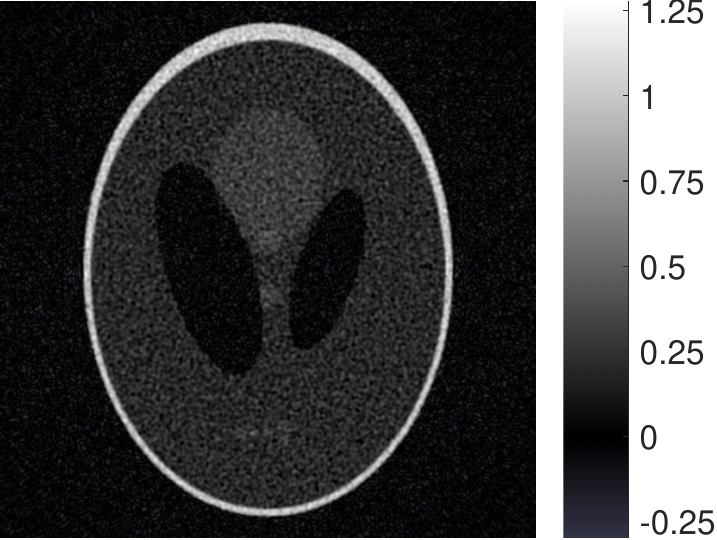}}
\hfill
\subfigure[MR-FBP]{\includegraphics[height=3.85cm]{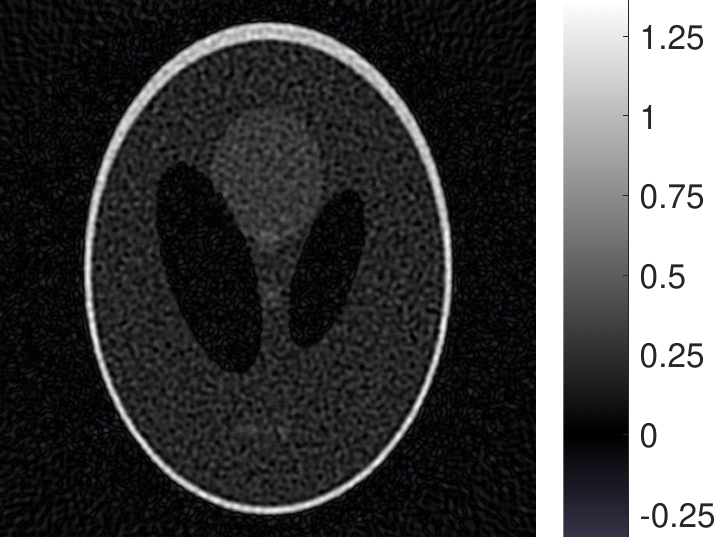}}
\hfill
\subfigure[MR-FBP$_\mathrm{GM}$]{\includegraphics[height=3.85cm]{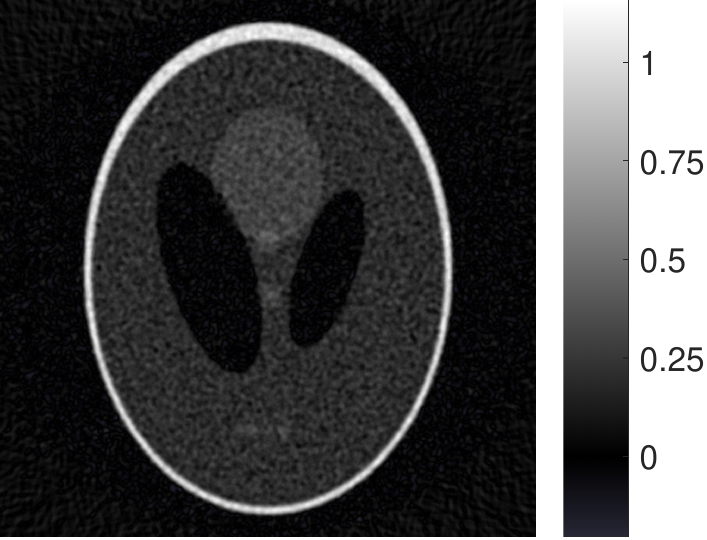}}
\hfill
\subfigure[SIRT-FBP]{\includegraphics[height=3.85cm]{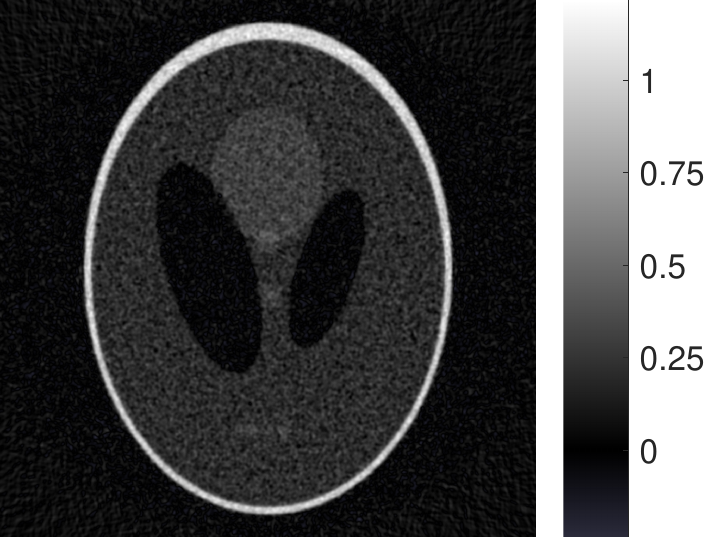}}
\hfill
\subfigure[ERM-FBP]{\includegraphics[height=3.985cm]{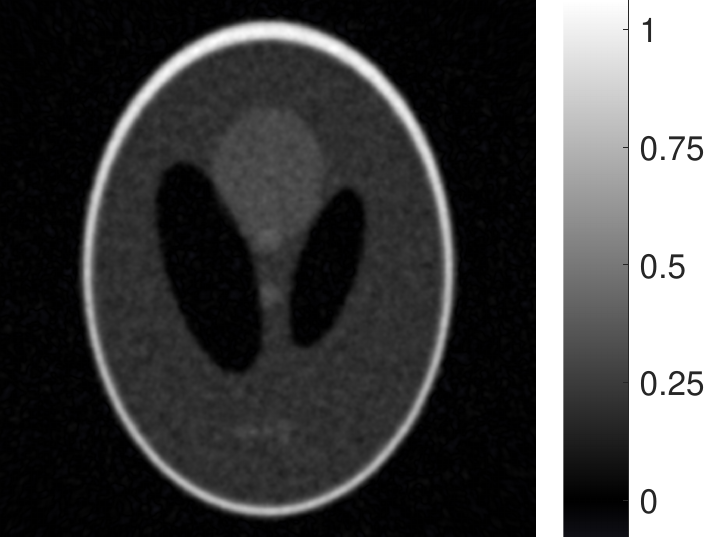}}
\hfill
\subfigure[$A_{L,D}^\varepsilon$]{\includegraphics[height=3.85cm]{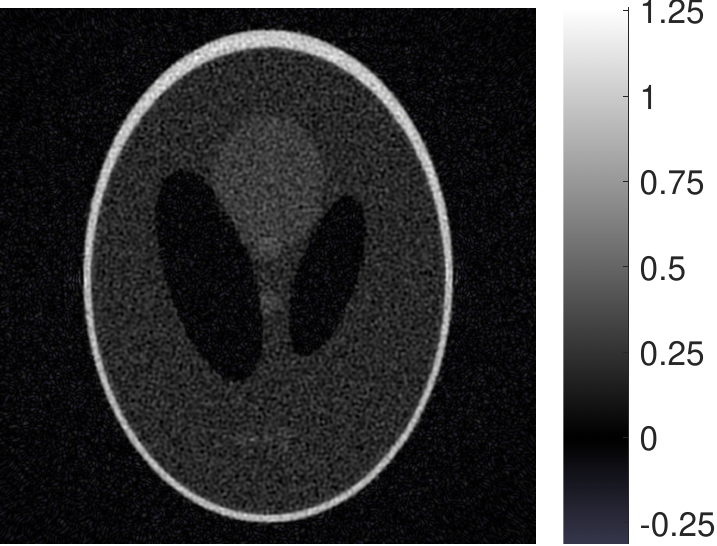}}
\hfill
\subfigure[$\hat{A}_{L,D}^\varepsilon$]{\includegraphics[height=3.85cm]{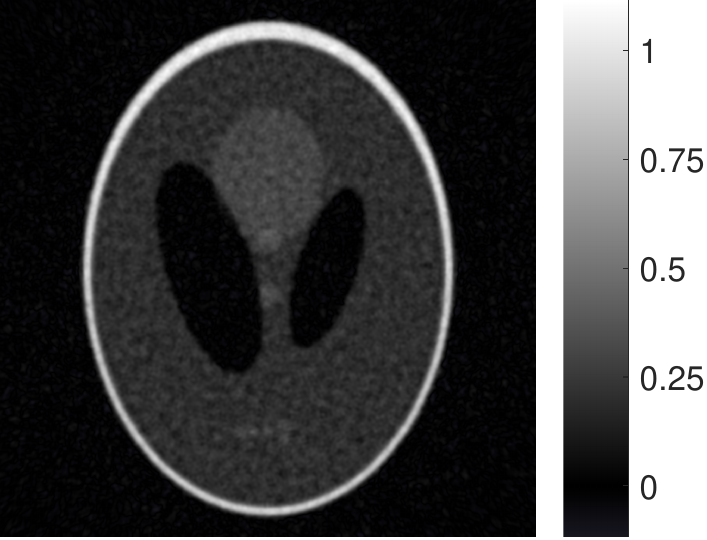}}
\hfill
\subfigure[$A_{L,D}^\ast$]{\includegraphics[height=3.85cm]{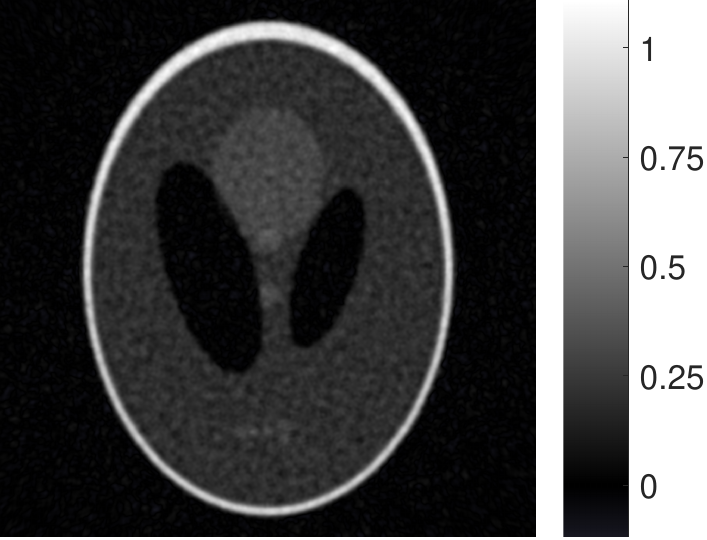}}
\caption{Reconstructions of Shepp-Logan phantom from noisy Radon data ($p_\mathrm{noise}=0.1$, $N_\varphi = 360$).}
\label{fig:shepp_logan_recon}
\end{figure}

The results of our numerical experiments are depicted in Figure~\ref{fig:errors_shepp_logan_noise}, where we plot the mean squared error (MSE) in logarithmic scale and the structural similarity index measure (SSIM) from~\cite{Wang2004} of the FBP reconstruction on an equidistant grid of $1024 \times 1024$ pixels as a function of the number of angles $N_\varphi \in \{90, 180, 270, 360, 450, 540, 630, 720\}$ for various filter functions.
To account for the randomness, we computed averages over $50$ reconstructions with different noise realizations.
Involved filter parameters were optimized based on $20$ independent samples.

In terms of MSE, we observe that for the classical filters used (Ram-Lak, Shepp-Logan, Cosine, Hamming) the error first decreases, but then increases again with increasing $N_\varphi$.
In contrast to this, for the MR-FBP$_\mathrm{GM}$ method from~\cite{Pelt2014}, the SIRT-FBP method from~\cite{Pelt2015} and the filter from~\cite{Kabri2024}, which we refer to as ERM-FBP method, as well as for our proposed filter functions $A_{L,D}^\ast$ and $\hat{A}_{L,D}^\varepsilon$, the error continues to decrease for all chosen $N_\varphi$.
The MR-FBP method from~\cite{Pelt2014} and our simple workaround $A_{L,D}^\varepsilon$ give decent results for small noise, but also show an increasing error for increasing $N_\varphi$ already for moderate noise.
In total, we observe that our filters $A_{L,D}^\ast$ and $\hat{A}_{L,D}^\varepsilon$ perform best, which was expected since the MSE is the discrete counterpart of the objective we optimized in the derivation of our filter function. Moreover, we see that the difference in performance increases with increasing $p_{\mathrm{noise}} \in \{0.05, \, 0.1, \, 0.15\}$.
For illustration, Figure~\ref{fig:optimized_filter_shepp_logan} shows $A_{L,D}^\ast$, $A_{L,D}^\varepsilon$ and $\hat{A}_{L,D}^\varepsilon$ for the noisy Shepp-Logan sinogram with $N_\varphi \in \{90, 180, 360\}$ and $p_{\mathrm{noise}} \in \{0.05, \, 0.1, \, 0.15\}$.

\begin{figure}[t]
\centering
\includegraphics[height=0.5cm]{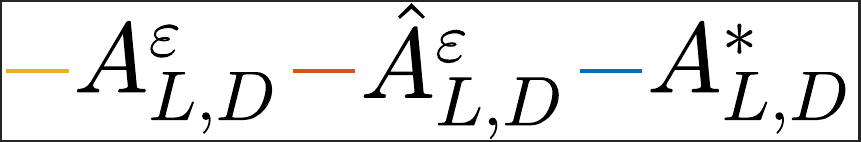}\\
\subfigure[$p_\mathrm{noise}=0.05$, $N_\varphi = 90$]{\includegraphics[width=.31\textwidth]{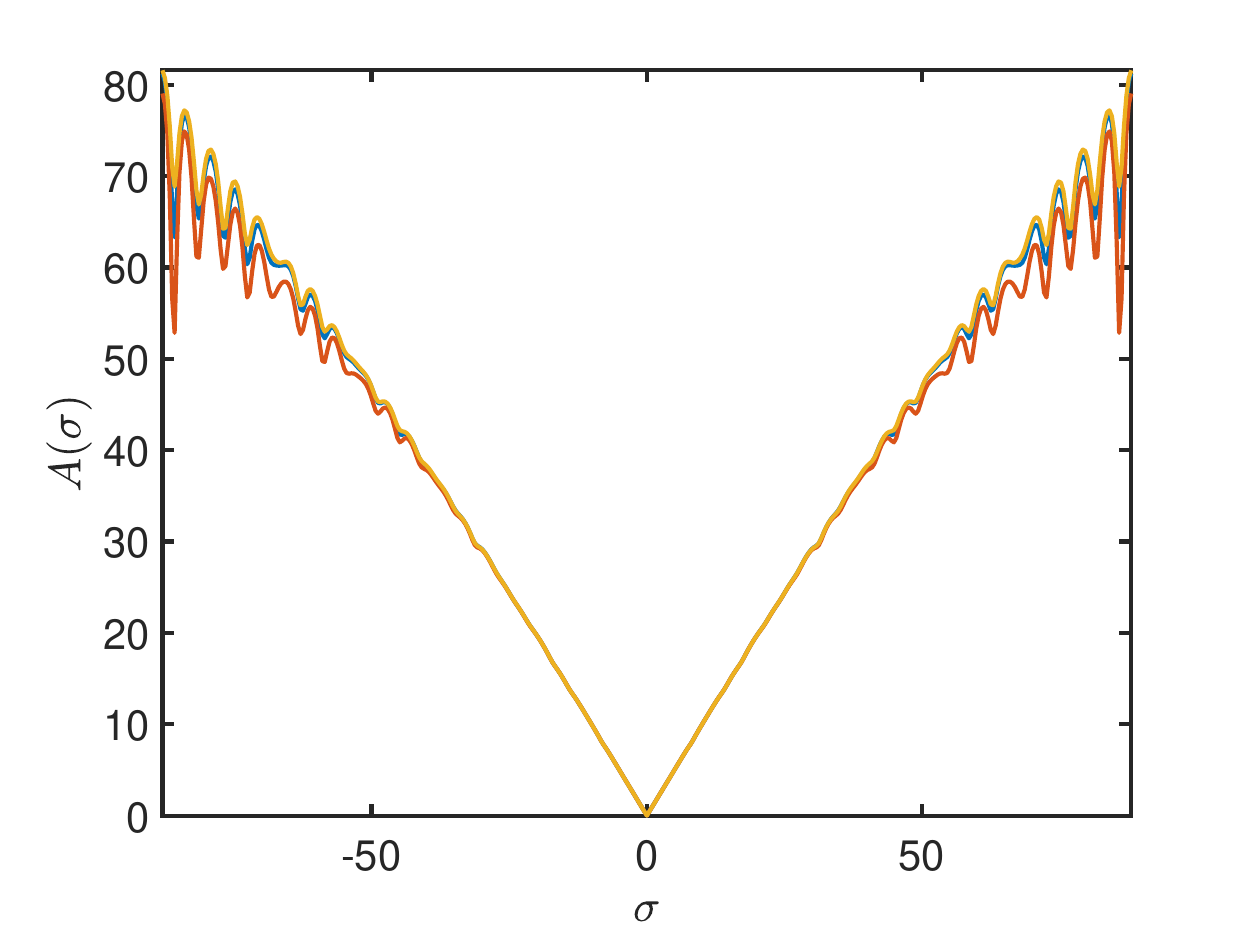}}
\hfill
\subfigure[$p_\mathrm{noise}=0.05$, $N_\varphi = 180$]{\includegraphics[width=.31\textwidth]{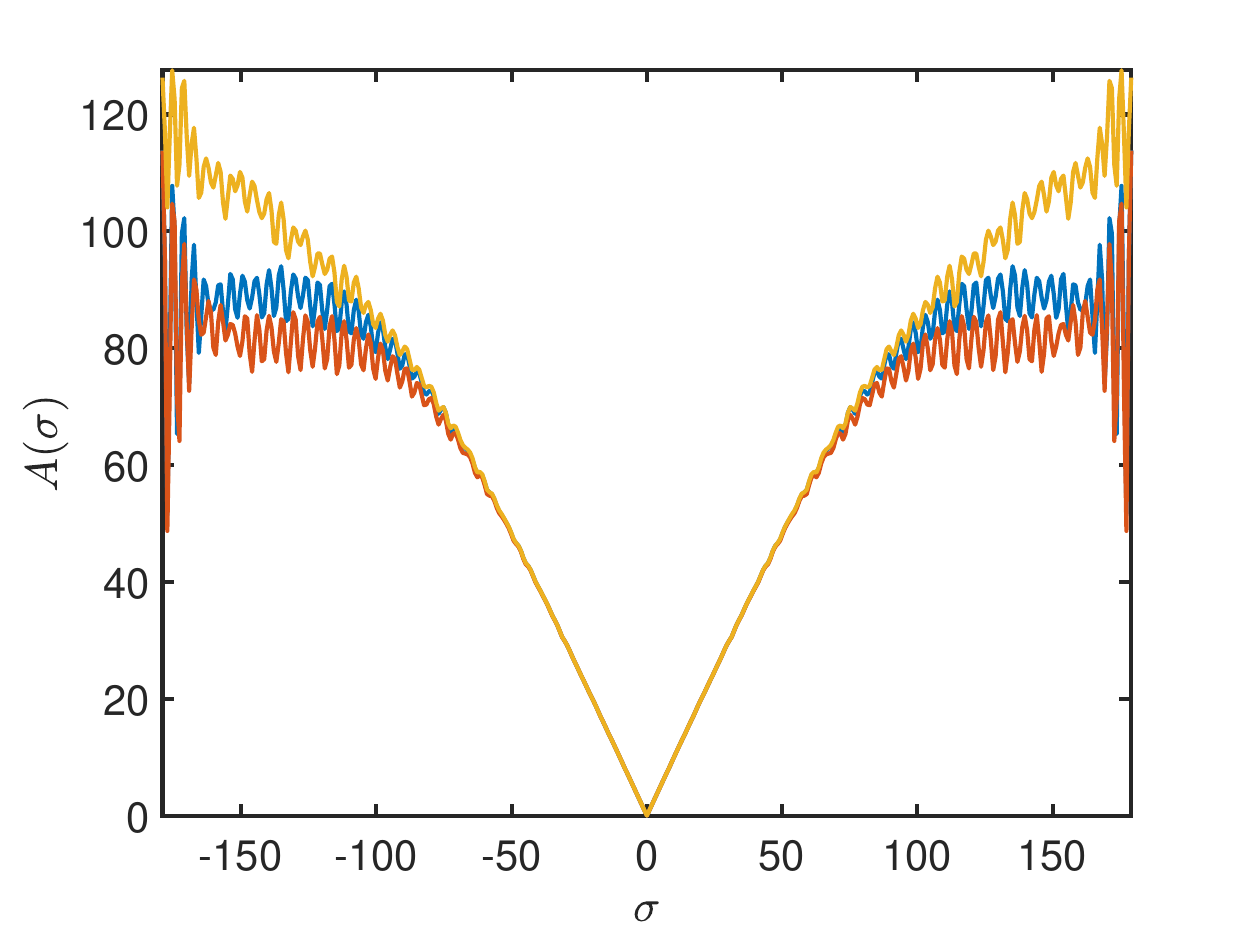}}
\hfill
\subfigure[$p_\mathrm{noise}=0.05$, $N_\varphi = 360$]{\includegraphics[width=.31\textwidth]{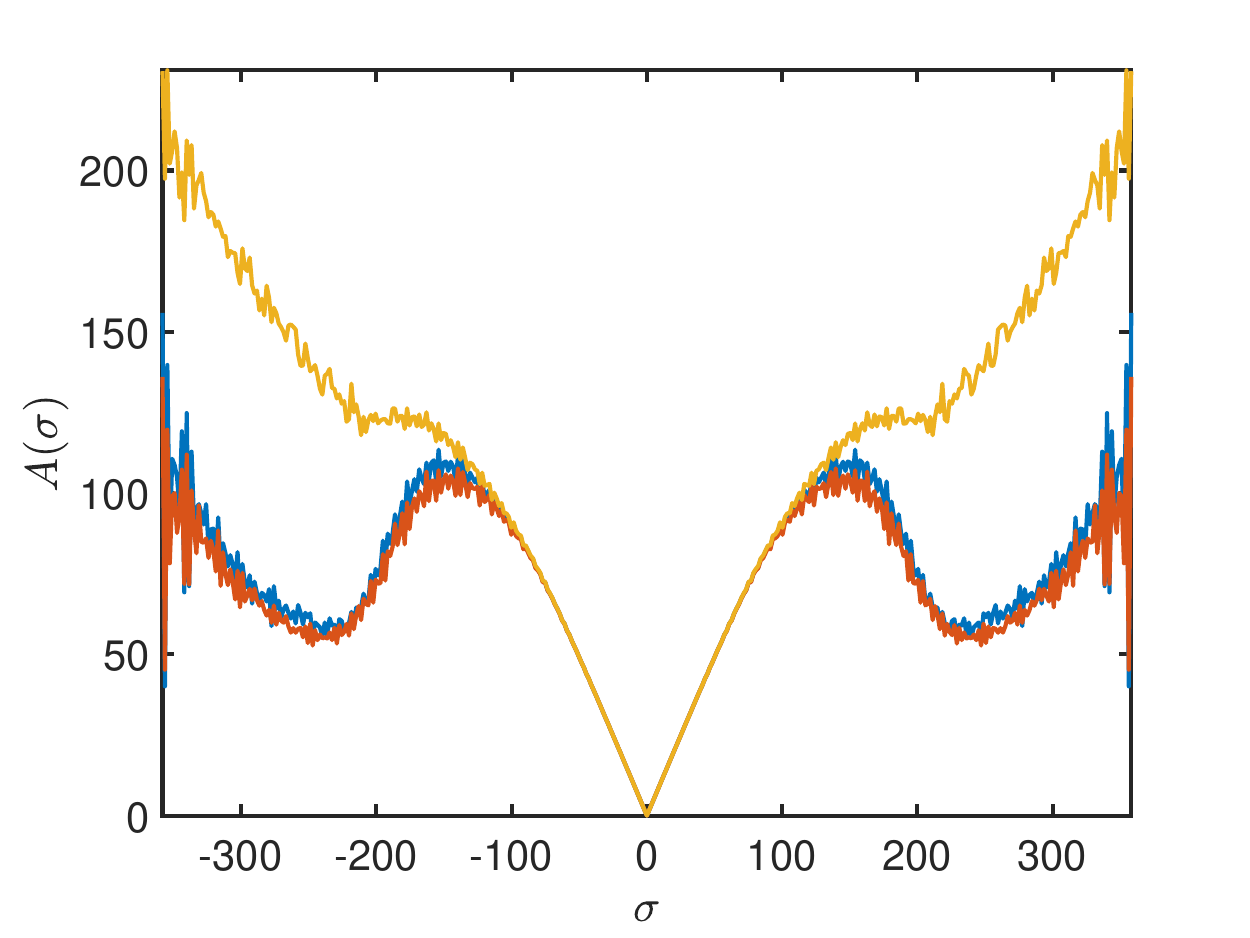}}
\hfill
\subfigure[$p_\mathrm{noise}=0.1$, $N_\varphi = 90$]{\includegraphics[width=.31\textwidth]{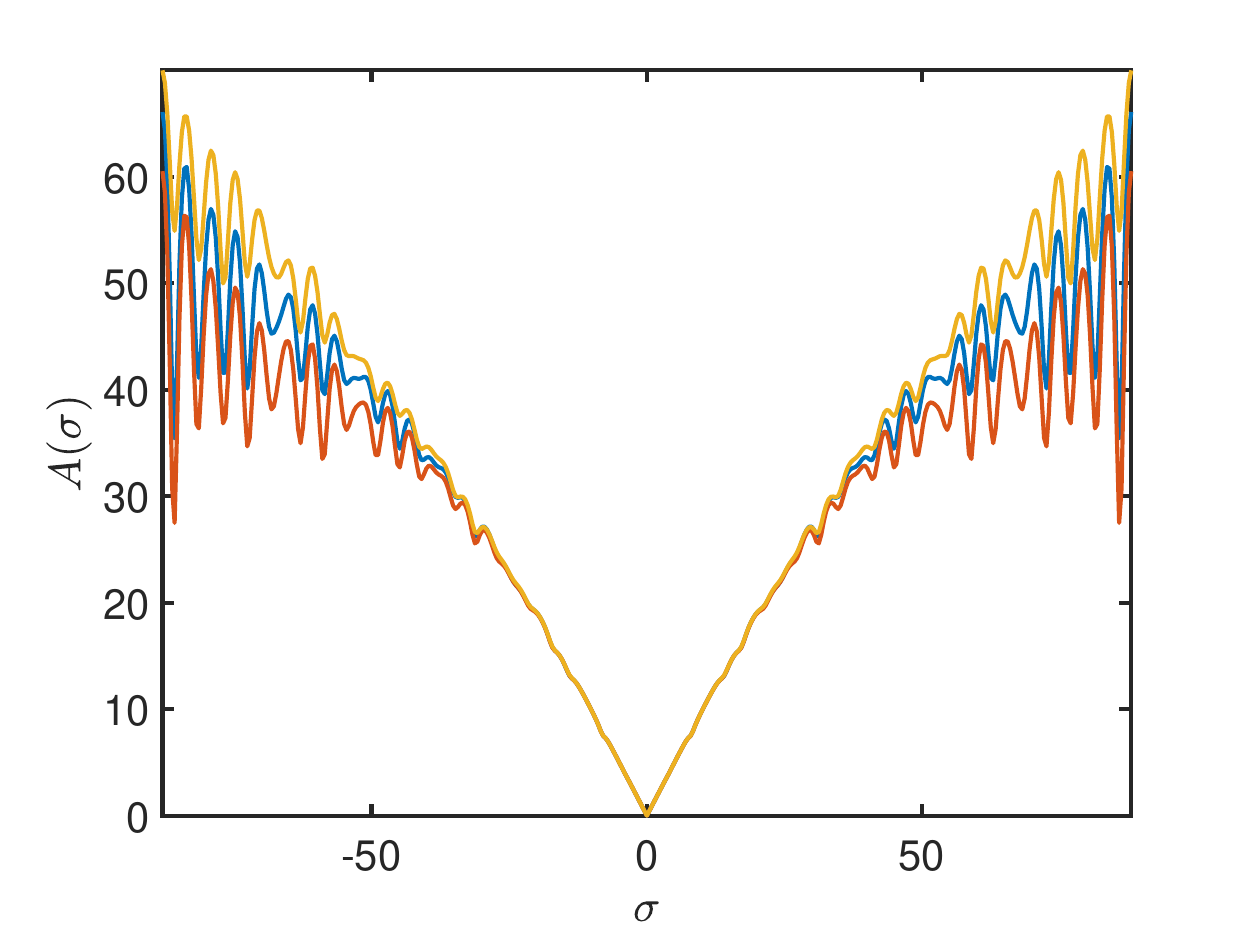}}
\hfill
\subfigure[$p_\mathrm{noise}=0.1$, $N_\varphi = 180$]{\includegraphics[width=.31\textwidth]{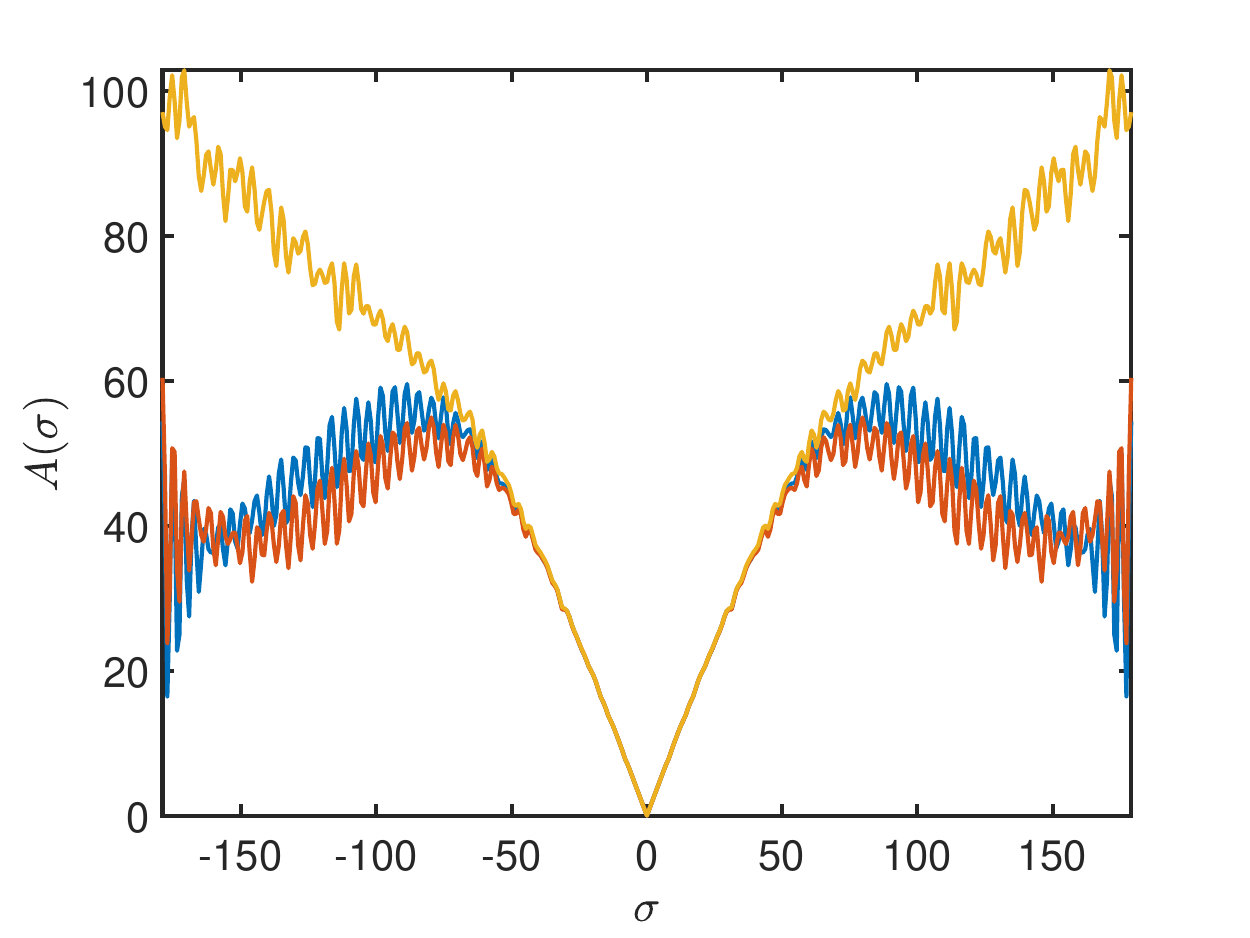}}
\hfill
\subfigure[$p_\mathrm{noise}=0.1$, $N_\varphi = 360$]{\includegraphics[width=.31\textwidth]{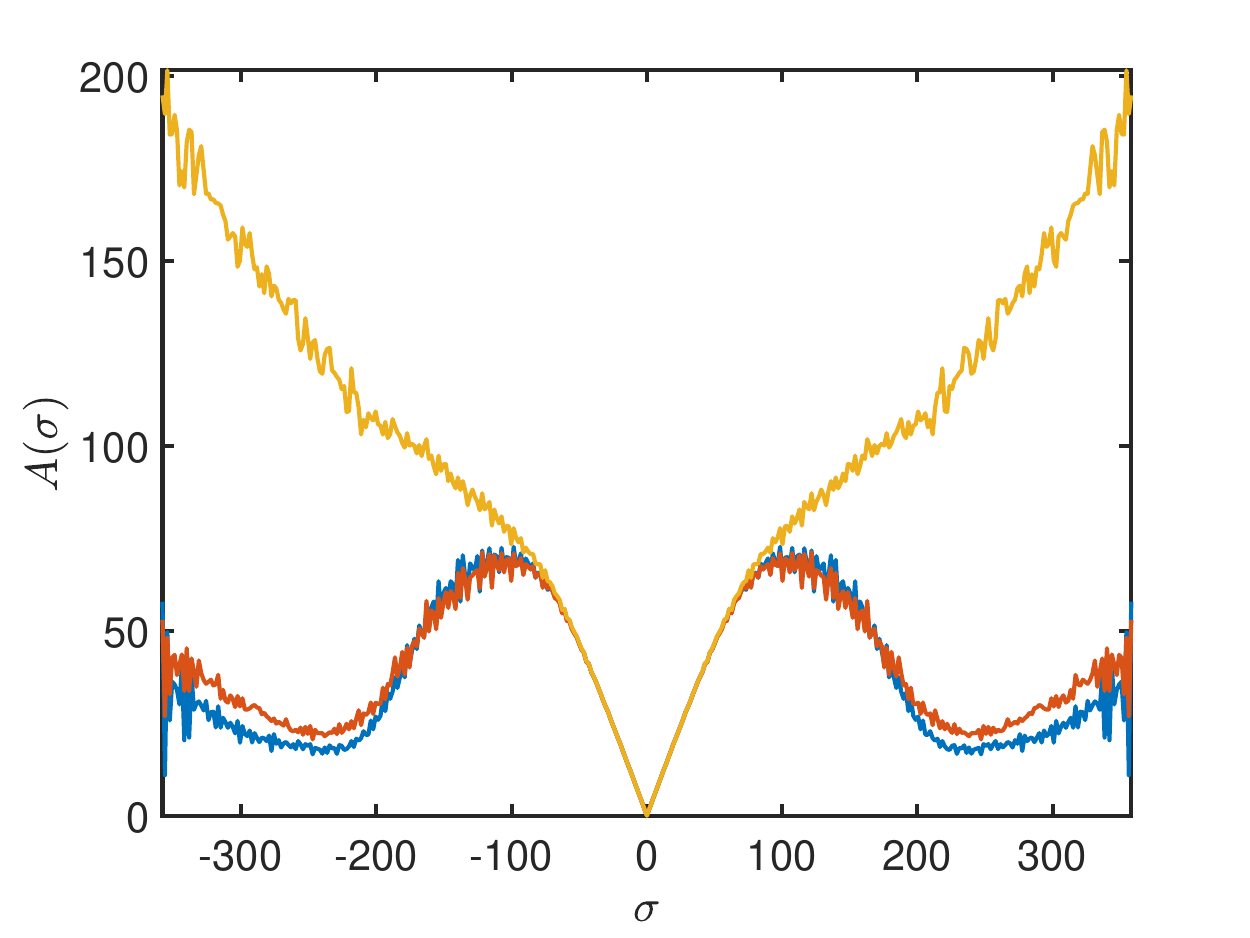}}
\hfill
\subfigure[$p_\mathrm{noise}=0.15$, $N_\varphi = 90$]{\includegraphics[width=.31\textwidth]{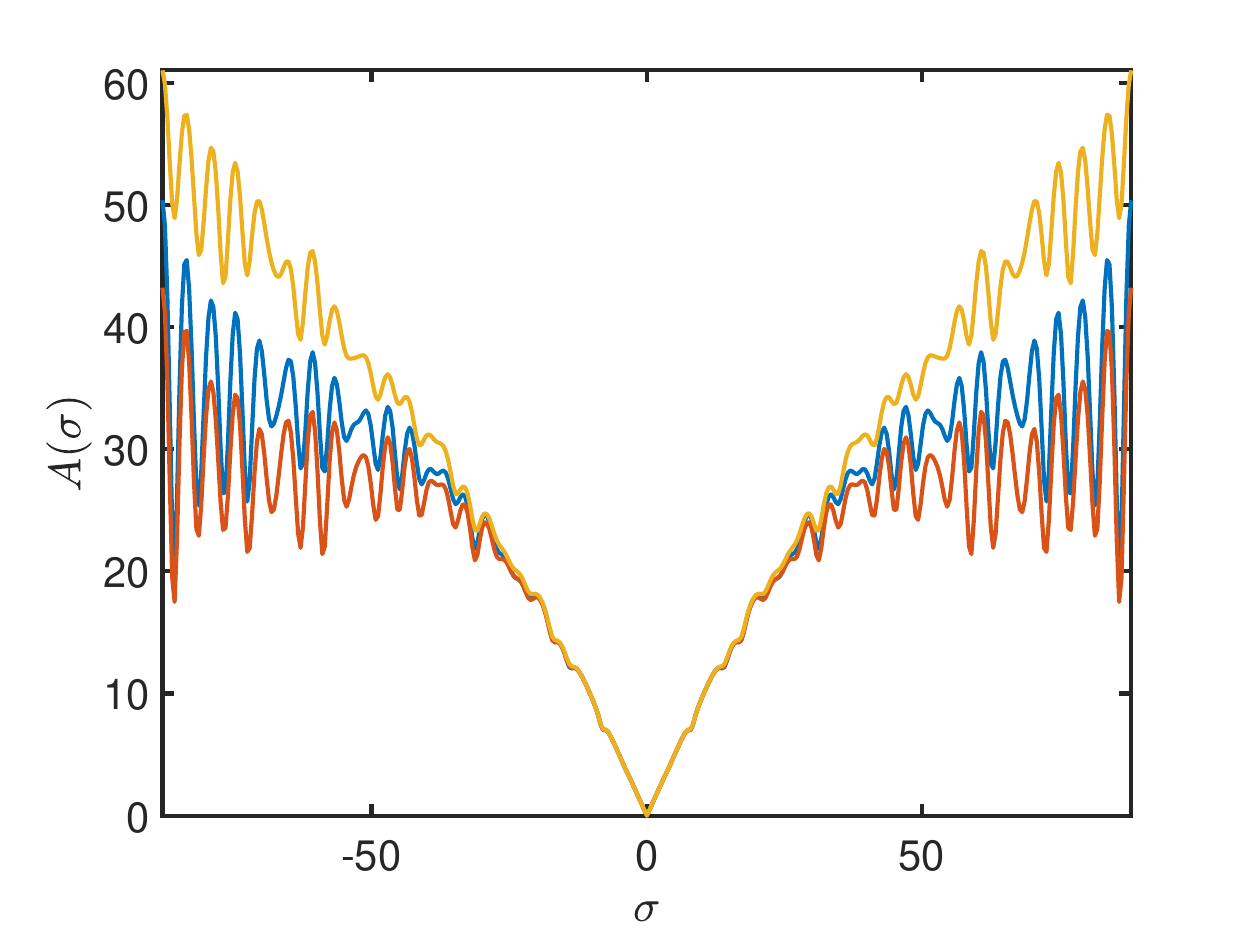}}
\hfill
\subfigure[$p_\mathrm{noise}=0.15$, $N_\varphi = 180$]{\includegraphics[width=.31\textwidth]{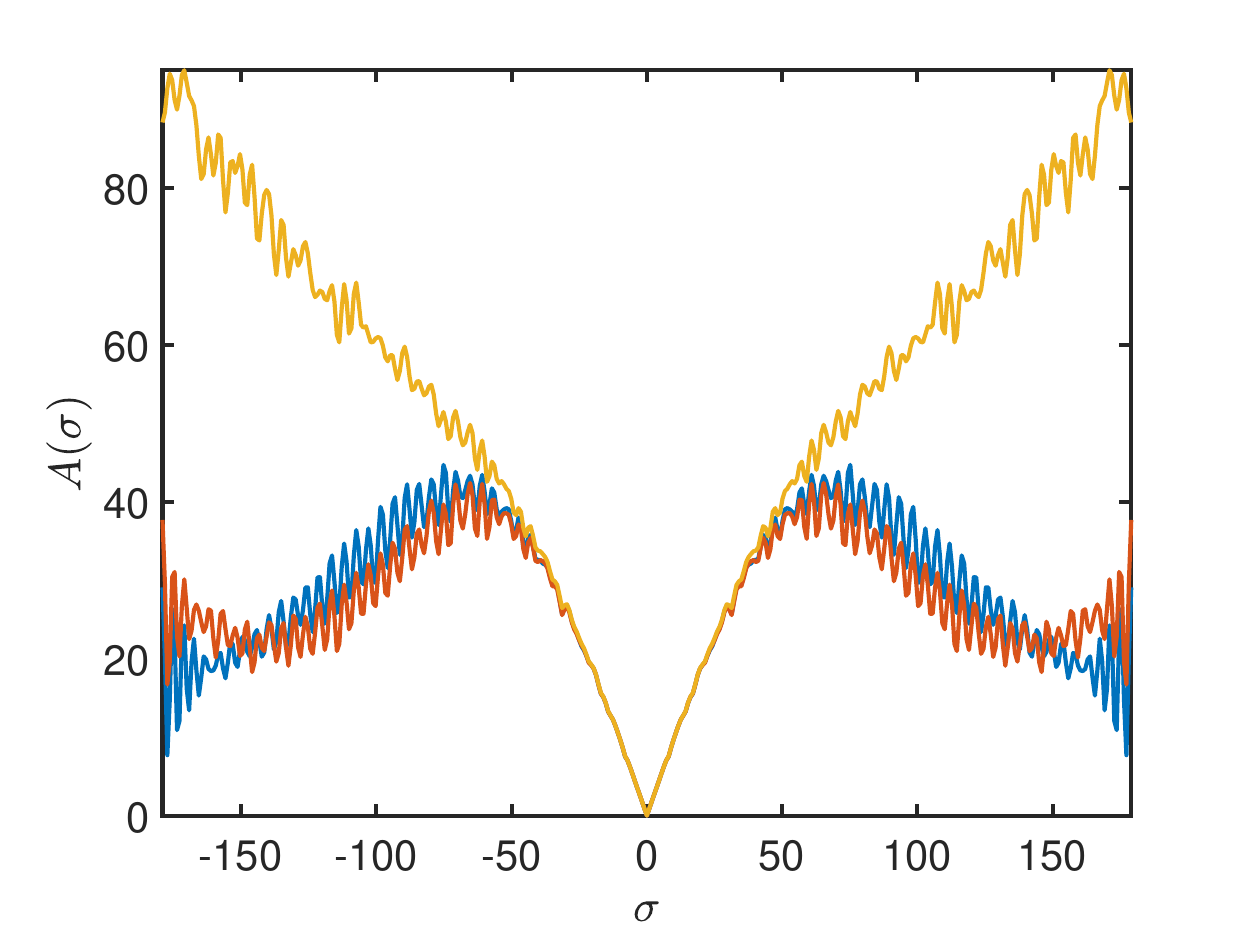}}
\hfill
\subfigure[$p_\mathrm{noise}=0.15$, $N_\varphi = 360$]{\includegraphics[width=.31\textwidth]{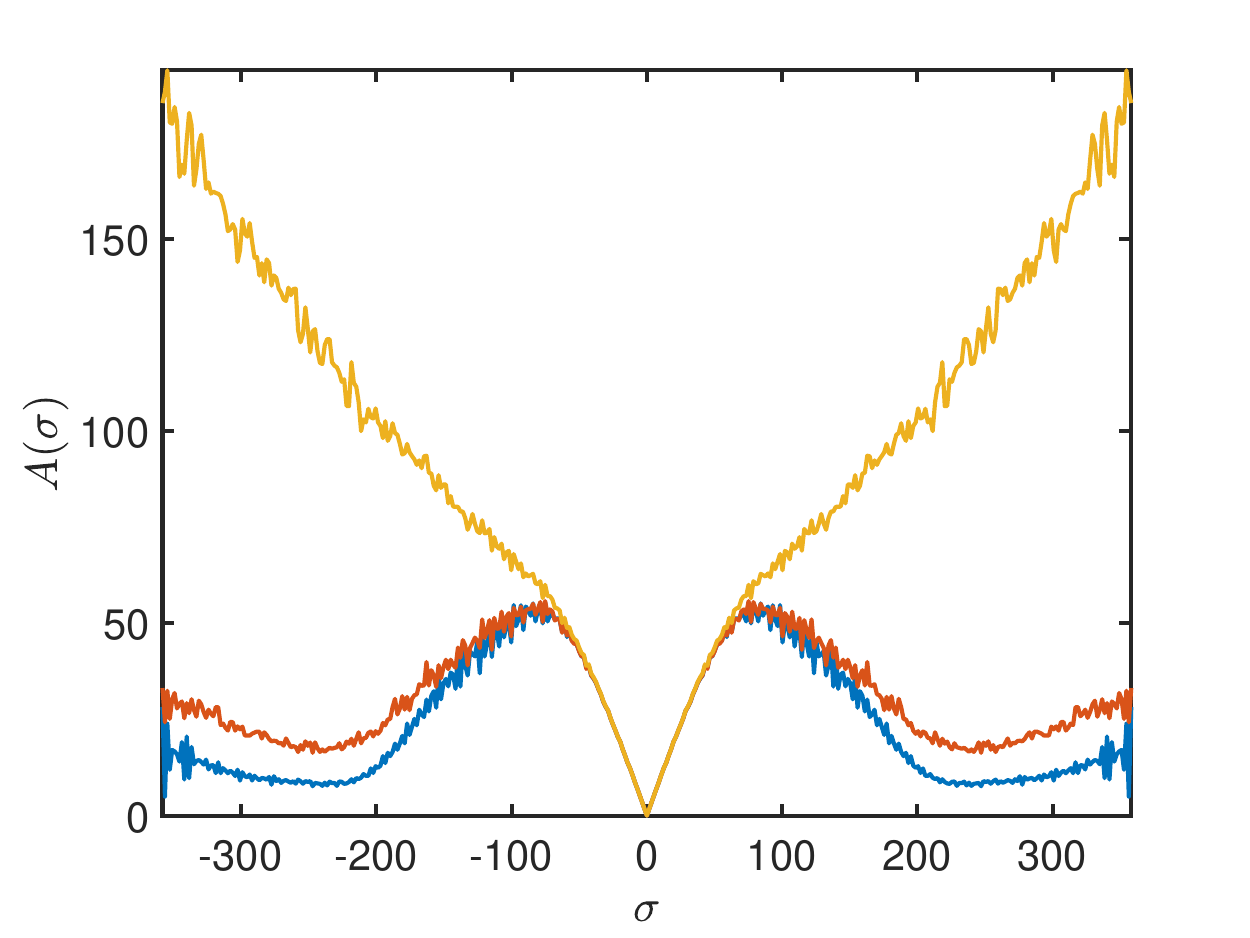}}
\caption{Plots of the optimized filter functions $A_{L,D}^\ast$, $A_{L,D}^\varepsilon$ and $\hat{A}_{L,D}^\varepsilon$ for the noisy Shepp-Logan sinogram with $N_\varphi \in \{90, 180, 360\}$ and $p_{\mathrm{noise}} \in \{0.05, \, 0.1, \, 0.15\}$.}
\label{fig:optimized_filter_shepp_logan}
\end{figure}

In terms of SSIM, we observe that for the classical filters, our simple workaround $A_{L,D}^\varepsilon$, the MR-FBP and MR-FBP$_\mathrm{GM}$ method as well as the SIRT-FBP method the reconstruction quality decreases with increasing $N_\varphi$.
For the ERM-FBP method as well as our optimized filter $A_{L,D}^\ast$, however, the SSIM increases with increasing $N_\varphi$.
For its variant $\hat{A}_{L,D}^\varepsilon$ the behaviour of the SSIM depends on the amount of noise. For small noise, the reconstruction quality increases with increasing $N_\varphi$, but for moderate and high noise, the SSIM begins to decrease again for large values of $N_\varphi$.
In total, we observe that now ERM-FBP performs best, closely followed by our filters $A_{L,D}^\ast$ and $\hat{A}_{L,D}^\varepsilon$.
Note, however, that ERM-FBP requires access to training data comprising noise-free sinograms and noise samples, whereas our filters can be used off the shelf.
As proposed in~\cite{Kabri2024}, in our simulations the ERM-FBP filter was trained on $10000$ random ellipse phantoms and different realizations of noise.
For illustration, Figure~\ref{fig:shepp_logan_recon} shows the different reconstructions from noisy Radon data with $p_\mathrm{noise}=0.1$ and $N_\varphi = 360$.

\subsection{Modified Shepp-Logan phantom}

To study the generalization abilities of the filters, in our second set of numerical simulations we consider a modified version of the Shepp-Logan phantom comprising smoother as well as rougher components, see Figure~\ref{fig:phantoms}~(c)-(d).
To be more precise, we replace the characteristic function
\begin{equation*}
\chi_{B_1(0)}(x,y) = \begin{cases}
1 & \text{if } x^2+y^2 \leq 1\\
0 & \text{if } x^2 + y^2 > 1
\end{cases}
\end{equation*}
in the definition of the classical Shepp-Logan phantom by
\begin{equation*}
p_\nu(x,y) = \begin{cases}
(1 - x^2 - y^2)^\nu & \text{if } x^2+y^2 \leq 1\\
0 & \text{if } x^2 + y^2 > 1
\end{cases}
\end{equation*}
with smoothness parameter $\nu = 1.5$ and add two rectangles of different sizes.
Thereon, we rescale the resulting attenuation function $f_{\mathrm{MSL}} \in \L^2(\R^2)$ with $\supp(f_{\mathrm{MSL}}) \subseteq B_1(0)$ so that its Radon transform $\Radon f_{\mathrm{MSL}}$ has the same arithmetic mean as $\Radon f_{\mathrm{MSL}}$, i.e., $m_{\Radon f_{\mathrm{MSL}}} = m_{\Radon f_{\mathrm{SL}}}$. 

To investigate the effect of underestimating the amount of noise, the noisy measurements $g_D^\varepsilon(i,j)$ are simulated by adding white Gaussian noise with variance $\varepsilon^2$ to $\Radon f_{\mathrm{MSL}}(s_i,\varphi_j)$, where $\varepsilon$ is $20\%$ larger than expected, i.e., $\varepsilon = 1.2 \cdot p_{\mathrm{noise}} \, m_{\Radon f_{\mathrm{MSL}}}$ with $p_{\mathrm{noise}} \in \{0.05, \, 0.1, \, 0.15\}$ and we omit the factor $1.2$ in the definition of our proposed filters $A_{L,D}^\ast$, $\hat{A}_{L,D}^\varepsilon$ and $A_{L,D}^\varepsilon$.

\begin{figure}[p]
\centering
\includegraphics[height=1cm]{Results_legend.pdf}\\
\subfigure[$p_\mathrm{noise}=0.05$]{%
\includegraphics[width=.5\textwidth, viewport=50 25 1075 675]{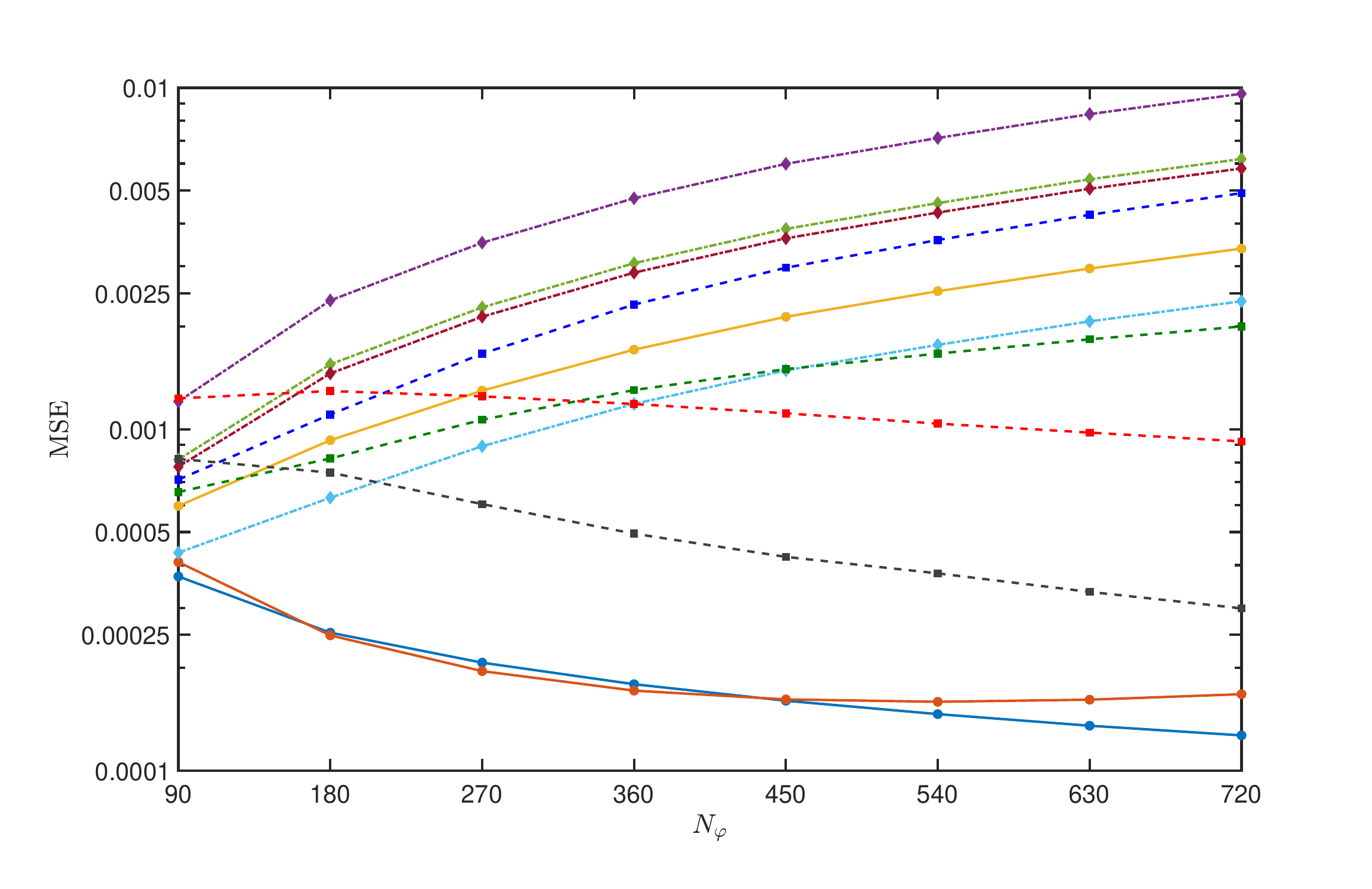}%
\includegraphics[width=.5\textwidth, viewport=50 25 1075 675]{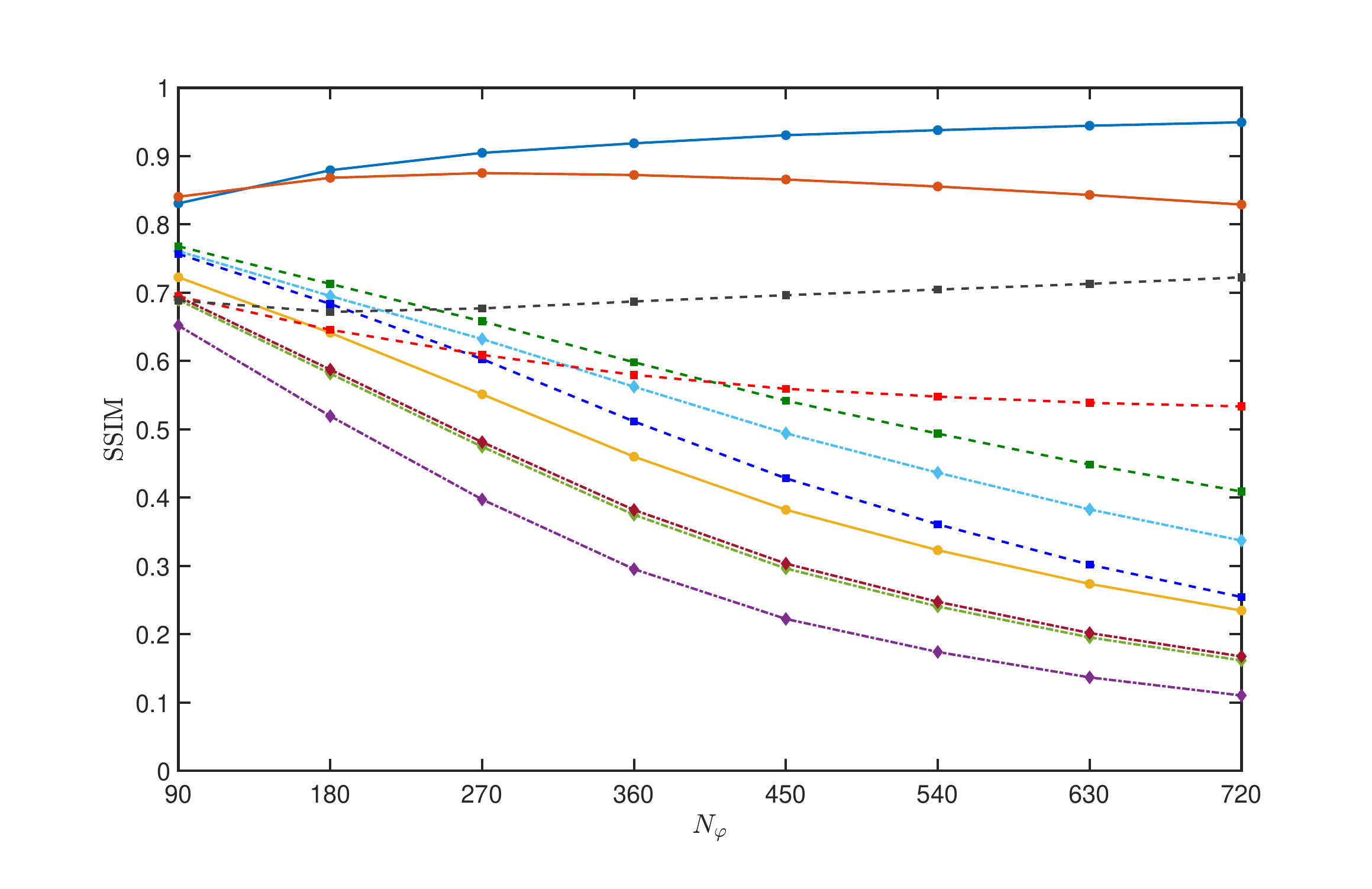}}
\hfill
\subfigure[$p_\mathrm{noise}=0.1$]{%
\includegraphics[width=.5\textwidth, viewport=50 25 1075 675]{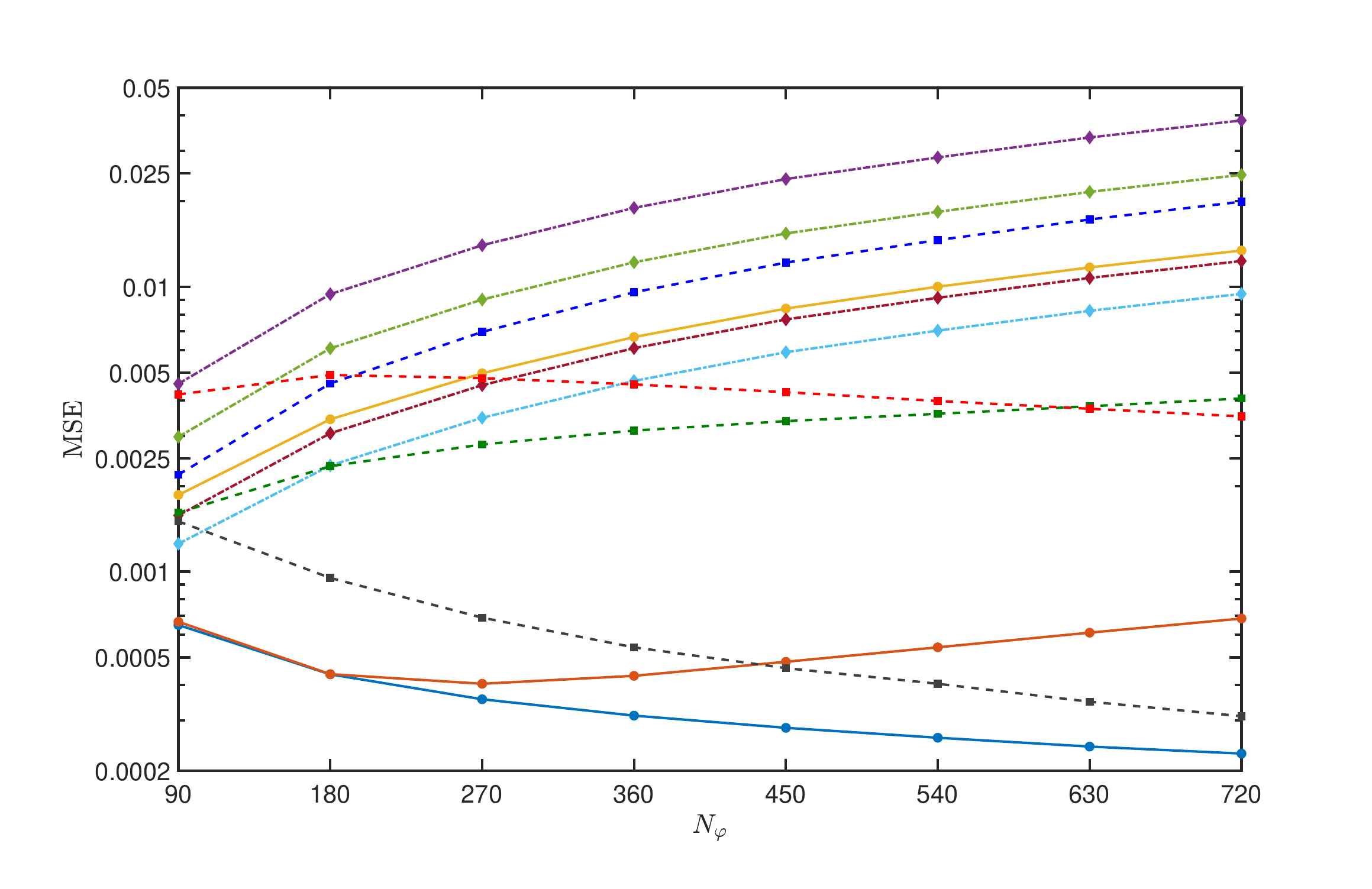}%
\includegraphics[width=.5\textwidth, viewport=50 25 1075 675]{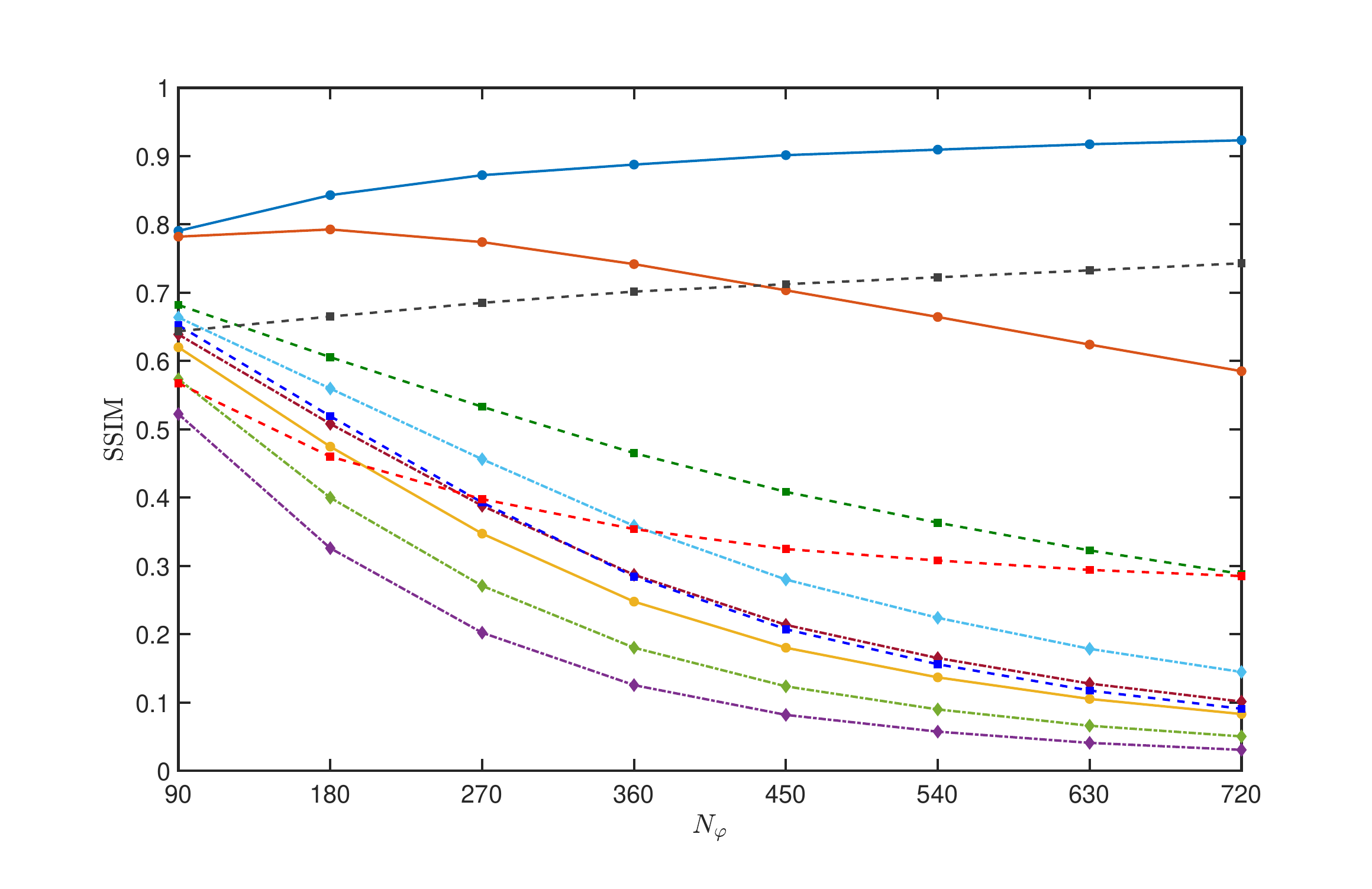}}
\hfill
\subfigure[$p_\mathrm{noise}=0.15$]{%
\includegraphics[width=.5\textwidth, viewport=50 25 1075 675]{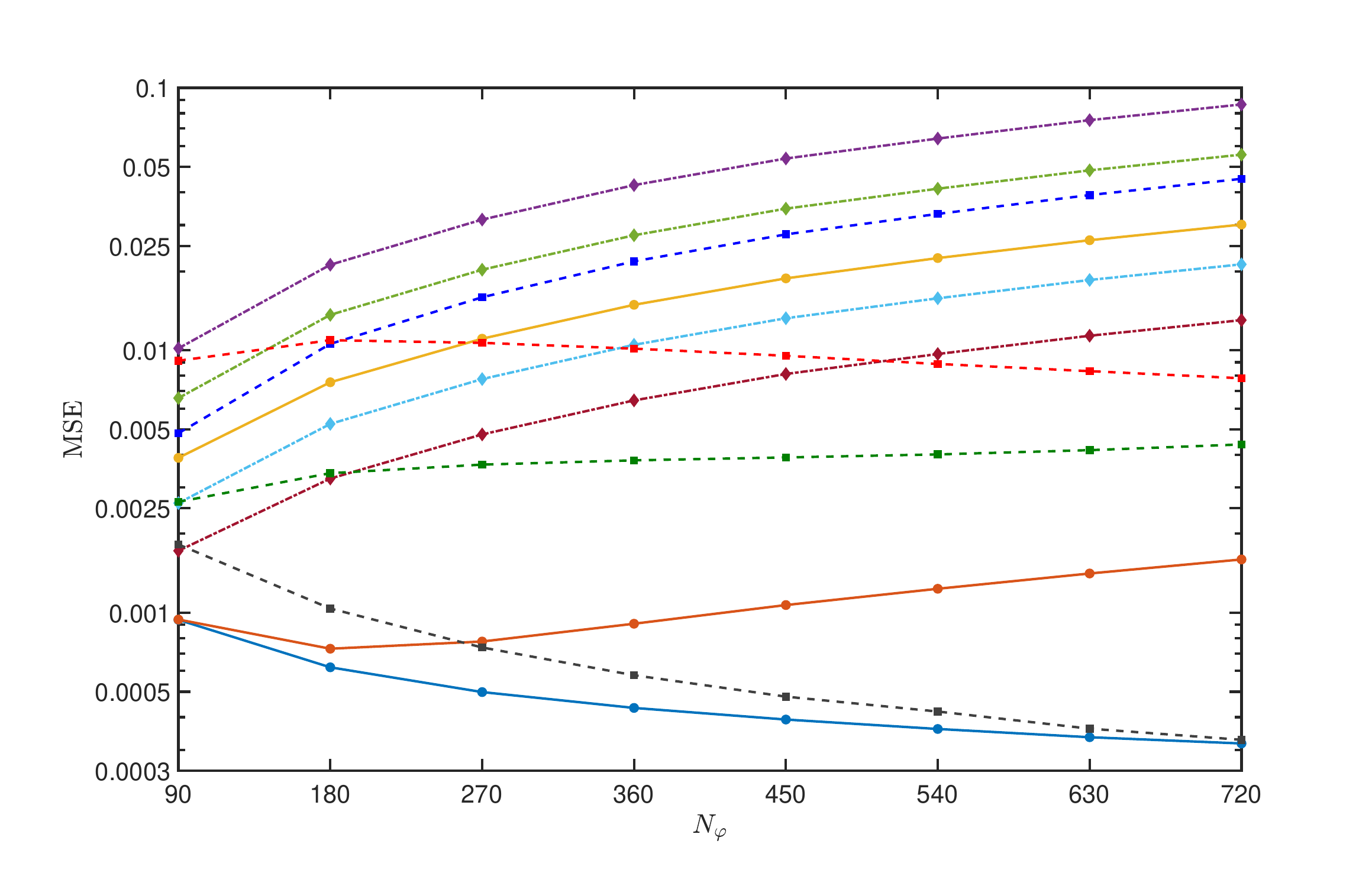}%
\includegraphics[width=.5\textwidth, viewport=50 25 1075 675]{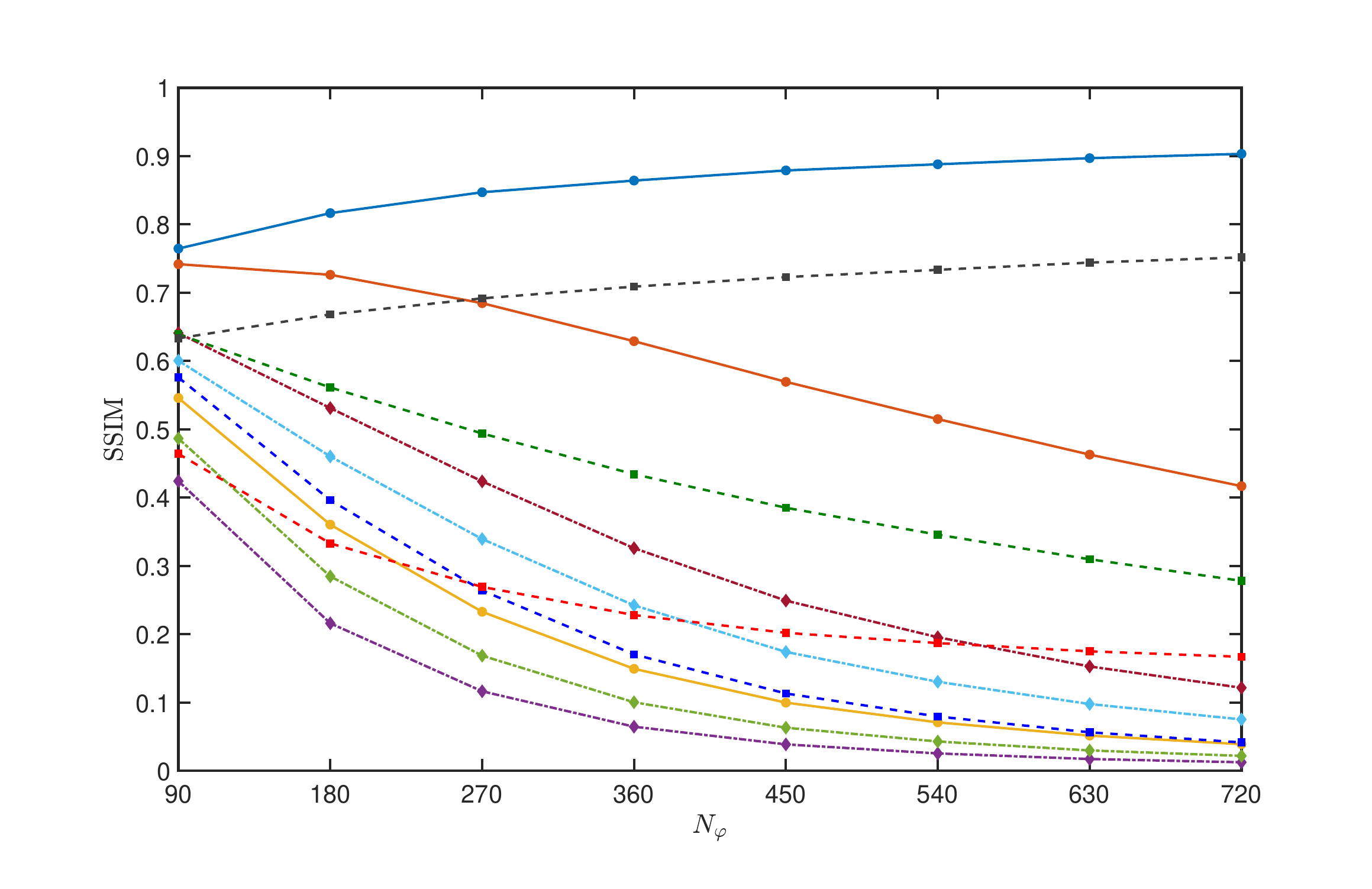}}
\caption{Plots of the MSE and SSIM of FBP reconstructions for the modified Shepp-Logan phantom.}
\label{fig:errors_modified_shepp_logan_noise}
\end{figure}

\begin{figure}[p]
\centering
\subfigure[Ground truth]{\includegraphics[height=3.85cm]{SmoothRectangle_phantom.pdf}}
\hfill
\subfigure[Ram-Lak]{\includegraphics[height=3.85cm]{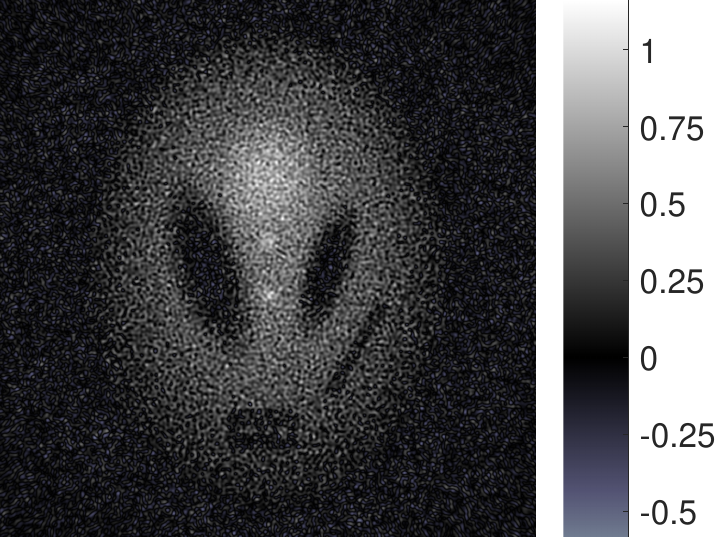}}
\hfill
\subfigure[Shepp-Logan]{\includegraphics[height=3.85cm]{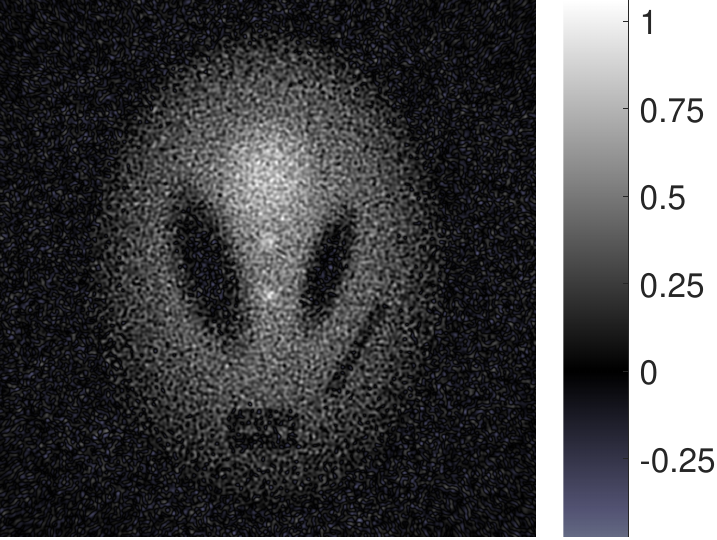}}
\hfill
\subfigure[Cosine]{\includegraphics[height=3.85cm]{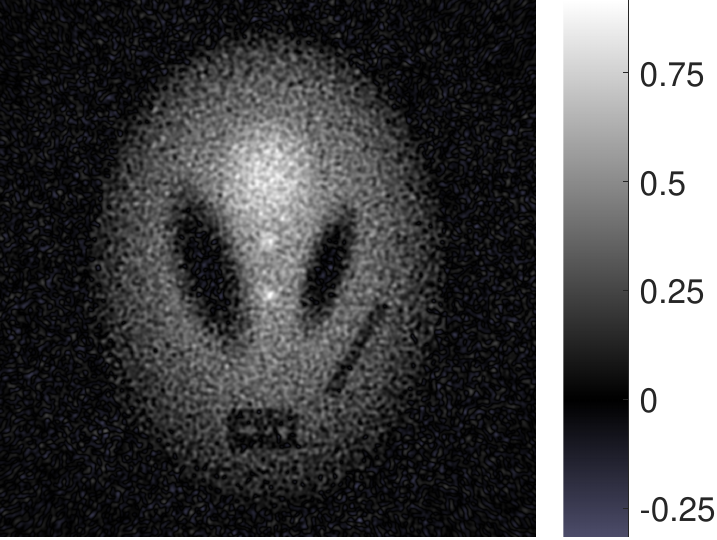}}
\hfill
\subfigure[Hamming]{\includegraphics[height=3.85cm]{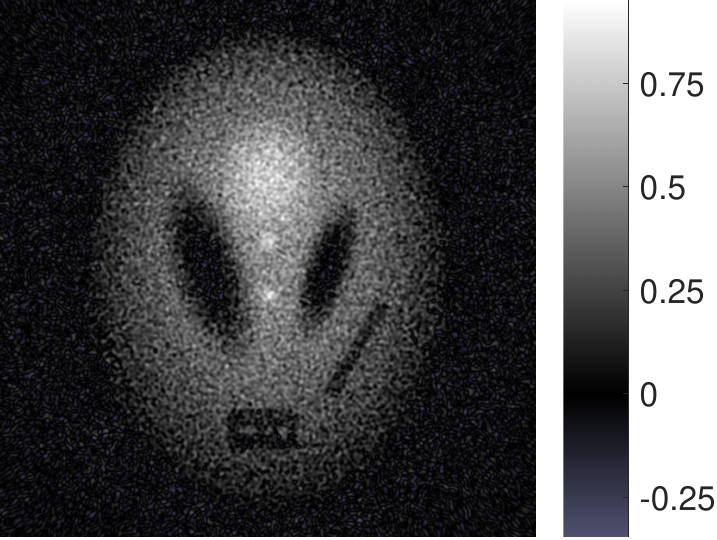}}
\hfill
\subfigure[MR-FBP]{\includegraphics[height=3.85cm]{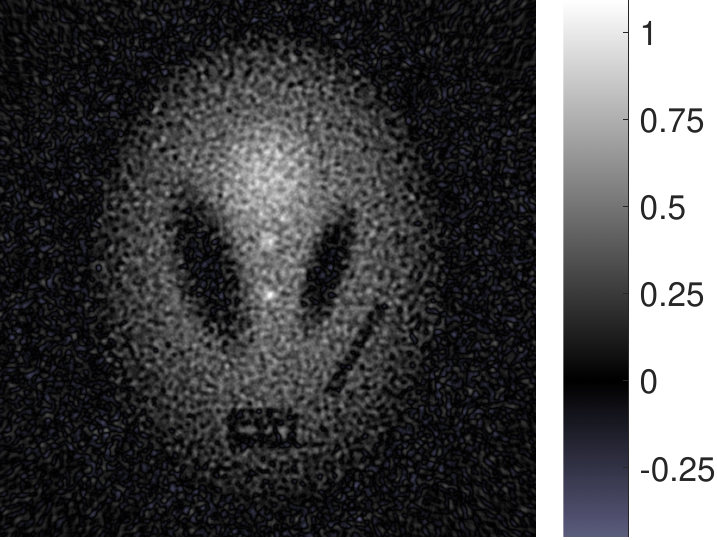}}
\hfill
\subfigure[MR-FBP$_\mathrm{GM}$]{\includegraphics[height=3.85cm]{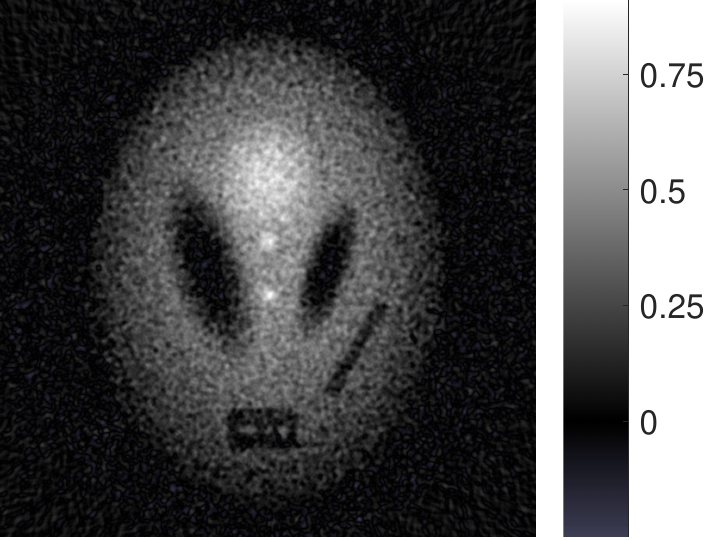}}
\hfill
\subfigure[SIRT-FBP]{\includegraphics[height=3.85cm]{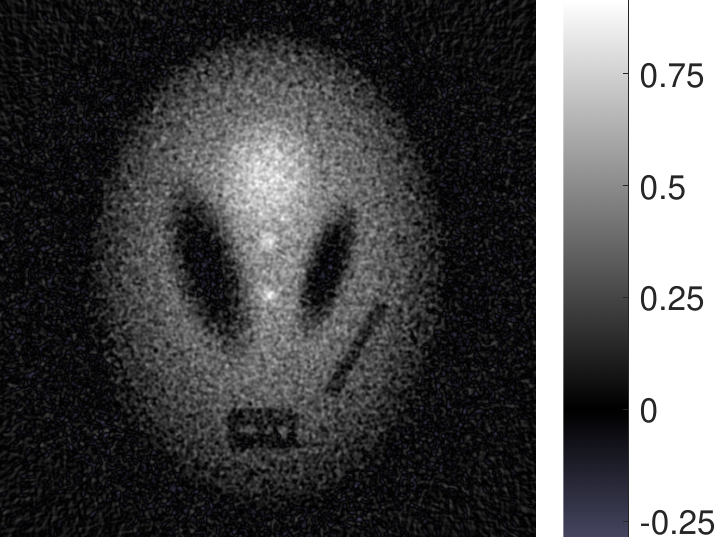}}
\hfill
\subfigure[ERM-FBP]{\includegraphics[height=3.85cm]{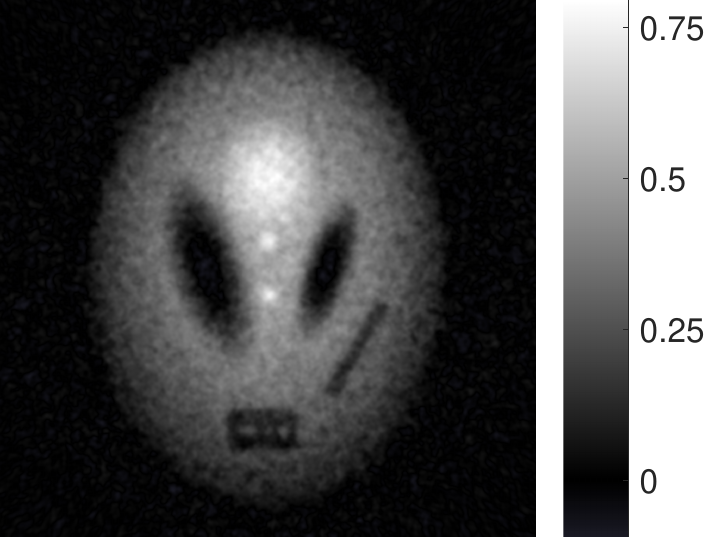}}
\hfill
\subfigure[$A_{L,D}^\varepsilon$]{\includegraphics[height=3.85cm]{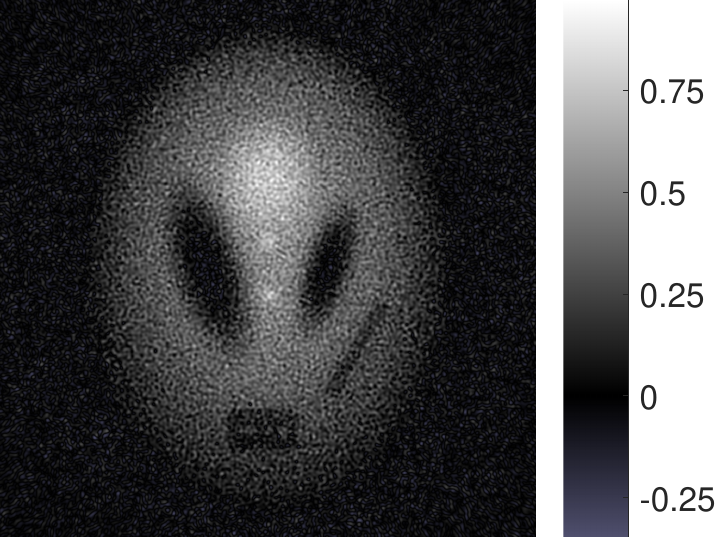}}
\hfill
\subfigure[$\hat{A}_{L,D}^\varepsilon$]{\includegraphics[height=3.85cm]{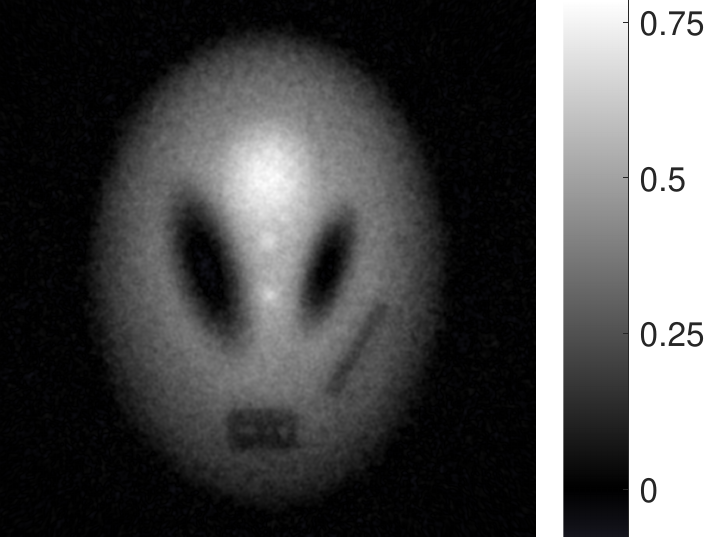}}
\hfill
\subfigure[$A_{L,D}^\ast$]{\includegraphics[height=3.85cm]{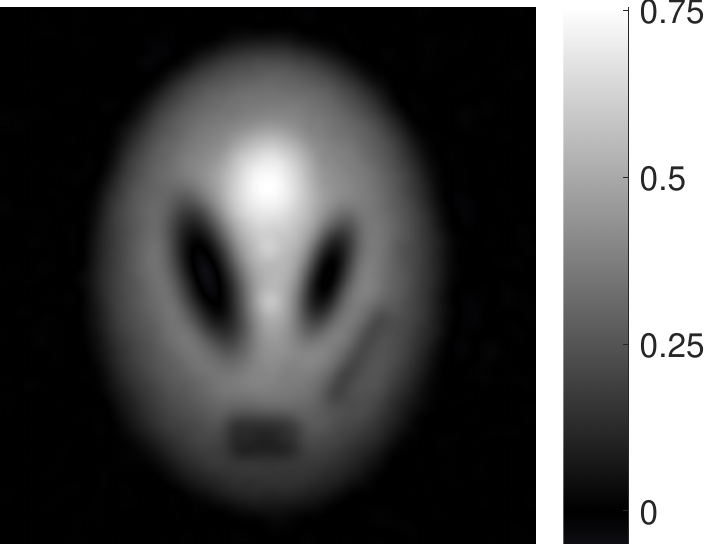}}
\caption{Reconstructions of modified Shepp-Logan phantom from noisy Radon data ($p_\mathrm{noise}=0.1$, $N_\varphi = 360$).}
\label{fig:modified_shepp_logan_recon}
\end{figure}

The results of our numerical experiments are depicted in Figure~\ref{fig:errors_modified_shepp_logan_noise}, where we use the same filter parameters as in the first set of experiments for the Shepp-Logan phantom in Section~\ref{sec:shepp_logan}.
In terms of MSE, we observe that for the classical filters, the MR-FBP method as well as the MR-FBP$_\mathrm{GM}$ method the error increases with increasing $N_\varphi$.
The same is true for our simple workaround $A_{L,D}^\varepsilon$, which shows that this choice cannot deal with underestimated noise levels.
For $\hat{A}_{L,D}^\varepsilon$ the behaviour depends on the amount of noise. For small noise, the MSE increases with increasing $N_\varphi$, but for moderate and high noise, the MSE begins to increase again for large values of $N_\varphi$.
The SIRT-FBP method only slightly decreases with increasing $N_\varphi$, whereas ERM-FBP and $A_{L,D}^\ast$ show the best results with $A_{L,D}^\ast$ outperforming ERM-FBP.
However, the difference in performance decreases with increasing $p_{\mathrm{noise}}$.

In terms of SSIM, we again observe that for the classical filters, our simple workaround $A_{L,D}^\varepsilon$, the MR-FBP and MR-FBP$_\mathrm{GM}$ method as well as the SIRT-FBP method the SSIM decreases with increasing $N_\varphi$.
Also for $\hat{A}_{L,D}^\varepsilon$ the reconstruction quality deteriorates for large values of $N_\varphi$.
In contrast to this, for the ERM-FBP method as well as our optimized filter $A_{L,D}^\ast$ the SSIM increases with increasing $N_\varphi$, where now $A_{L,D}^\ast$ performs best.
This suggest that $A_{L,D}^\ast$ is able to compensate for an inaccurately estimated noise level, while ERM-FBP has problems in reconstructing target functions very different from the training samples.

For illustration, Figure~\ref{fig:modified_shepp_logan_recon} shows the different reconstructions of the modified Shepp-Logan phantom from noisy Radon data with $p_\mathrm{noise} = 0.1$ and $N_\varphi = 360$.
Although $A_{L,D}^\ast$ performs best in terms of MSE and SSIM, we observe that while suppressing the noise very well, it smooths out fine details in the reconstruction.
As opposed to this, $\hat{A}_{L,D}^\varepsilon$ seems to provide a decent compromise between noise reduction and detail reconstruction, underpinning the potential of alternative denoising techniques for the noisy Radon samples $g_D^\varepsilon(i,j)$ in the definition of $A_{L,D}^\varepsilon$.

\subsection{2DeteCT}

\begin{figure}[t]
\centering
\subfigure[Target]{\includegraphics[height=3.85cm]{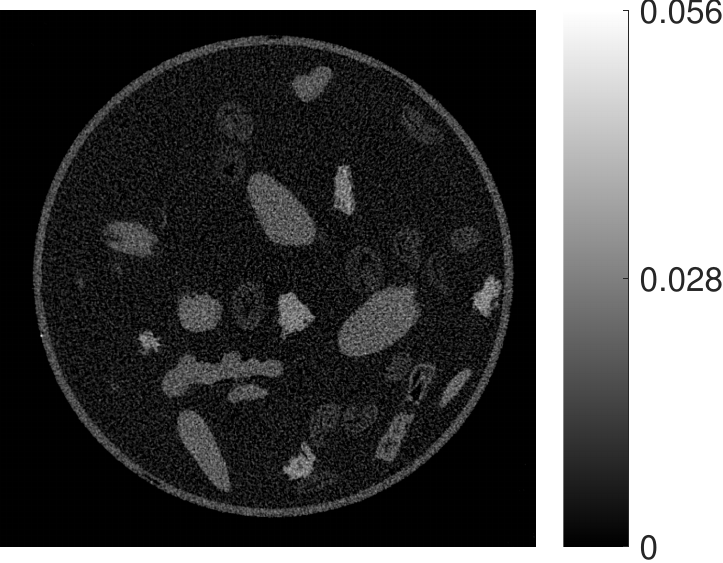}}
\hfill
\subfigure[$\hat{A}_{L,D}^\varepsilon$]{\includegraphics[height=3.85cm]{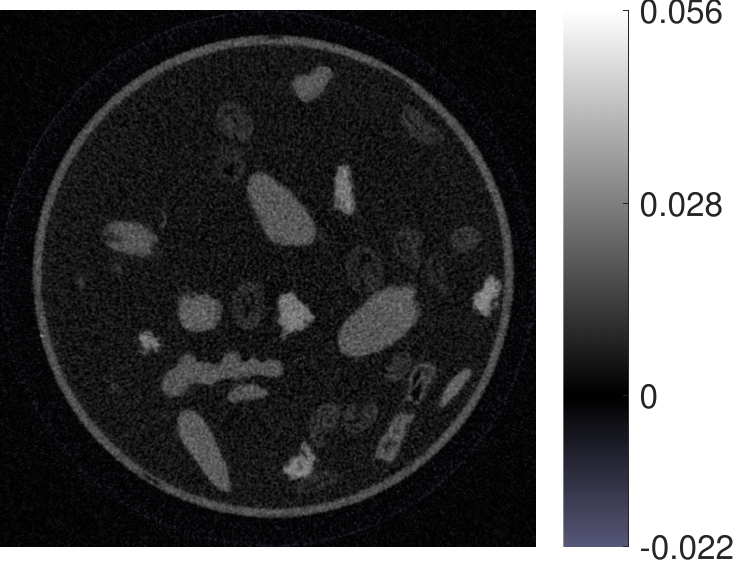}}
\hfill
\subfigure[$A_{L,D}^\varepsilon$]{\includegraphics[height=3.85cm]{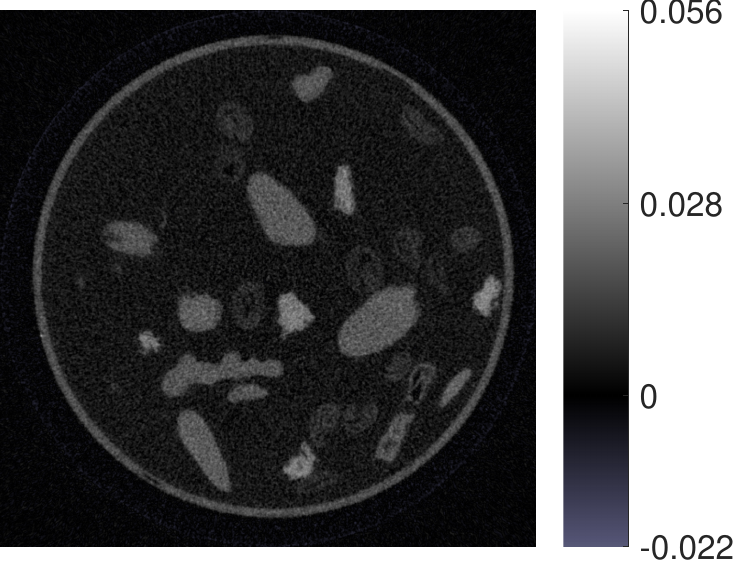}}
\hfill
\subfigure[Ram-Lak]{\includegraphics[height=3.85cm]{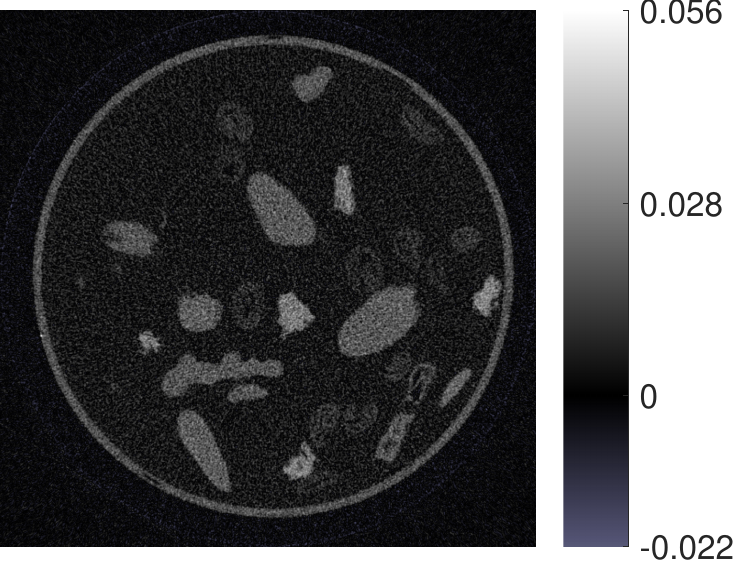}}
\hfill
\subfigure[Shepp-Logan]{\includegraphics[height=3.85cm]{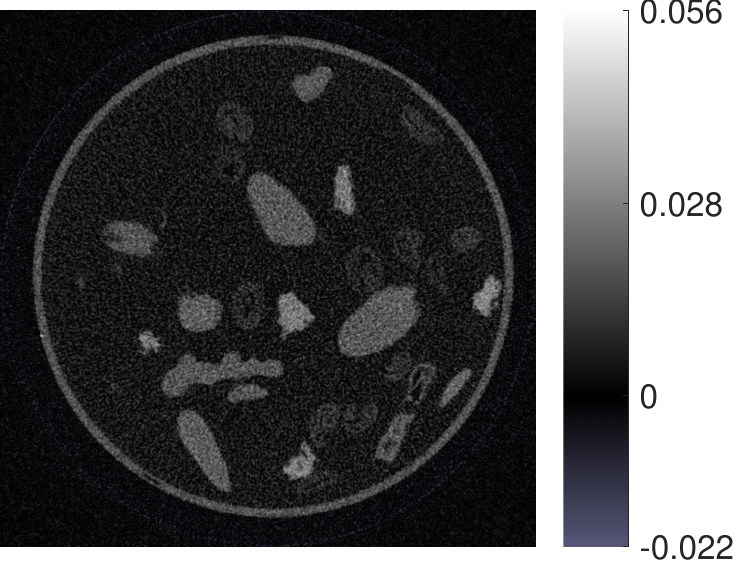}}
\hfill
\subfigure[Cosine]{\includegraphics[height=3.85cm]{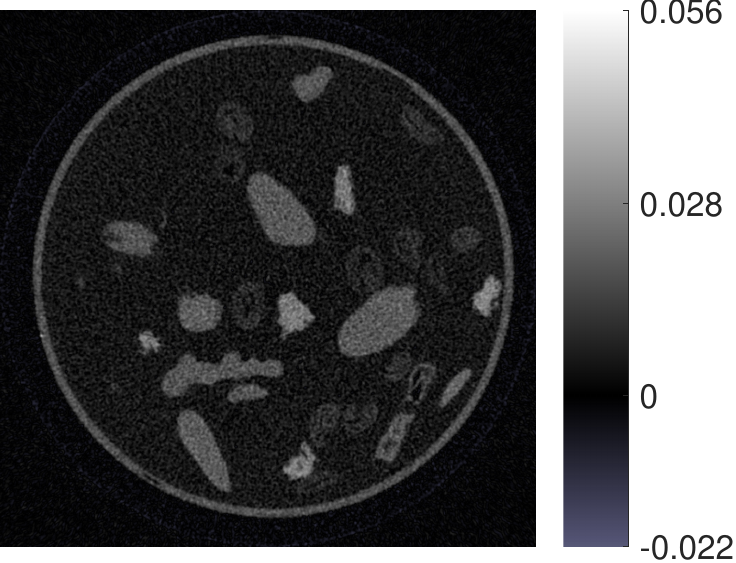}}
\hfill
\subfigure[MR-FBP]{\includegraphics[height=3.85cm]{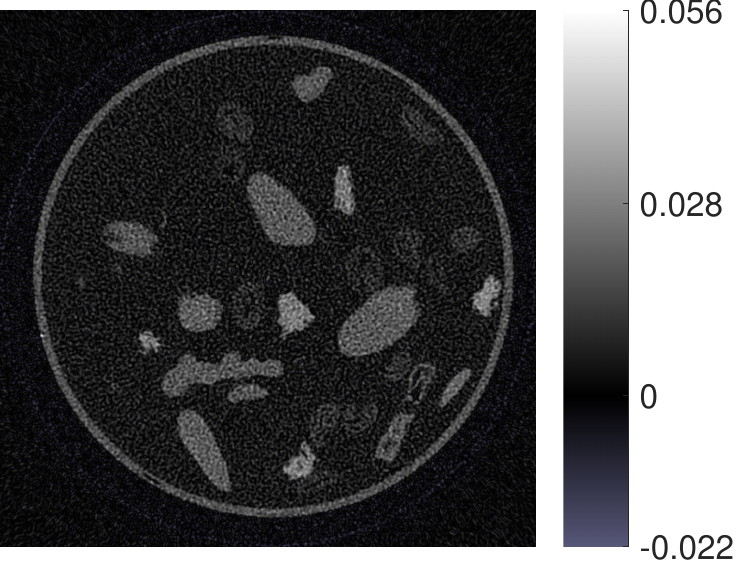}}
\hfill
\subfigure[MR-FBP$_\mathrm{GM}$]{\includegraphics[height=3.85cm]{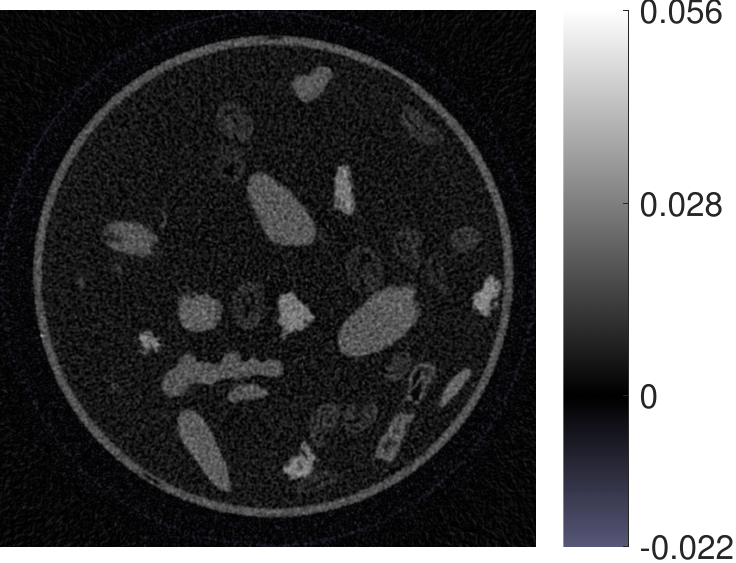}}
\hfill
\subfigure[SIRT-FBP]{\includegraphics[height=3.85cm]{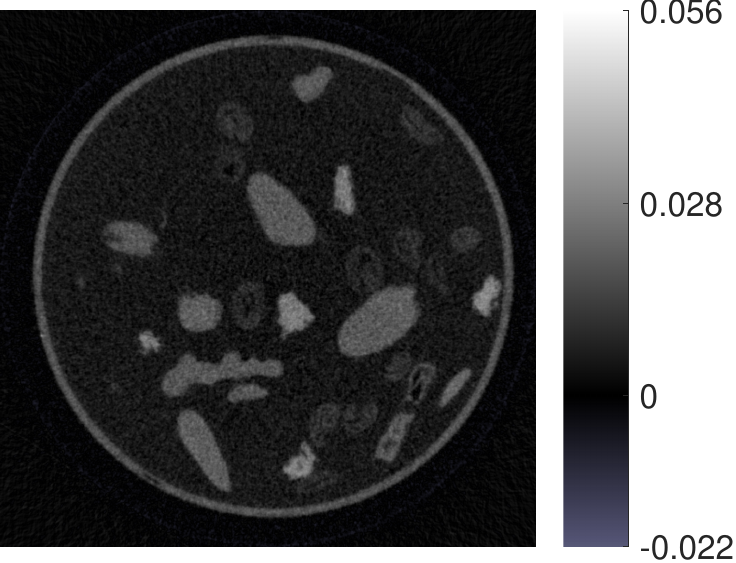}}
\caption{Reconstructions of 2DeteCT slice 1942 (mode 1).}
\label{fig:2detect_1942}
\end{figure}

In our third and last set of numerical experiments, we evaluate the performance of the different filters on realistic Radon data.
To this end, we make use of the 2DeteCT dataset~\cite{Kiss2023} consisting of real X-ray CT measurements in three acquisition modes.
Here, we focus on mode 1 yielding noisy low-dose fan-beam Radon data of slices of various test samples that produce similar image features as exhibited in medical abdominal CT scans.
In addition to the $5000$ slices of raw projection data, the dataset provides target reconstructions that serve as ground truth and were computed via an iterative reconstruction scheme to solve a non-negative least squares problem based on Nesterov accelerated gradient descent.

In our experiments, we first transform the provided fan-beam data into a parallel scanning geometry, which is needed for applying the discrete FBP reconstruction formula~\eqref{eq:discrete_FBP_formula}. 
Thereon, we subsample the sinograms with factor 2 in both the radial and angular variable, which serves as basis for the FBP reconstructions on an equidistant grid of $1024 \times 1024$ pixels.

Exemplary reconstructions are depicted in Figure~\ref{fig:2detect_1942}, where we use the same filters as in our previous experiments.
Note, however, that in this setting our optimized filter $A_{L,D}^\ast$ cannot be used as noise-free Radon data is not available.
Also the ERM-FBP method from~\cite{Kabri2024} cannot be applied due to the lack of training data consisting of noise-free sinograms and noise samples.
For the other filters, involved parameters were optimized w.r.t.\ the mean MSE of the first $50$ slices of the dataset. 
For the Hamming filter, the optimal parameter is $\beta = 1$ so that it agrees with the Ram-Lak filter.
For $A_{L,D}^\varepsilon$, we use $\varepsilon = 0.045$ and for $\hat{A}_{L,D}^\varepsilon$ the optimal choice is $\varepsilon = 0.032$.

Visually, all filters yield comparable results that are close to the 2DeteCT target reconstruction, which is shown in Figure~\ref{fig:2detect_1942}~(a).
However, our filters $A_{L,D}^\varepsilon$ and $\hat{A}_{L,D}^\varepsilon$ seem to produce reconstruction with the best noise reduction.
Also in terms of MSE, our filters show the best performance with $9.0792 \cdot 10^{-6}$ for $A_{L,D}^\varepsilon$ and $9.1616 \cdot 10^{-6}$ for $\hat{A}_{L,D}^\varepsilon$.
The next best performance is achieved by the Shepp-Logan filter with an MSE of $9.1803 \cdot 10^{-6}$, followed by the Cosine filter with $9.5116 \cdot 10^{-6}$ and the MR-FBP$_\mathrm{GM}$ method with $9.8905 \cdot 10^{-6}$.
The worst performance is achieved by the Ram-Lak filter with an MSE of $1.0703 \cdot 10^{-5}$, followed by the SIRT-FBP method with $1.0742 \cdot 10^{-5}$ and lastly the MR-FBP method with an error of $1.2231 \cdot 10^{-5}$.

Remarkably, although the measurement noise is not expected to be Gaussian, our simple workaround $A_{L,D}^\varepsilon$ performs best, closely followed by $\hat{A}_{L,D}^\varepsilon$, which are explicitly designed for Gaussian noise.
We expect that the performance of $\hat{A}_{L,D}^\varepsilon$ can be further improved by utilizing a denoising scheme specifically tailored to the true noise in the Radon measurements.

\section{Discussion and Outlook}\label{sec:outlook}

This work focuses on improving the reconstruction quality of the approximate filtered back projection reconstruction $f_L^\varepsilon$ in computerized tomography by optimizing the applied low-pass filter $A_L$ with bandwidth $L>0$.
The definition of our filter function is motivated by investigations of the noise behaviour of measured CT data, resulting in an approximate Gaussian white noise model.
Based on this, our proposed filter minimizes the expected squared error $\Ex(\Vert f_L^\varepsilon - f\Vert_{\L^2(\R^2)}^2)$ given infinite noisy measurement data with noiselevel $\varepsilon>0$ for a fixed target attenuation function~$f$.
Our resulting filter $A_L^\ast$ depends on the noise approximation $h_a$, noiselevel $\varepsilon$ and Radon data~$\Radon f$. We extend this filter to handle finitely many measurements corrupted by standard discrete additive Gaussian white noise in a parallel scanning geometry. The discretized filter $A_{L,D}^\ast$ then depends on the variance of the discrete Gaussian white noise, the discretization parameters and the true Radon data $\Radon f$ of the target function~$f$. 
To circumvent the dependency on $\Radon f$, which is rarely known in practice, we propose two adaptations. The first replaces $\Radon f$ by the measurements $g_D^\varepsilon$, resulting in the filter $A_{L,D}^\varepsilon$. The second replaces $\Radon f$ by Wiener filtered data $\hat{g}_D^\varepsilon$, resulting in the filter $\hat{A}_{L,D}^\varepsilon$.

In addition to our theoretical investigations, we conduct extensive numerical experiments to evaluate the performance of our optimized filter function and its adaptations compared to both standard and proposed filters from the literature. These experiments are performed on the conventional Shepp-Logan phantom, a modified Shepp-Logan phantom to examine the impact of underestimating noise, and the 2DeteCT dataset to test the reconstruction quality for real X-ray CT measurements.
The results reveal that our proposed filter functions $A_{L,D}^\ast$ and $\hat{A}_{L,D}^\varepsilon$ significantly surpass all dataset-independent filter functions in terms of MSE and SSIM and yield reconstruction qualities comparable to the dataset-dependent ERM-FBP filter. Moreover, our proposed filters are easy to implement, similar to standard filter functions, thus providing an out-of-the-box solution with theoretical justification.
Consequently, our filters bridge the gap in the literature between filters without closed-form representation {\textendash} requiring iterative schemes in advance or minimization problems for each measurement {\textendash} and filters requiring a training dataset, which is often challenging to obtain in practical applications.

We observe that our adaptation $\hat{A}_{L,D}^\varepsilon$ shows promising results in particular for the real 2DeteCT data, where $A_{L,D}^\ast$ is not applicable, and in scenarios with inaccurate noise level estimation, offering a good balance between noise reduction and the reconstruction of finer details. However, our chosen denoising technique for calculating $\hat{A}_{L,D}^\varepsilon$ might not be optimal for real-world data. Therefore, we propose exploring different denoising strategies in future works, where data-driven approaches seem promising, as they can learn the true noise distribution from the measured data.
Besides using denoised measured data, one could also employ deep learning to overcome the dependence of the true Radon data $\Radon f$. This approach also allows for including penalty terms in the loss function to incorporate prior knowledge.
Beyond these natural adaptations of our proposed filters, future work could also build on our theoretical investigations and modify the underlying minimization problem by, for instance, changing the loss function, adding additional penalty terms, or directly optimizing in the discrete setting.

\smallskip

In summary, our analytical derivations provide a solid foundation for future investigations, where we see particular potential in combining our theoretical results with neural networks to overcome the dependence of the filter on the target function and further improve the reconstruction quality.


\end{document}